\documentclass[11pt,reqno]{amsart}
\usepackage[margin=1in]{geometry}

\usepackage{hyperref}
\usepackage{amsthm}
\usepackage{amssymb}
\usepackage{amsxtra}
\usepackage{mathrsfs}
\usepackage{kopplagarias_clean} 

\usepackage[usenames,dvipsnames]{color}
\usepackage{bbm}
\usepackage{epsfig}
\usepackage{verbatim}
\usepackage{cleveref}
\usepackage{tikz-cd}
\usepackage{cite}
\usepackage{mleftright}
\usepackage{hhline}
\usepackage{autonum} 

\numberwithin{equation}{section}
\numberwithin{table}{section}
\numberwithin{figure}{section}

\theoremstyle{plain}
\newtheorem{theorem}{Theorem}[section]
\newtheorem{thm}[theorem]{Theorem}
\newtheorem{prop}[theorem]{Proposition}
\newtheorem{cor}[theorem]{Corollary}
\newtheorem{lem}[theorem]{Lemma}

\newtheorem{conj}[theorem]{Conjecture}

\newtheorem{empirp}[theorem]{Empirical Proposition}

\theoremstyle{definition}
\newtheorem{defn}[theorem]{Definition}

\newtheorem{egz}[theorem]{Example}

\theoremstyle{remark}
\newtheorem{rmk}[theorem]{Remark}

\AddToHook{env/prop/begin}{\crefalias{theorem}{prop}}
\AddToHook{env/cor/begin}{\crefalias{theorem}{cor}}
\AddToHook{env/lem/begin}{\crefalias{theorem}{lem}}
\AddToHook{env/que/begin}{\crefalias{theorem}{que}}
\AddToHook{env/conj/begin}{\crefalias{theorem}{conj}}
\AddToHook{env/assu/begin}{\crefalias{theorem}{assu}}
\AddToHook{env/empirp/begin}{\crefalias{theorem}{empirp}}
\AddToHook{env/defn/begin}{\crefalias{theorem}{defn}}
\AddToHook{env/nota/begin}{\crefalias{theorem}{nota}}
\AddToHook{env/cond/begin}{\crefalias{theorem}{cond}}
\AddToHook{env/egz/begin}{\crefalias{theorem}{egz}}
\AddToHook{env/intro/begin}{\crefalias{theorem}{intro}}
\AddToHook{env/rmk/begin}{\crefalias{theorem}{rmk}}

\begin{document}

\include{header}

\title[{SIC-POVM}s and real quadratic orders]{{SIC-POVM}s and orders of real quadratic fields}
\author{Gene S. Kopp}
\address{Department of Mathematics, Louisiana State University, Baton Rouge, LA, USA}
\email{kopp@math.lsu.edu}
\author{Jeffrey C. Lagarias}
\address{Department of Mathematics, University of Michigan, Ann Arbor, MI, USA}
\email{lagarias@umich.edu}


\date{May 13, 2026 (originally submitted December 27, 2024)}


\begin{abstract}
This paper concerns SIC-POVMs and their relationship to class field theory. SIC-POVMs are generalized quantum measurements (POVMs) described by $d^2$ equiangular complex lines through the origin in $\C^d$. Weyl--Heisenberg covariant SICs are those SIC-POVMs described by the orbit a single vector under a finite Weyl--Heisenberg group $\WH(d)$.

We relate known data on the structure and classification of Weyl--Heisenberg covariant SICs in low dimensions to arithmetic data attached to certain orders of real quadratic fields. For $4 \le d \leq 90$, we show the number of known geometric equivalence classes of Weyl--Heisenberg covariant SICs in dimension $d$ equals the cardinality of the ideal class monoid of (not necessarily invertible) ideal classes in the real quadratic order $\OO_{\Delta_d}$ of discriminant $\Delta_d=(d+1)(d-3)$; we conjecture the equality extends to all $d \ge 4$. We prove that this conjecture implies the existence of more than one geometric equivalence class of Weyl--Heisenberg covariant SICs for every $d>22$. We conjecture that Galois multiplets of SICs may be put in one-to-one correspondence with the over-orders $\OO'$ of $\OO_{\Delta_d}$ in such a way that the number of geometric classes in the multiplet equals the ring class number of $\OO'$. We test that conjecture against known data on exact SICs in low dimensions.

We refine the ``class field hypothesis'' of Appleby, Flammia, McConnell, and Yard \cite{appleby1} to predict the exact class field over $\Q(\sqrt{\Delta_d})$ generated by the ratios of vector entries for the equiangular lines defining a Weyl--Heisenberg SIC. The refined conjectures use a recently developed class field theory for orders of number fields \cite{kopplagarias}. The refined class fields assigned to over-orders $\OO'$ have a natural partial order under inclusion; the inclusions of these class fields fail to be strict in some cases. We characterize such cases and give a table of them for $d < 500$.
\end{abstract}

\maketitle

\section{Introduction}\label{sec:introduction}

SIC-POVMs are generalized quantum measurements (POVMs) described by $d^2$ equiangular complex lines through the origin in $\C^d$. Sets of $d^2$ equiangular complex lines in $\C^d$ are extremal objects in combinatorial design theory; they give {\em equiangular tight frames} \cite[Ch.~12--14]{waldron18}. The name \textit{SIC-POVM} (abbreviated `SIC') was introduced in a 2004 paper of Renes, Blume-Kohout, Scott, and Caves \cite{rbsc}, whose study was motivated by quantum state tomography, the problem of reconstructing a quantum state from a series of measurements performed on identical states. Scott showed that SICs are optimal measurement ensembles for quantum tomography \cite{scotttight}; see also \cite{grassltomography}. SICs have found further applications to quantum information processing \cite{detection,power,keydist,magic}, quantum foundations \cite{qbism}, compressed sensing for radar \cite{radar}, and classical phase retrieval \cite{phaseretrieval}. At present SICs have been constructed for finitely many dimensions $d$,
and conjecturally they exist for all $d$.

Twenty years of work on SICs by physicists since 2004 has uncovered strong connections of their structure with algebraic number theory (as described below and in \Cref{sec:12}). In particular, conjectures of Appleby, Flammia, McConnell, and Yard \cite{appleby2, appleby1} postulate the existence and number-theoretic structure of a certain class of SICs (Weyl--Heisenberg SICs, defined in \Cref{def:WH-SIC}), for all $d \ge 4$.

The objectives of this paper are:
\begin{enumerate} 
\item[(1)]
To extend and refine the conjectures on Weyl--Heisenberg SICs of \cite{appleby1}.
\item[(2)]
To test the extended conjectures against empirical data. 
\item[(3)]
To draw consequences of the extended conjectures counting and classifying geometric equivalence classes of Weyl--Heisenberg SICs. 
\end{enumerate} 

The conjectures of \cite{appleby1} relate Weyl--Heisenberg SICs in $\C^d$ for $d \geq 4$ to class field theory over real quadratic fields. The ratios of entries of each of the $d^2$ vectors in the currently known SICs in low dimensions were found to lie in a number field that is an abelian extension of the real quadratic field $K = \Q(\sqrt{\Delta_d})$ with $\Delta_d = (d+1)(d-3)$; see \Cref{sec:12}. Class field theory classifies the abelian extensions of a given number field $K$ in terms of ray class groups, which encode the arithmetic of ideals of the ring $\OO_K$ of algebraic integers of $K$.

The number fields associated to individual SICs in dimension $d$ may vary with the SIC. Appleby, Flammia, McConnell, and Yard \cite{appleby1} conjecture that the minimal number field that occurs for dimension $d$ SICs is a specific ray class field of $\Q(\sqrt{\Delta_d})$; see \Cref{conj:AFMY20}. The SICs with this associated number field comprise one or more \textit{Galois multiplets} (see \Cref{defn:gal-multiplet}), which are themselves finite sets of \textit{geometric equivalence classes} of Weyl--Heisenberg SICs (see \Cref{def:29}). The union of the classes in a multiplet form a single orbit of a Galois group action defined in \Cref{subsec:G-WHSIC}.

The extended conjectures made in this paper account for the remaining Galois multiplets of Weyl--Heisenberg SICs, those having larger associated number fields. The initial refinement asserts that, for each $d \ge 4$, the set of Galois multiplets can be put in one-to-one correspondence with a certain subset of the orders $\OO'$ of the real quadratic number field $K_d = \Q(\sqrt{\Delta_d})$---specifically, those orders $\OO'$ containing the order $\OO = \Z[\varepsilon_d]$ with $\varepsilon_d = \frac{1}{2}(d-1 + \sqrt{\Delta_d})$ and contained in the maximal order $\OO_{K_d}$ of all algebraic integers in $K_d$. Here an \textit{order} $\OO$ of $K_d$ is a subring of $K_d$ which contains $1$ and is of finite index in $\OO_K$. 

More precisely, the extended conjectures assert that a one-to-one correspondence exists in which the size of the multiplet (number of geometric equivalence classes) and the number field associated with the multiplet are specified by data defined purely in terms of $\OO'$. Namely:
\begin{enumerate}
\item
The number of geometric equivalence classes of SICs in the Galois multiplet is the ring class number of the order $\OO'$ (\Cref{conj:sic-to-order0}).
\item
The number field associated to every geometric equivalence class in the Galois multiplet is a specified ray class field of the order $\OO'$ (\Cref{conj:14}).
\end{enumerate}
 
The notion of ray class fields associated to orders is an extension of the usual class field theory (which is the special case of the maximal order), variants of which were developed (independently) by Campagna and Pengo \cite{CampagnaP22} and the authors \cite{kopplagarias}. Ray class fields associated to orders are discussed in \Cref{sec:81}. Inclusions of orders induce reverse inclusion of the corresponding ray class fields, so the maximal order corresponds to the minimal ray class field studied in \cite{appleby1}.

The remainder of the introduction describes SICs, their history, and conjectures about them; it also summarizes main results of the paper. 
Precise definitions and proofs are given later in the paper.

\subsection{SICs}\label{sec:11}

A set $S$ of complex lines through the origin in $\C^d$, written $S=\{\C \bv_1, \ldots, \C \bv_n \} \subset \Pj^{d-1}(\C)$ with $\bv_j$ unit vectors, is \textit{equiangular} if all the pairwise Hermitian inner products $\langle \bv_i, \bv_j \rangle$ for $i \neq j$ have the same absolute value $\abs{\langle \bv_i, \bv_j \rangle} = \alpha \ge 0$. In 1975, Delsarte, Goethals, and Seidel \cite{delsarte} proved an upper bound of $n \le d^2$ on the size of any set of complex equiangular lines in $\C^d$. They showed that if the bound $n=d^2$ is attained, then necessarily $\alpha = \frac{1}{\sqrt{d+1}}$. 

A \textit{symmetric informationally complete positive operator-valued measure} or \textit{SIC-POVM} (henceforth \textit{SIC}) is a set of normalized projections onto a (maximal) set of $d^2$ complex equiangular lines in $\C^d$. For $d^2$ equiangular lines $\{\C \bv_1, \ldots, \C \bv_{d^2}\}$ represented by unit column vectors $\bv_j \in \C^d$, the associated SIC is the set of normalized rank one Hermitian projection operators $\PPi(S) := \left\{\frac{1}{d} \Pi_1, \ldots, \frac{1}{d}\Pi_{d^2}\right\}$ where $\Pi_j = \bv_j \bv_j^\ct$ are rank one projection operators. (Here $\bv$ are column vectors, $\bv^\ct$ is the conjugate transpose of $\bv$, a row vector, and the Hermitian inner product used is $\langle \bv_i, \bv_j \rangle = \bv_i^{\ct} \bv_j$. Unit vectors have $\bv_i^{\ct} \bv_i=1$. Here $\Pi_j= \vert \bv_j \rangle \langle \bv_j \vert$ in Dirac notation.) The quantum information theory definition of a SIC-POVM uses the operators $E_i = \frac{1}{d}\Pi_i$ scaled by $\frac{1}{d}$ so that $\sum_i E_i = I$ (the identity matrix); see Renes, Blume-Kohout, Scott, and Caves ~\cite{rbsc} and Appleby \cite[eq.~(1)]{App05}. The $E_i$ give also an equiangular tight frame with frame constant $A=1$; see \cite[Sec.~14.3]{waldron18}.

The SIC condition is equivalent to the following $d^4$ identities on the traces of products of the projections:
\begin{align}\label{eqn:SIC-projection}
\Tr(\Pi_i \Pi_j) &=
\begin{cases}
1 & \mbox{if } i=j, \\
\frac{1}{d+1} & \mbox{if } i \ne j, 
\end{cases}
&&
\mbox{for $1 \le i, j \le d^2$.}
\end{align}

We will use the term ``SIC'' interchangeably to mean either a ``set of $d^2$ (scaled) Hermitian projection matrices $\PPi(S)$ forming a SIC-POVM'' or a ``set of $d^2$ complex equiangular lines $S$ in $\C^d$.'' When more specificity is needed, we term the latter a {\em line-SIC}. 

Zauner's seminal 1999 PhD thesis \cite{zauner,zaunertrans} on quantum designs established SICs as a research topic, inspiring much subsequent work. A few sporadic examples (such as the Hoggar lines in dimension $8$ \cite{hoggar1,hoggar2}) were known prior to Zauner's work. Zauner conjectured, based on his numerical investigations, that SICs exist in every dimension $d$, and moreover, that a SIC in $\C^d$ could be found as the orbit of a single vector $\bv$ under the Weyl--Heisenberg group \cite{zauner,zaunertrans}; the latter objects are called \textit{Weyl--Heisenberg covariant SICs} (or just \textit{Weyl--Heisenberg SICs}). A generating vector $\bv \in \C^d$ of such an orbit is called a \textit{fiducial vector}.

Weyl--Heisenberg SICs are defined and discussed in \Cref{sec:2} (\Cref{def:WH-SIC}). They are the main topic of this paper.

Over the past 25 years, a large set of SICs has been computed in particular dimensions---sometimes approximately and sometimes exactly---resulting in extensive datasets \cite{scott1,scott2,flammiaweb}. SICs have been rigorously constructed in a finite set of dimensions, including Weyl--Heisenberg SICs in all dimensions $1 \le d \le 53$ and the Hoggar lines for $d=8$. Beyond these exact constructions, numerical computations in dimensions including every $d \leq 90$ suggest there are only finitely many Weyl--Heisenberg SICs in each dimension $d \ne 3$. It is known that $d=3$ has infinitely many equivalence classes of Weyl--Heisenberg SICs, and we will not treat $d=3$ here. An extensive treatment of the history of SICs, prior to the developments in \Cref{sec:12}, is given by Fuchs, Hoang, and Stacy \cite{play}.

Weyl--Heisenberg SICs can be related to each other. There is a geometric action moving Weyl--Heisenberg SICs around by the normalizer of the Weyl--Heisenberg group inside the extended unitary group. Weyl--Heisenberg SICs $S$ may be grouped using this action into geometric equivalence classes of SICs $S$, denoted $[S]$, defined in \Cref{subsec:EUU}. This geometric action makes sense for ``numerical SICs,'' and the empirical count for ``numerical'' geometric classes $[S]$ for $d \le 90$ is shown in \Cref{table:classgroup-1CC} and \Cref{table:classgroup-2BB}.

In addition all known exact Weyl--Heisenberg SICs in dimensions $d \ne 3$ are constructed using algebraic numbers. It was discovered that Galois automorphisms sometimes take SICs to other SICs and sometimes do not; see \Cref{subsec:G-SIC}.

The process of analyzing these data has led to formulation of a series of interlocking hypotheses. The two most basic conjectures, which are both unsolved, are:
\begin{enumerate}
\item[(1)]
Weyl--Heisenberg SICs exist in all dimensions (conjectured by Zauner \cite{zauner});
\item[(2)] The set of Weyl--Heisenberg line-SICs in every fixed dimension $d \ne 3$ is a finite set.
\end{enumerate}
 
The finiteness assertion has been a ``folklore'' conjecture since about 2004. It was observed empirically under heuristics up to dimension $d=7$ by Renes, Blume-Kohout, Scott, and Caves \cite[Tab.~I and Sec.~V]{rbsc}, up to dimension $d=50$ by Scott and Grassl (who state they are confident their list is complete for $d \leq 50$ in \cite[Sec.~4]{scott1}), and up to $d=90$ by Fuchs, Hoang, and Stacy \cite[Sec.~6]{play} using data of Scott \cite{scott2}.  It was explicitly formulated as a conjecture in the first author's PhD thesis \cite[Conj.~V.12]{koppthesis}.

A third conjecture implicit in the literature (e.g., in \cite{appleby2}) is an ``algebraicity conjecture.''
\begin{enumerate}
\item[(3)] For $d \neq 3$, every Weyl--Heisenberg SIC has a fiducial vector whose entries are algebraic numbers. (Equivalently, For $d \neq 3$, every fiducial vector of a Weyl--Heisenberg SIC can be scaled so that all its entries are algebraic numbers.)
\end{enumerate}
In \Cref{sec:C1}, we prove that conjectures (2) and (3) are equivalent in every fixed dimension $d \ne 3$; see \Cref{prop:finalgequiv}. For further results on these conjectures, see \Cref{subsec:WHSIC}.

\subsection{Connection of SICs to algebraic number theory and class fields}\label{sec:12}

In all currently known examples of exact SICs for $d \ne 3$ the coordinates of fiducial vectors of Weyl--Heisenberg SICs can be chosen to have all entries algebraic numbers. 

Recent work has connected SICs to explicit class field theory over real quadratic number fields. Appleby, Yadsan-Appleby, and Zauner \cite{AYZ13} discovered in 2013 that in all (then) known examples (for $d \neq 3$), the algebraic number field $\E$ generated by the entries of the projections $\PPi(S)$ contains the quadratic field $K_d = \Q(\sqrt{(d+1)(d-3)})$, and Galois automorphisms often map SICs to new SICs. Moreover, they observed that (in all those cases) $\E$ is an \textit{abelian} extension of $K_d$.

In 2016, Appleby, Flammia, McConnell, and Yard \cite{appleby1,appleby2} made explicit predictions relating the number fields generated by Weyl--Heisenberg SICs to ray class fields. Based on extensive data given in \cite[Prop.~1]{appleby1}, they formulated the following conjecture \cite[Conj.~2]{appleby1}.

\begin{conj}[Appleby, Flammia, McConnell and Yard]\label{conj:AFMY20}
Let $d \geq 4$, $\Delta_d = (d+1)(d-3)$, and $K_d = \Q(\sqrt{\Delta_d})$.
\begin{enumerate}
\item At least one Weyl--Heisenberg line-SIC $S$ exists in dimension $d$ having SIC field $\fieldvec{S}$ equal to $H_{d'\infty_1\infty_2}$, the ray class field of $K_d$ having ray class modulus $d'\infty_1\infty_2$, where $d'=d$ if $d$ is odd, and $d'=2d$ if $d$ is even.
\item Every Weyl--Heisenberg line-SIC $S$ in dimension $d$ has SIC field $\fieldvec{S}$ that is a finite Galois extension of $\Q$ containing $H_{d'\infty_1\infty_2}$.
\end{enumerate}
\end{conj}

The technical definition of the SIC field $\fieldvec{S}$ appearing in \Cref{conj:AFMY20} is given in \Cref{subsec:NFS}. Ray class fields are discussed in \Cref{sec:81}.

Appleby, Flammia, McConnell, and Yard verified assertion (1) of \Cref{conj:AFMY20} for
\begin{equation}
d \in \{4, 5, 6, 7, 8, 9, 10, 11, 12, 13, 14, 15, 16, 17, 18, 19, 20, 21, 24, 28, 30, 35, 39, 48\}.
\end{equation}
They verified assertion (2) of the conjecture for all known Weyl--Heisenberg SICs in these dimensions. They also verified that the fields $\fieldvec{S}$ are all abelian over $K_d$ in the known cases. (The list of SICs used is believed to be complete in these dimensions.)

Subsequently, in 2018, the first author \cite{koppsic} proposed an explicit, conjectural construction of a SIC in prime dimensions $d \con 2 \Mod{3}$ for $d \ge 5$ in terms of Stark units and verified it in the cases $d \in \{5, 11, 17, 23\}$. (He suggested a connection to Stark units in every dimension). Recently Appleby, Bengtsson, Grassl, Harrison, and McConnell \cite{abghm} gave a (different, related) conjectural construction in prime dimensions of the form $d=p=n^2+3$, and Bengtsson, Grassl, and McConnell \cite{bgm} extended the latter to dimensions $d=4p=n^2+3$ with $p$ prime.

The origin of the discriminant $\Delta_d = (d+1)(d-3)$ appearing in the conjecture is not completely understood.

\subsection{Main results}\label{sec:13}

We extend and refine \Cref{conj:AFMY20} into two conjectures that count the number of geometric equivalence classes of such SICs in each Galois multiplet and that specify the number fields associated to all the Weyl--Heisenberg SICs in each dimension $d \ge 4$.

We recall necessary terminology. For a Weyl--Heisenberg line-SIC $S$,
\begin{enumerate}
\item[(1)] the \textit{geometric equivalence class} $[S]$ (also called the $\PEC(d)$-orbit) is the orbit of $S$ under the action of the \textit{projective extended Clifford group} $\PEC(d)$ (as given by \Cref{def:29}), and
\item[(2)] the \textit{Galois multiplet} $[[S]]$ is the set of all geometric equivalence classes $[\sigma(S)]$ that are reachable from $[S]$ under the action of Galois automorphisms $\sigma$ that map SICs to SICs (if $S$ is algebraic, as given by \Cref{defn:gal-multiplet}).
\end{enumerate}

To frame the conjectures, we must specify number fields attached to algebraic SICs. In \Cref{sec:fields,sec:algebraic}, we define several fields attached to a line-SIC $S = \{\C\bv_k : 1 \leq k \leq d^2\}$. These fields are finitely generated extensions of the rational numbers $\Q$. 
\begin{enumerate}
\item[(1)]
The {\em ratio SIC field} $\fieldvec{S}$ (called the ``SIC field'' in \cite{appleby1}) is the field generated by all the ratios $\frac{v_{ki}}{v_{kj}}$ of distinct entries ($i \ne j$, $v_{kj} \ne 0$) of the vectors $\bv_k$ with $\C \bv_k \in S$.
\item[(2)]
The {\em projection SIC field} $\fieldproj{S}$ is the field generated by the matrix entries of the Hermitian projection operators $\Pi_{i} = \bv_i \bv_i^\ct$ of the $d^2$ unit vectors $\{\bv_i: 1 \le i \le d^2\}$ defining $S$.
\item[(3)] The {\em extended projection SIC field} $\fieldeproj{S}=\fieldproj{S}\!(\xi_d)$ is obtained by adjoining the root of unity $\xi_d = -e^{\frac{\pi i}{d}}$. 
\item[(4)]
For algebraic SICs (ones with $\fieldvec{S}$ a finite extension of $\Q$), the {\em Galois projection SIC field} $\Nfieldproj{S}$ is the normal closure of $\fieldeproj{S}$. 
\end{enumerate}
We establish the following field inclusions, with the last being valid only for algebraic SICs:
\begin{equation}\label{eqn:field-inclusions}
\fieldvec{S} \subseteq \fieldproj{S} \subseteq \fieldeproj{S}\subseteq \Nfieldproj{S}.
\end{equation}
The field $\Nfieldproj{S}$ is an invariant of the Galois multiplet $[[S]]$ of an algebraic SIC $S$. In \Cref{lem:conj11-consequence}, we show that \Cref{conj:AFMY20} implies the equalities 
\begin{equation}\label{eqn:field-equalities}
\fieldvec{S} = \fieldproj{S} = \fieldeproj{S}= \Nfieldproj{S},
\end{equation}
whence $\fieldvec{S}$ is a finite Galois extension of $\Q$. 

The conjectures for Weyl--Heisenberg SICs in dimension $d \ge 4$ use the family of intermediate orders $\OO'$ of $K_d= \Q(\sqrt{\Delta_d})$ that lie between $\OO_{\Delta_d}= \Z[\varepsilon_d]$ and $\OO_{K_d}$. Given $\Delta_d$, we label each intermediate order $\OO'$ by its \textit{conductor} $f'$, a positive integer such that $\disc(\OO') = (f')^2\Delta_0$ where $\Delta_0 = \disc(\OO_{K_d})$. The conductors appearing are the positive divisors $f'$ of the conductor $f$ of $\OO_{\Delta_d}$. Quadratic orders are uniquely specified by their discriminants, so we may write $\OO'= \OO_{(f')^2 \Delta_0}$. (Further definitions and properties of orders, discriminants, and conductors are reviewed in \Cref{subsec:51}.)

\begin{conj}[Order-to-Multiplet Conjecture]\label{conj:sic-to-order0}
Fix a positive integer $d \ge 4$, let $\Delta = \Delta_d = (d+1)(d-3)$, and write $\Delta = f^2\Delta_0$ for $\Delta_0 = \disc(\OO_{K_d})$ and some positive integer $f$. Then there is a bijection $\MM$ between intermediate orders $\OO'=\OO_{(f')^2 \Delta_0}$, labeled by their conductors $f'$, 
and Galois multiplets, 
\begin{align}
\{\mbox{$f'$ a positive divisor of $f$}\} \xrightarrow{\MM} \{[[S]] : S \mbox{ a Weyl--Heisenberg line-SIC in } \C^d\},
\end{align}
having the following properties.
\begin{itemize}
\item[(1)] The number of geometric equivalence classes $[S]$ in the Galois multiplet $\MM(f')=[[S]]$ is the ring class number $\hcn{(f')^2 \Delta_0} := \abs{\Cl(\OO_{(f')^2\Delta_0})}$ of the the order $\OO'$ of discriminant $(f')^2\Delta_0$.
\item[(2)] If $f_1\div f_2$, $f_2 \div f$, $[S_1] \in \MM(f_1)$, and $[S_2] \in \MM(f_2)$, then $\fieldvec{S_1} \subseteq \fieldvec{S_2}$.
\end{itemize}
\end{conj}

The special case $f'=1$ corresponds to \Cref{conj:AFMY20}, where the order $\OO'$ is the maximal order. Appleby, Flammia, McConnell, and Yard \cite{appleby2} call the associated Galois multiplet $[[S]]$ the \textit{minimal multiplet}, as it has the smallest associated field $\fieldvec{S}$ under inclusion. The field inclusions asserted in \Cref{conj:AFMY20}(2) imply the truth of \Cref{conj:sic-to-order0}(2) in the special case $f_1=1$.

The following conjecture is an extension of both \Cref{conj:AFMY20} and \Cref{conj:sic-to-order0}, which uses the notion of ray class fields of an order $\OO'$ introduced by the authors in \cite{kopplagarias} and reviewed in \Cref{sec:81}.

\begin{conj}[Ray Class Fields of Orders Conjecture]\label{conj:14} 
Let $d \ge 4$. Then the map $\MM$ in \Cref{conj:sic-to-order0} may be chosen so that, for any Weyl--Heisenberg line-SIC $S$ in dimension $d$ with associated order $\OO'$ having discriminant $(f')^2\Delta_0$ (that is, with $\MM(f') = [[S]]$), one has 
\begin{equation}
\fieldvec{S} = \fieldproj{S} = H_{d'\OO', \{\infty_1, \infty_2\}}^{\OO'},
\end{equation}
where $H_{d'\OO', \{\infty_1,\infty_2\}}^{\OO'}$ is the ray class field of the order $\OO'$ having level datum $(\OO'; d'\OO', \{\infty_1,\infty_2\})$, $d'= 2d$ if $d$ is even, and $d'=d$ if $d$ is odd. 
\end{conj}

The paper checks predictions of these conjectures against empirical data for SICs obtained in the datasets \cite{scott1, scott2, flammiaweb, play}. The empirical data include both exact Weyl--Heisenberg SICs given by algebraic fiducial vectors and numerical SICs given as rational vectors approximately satisfying the fiducial vector SIC relations (to many decimal places). The notion of \textit{geometric equivalence class} makes sense for numerical SICs, because they have well-defined $\PEC(d)$-orbits; however, the notion of \textit{Galois multiplet} is not well-defined for them. In consequence we cannot test \Cref{conj:sic-to-order0} using numerical SICs, since its statement involves Galois multiplets. Instead, we extract a numerical prediction implied by \Cref{conj:sic-to-order0} that is testable.

This numerical prediction says that the total number of geometric equivalence classes of Weyl--Heisenberg SICs in dimension $d \ge 4$ is $\abs{\Clm(\OO_{\Delta_d})}$, the cardinality of the {\em class monoid of the order $\OO_{\Delta_d}$}, as defined in \Cref{subsec:51}. We formulate this prediction as \Cref{conj:count}, and we test it in \Cref{prop:emp1} using both exact and numerical SICs. These tests are summarized by \Cref{table:classgroup-1CC} and \Cref{table:classgroup-2BB} for all dimensions $4 \le d \le 90$. \Cref{conj:count} gives agreement in all these dimensions, assuming that the list of exact plus numerical SICs is complete (but see \Cref{rmk:empirical}). (The number $\abs{\Clm(\OO_{\Delta_d})}$ is equal to a sum of the ring class numbers $h_{(f')^2 \Delta_0}$ appearing in \Cref{conj:sic-to-order0}; see \Cref{thm:equiv}.)

The existence of the bijection $\MM$ asserted in \Cref{conj:sic-to-order0} implies a numerical prediction concerning the number of Galois multiplets, predicting that it equals the number of quadratic orders between $\OO_{\Delta_d}$ and $\OO_{K_d}$, which is the number of over-orders of $\OO_{\Delta_d}$ in $K= \Q(\sqrt{\Delta_d})$. This numerical prediction is formulated as \Cref{conj:supercount} and tested in \Cref{prop:emp2} against certain dimensions $d$ in which all known SICs are exact SICs. Supporting data is given in \Cref{table:classgroup-1CC}.

In \Cref{sec:field}, we provide further tests of \Cref{conj:sic-to-order0}(2) and \Cref{conj:14}. We first summarize results on ray class fields of orders required to understand the conjectures. In \Cref{prop:66}, we numerically test the equality of degrees over $\Q$ of the ray class fields of orders in \Cref{conj:14} to those of the fields $\fieldvec{S}$. Tests are performed for exact SICs in dimensions $4 \le d \le 15$ and $d=35$, and a supporting table of fields is given in \Cref{table:classfield-1a}.

We display for $d=35$ a (unique) bijective map $\MM$ of orders to Galois multiplets for the known exact SICs that in \Cref{example:59} fulfills condition (1) of \Cref{conj:sic-to-order0}, and, in \Cref{example:67}, fulfills condition (2) of \Cref{conj:sic-to-order0}.

In \Cref{sec:class-monoids} we derive, via the extended conjectures, bounds on when there is exactly one Galois multiplet. That is, we determine all $d$ for which the class monoid $\abs{\Clm(\OO_{\Delta_d}}$ has one element. They all have $d \le 22$. We state a result from \cite{KL-UGO2} on the asymptotic growth of the size of the class monoid as $d \to \infty$.

In \Cref{sec:degeneration} we state results proved elsewhere \cite{KL-UGO2}, showing that \Cref{conj:14} strongly restricts the map $\MM$, often, but not always, making it unique. We state in \Cref{thm:99} that the set of $d$ for which there exist two distinct orders $\OO:=\OO', \OO''$ with identical $H_{d'\OO, \{\infty_1, \infty_2\}}^{\OO}$ in \Cref{conj:14} is the set of $d$ for which the squarefree part of $(d+1)(d-3)$ is $1 \Mod{8}$. Then \Cref{prop:density-18a} computes the asymptotic density of this set to be $\frac{1}{48}$. Using these results we obtain \Cref{table:degeneration}, which lists all cases where equality of two such ray class fields occurs for $d \le 500$. These include all cases of observed equality of SIC fields $d \in \{47,67,259\}$ made by Grassl \cite{grasslpersonal}.

\begin{rmk}\label{rmk:empirical}
Neither the empirical work on numerical Weyl--Heisenberg SICs nor the constructions of exact Weyl--Heisenberg SICs prove that the lists found are exhaustive. Our results labeled ``Empirical Proposition'' assert only that the current lists of (numerical or exact) SICs give equality with the conjectured numerical predictions in the dimensions $d$ covered.
\end{rmk}

\subsection{Relations to prior work}\label{subsec:14}

This paper formulates precise mathematical definitions of various number fields attached to SICs in Sections \ref{sec:fields} and \ref{sec:algebraic} in a more systematic way than earlier treatments. We prove unconditional inclusions among the various fields, at the current state of knowledge. (Conjecturally, many of these fields should coincide; see \eqref{eqn:field-equalities}.) 

The definitions of geometric equivalence classes and Galois multiplets of Weyl--Heisenberg SICs build on previous definitions given by Appleby, Flammia, McConnell, and Yard ~\cite{appleby2} and by Waldron \cite[Sec.~14.22]{waldron18}. The present definition of a Galois multiplet defines a potentially smaller set than some prior definitions of that term; see \Cref{rmk:413}.

The book of Waldron \cite{waldron18} defines Weyl--Heisenberg SICs using the (mathematics) inner product $\langle \bv, \bw\rangle = \bv^{T} \overline{\bw}$ (as opposed to the convention $\langle \bv, \bw\rangle = \ol{\bv}^{T} \bw$ adopted here). This definition gives the same set of SICs in dimension $d$ as this paper, because the Weyl--Heisenberg group (of matrices) is invariant under complex conjugation, and because if $\bv$ is a fiducial vector of a Weyl--Heisenberg SIC, then its complex conjugate vector $\overline{\bv}$ is the fiducial vector of a (possibly different) Weyl--Heisenberg SIC. These two complex conjugate SICs are in the same $\PEC(d)$-orbit.

We use data of Scott \cite{scott2} on (numerical) $\PEC(d)$-orbits of SICs up to $d=90$. Scott's data are also reported in the book of Waldron \cite[Tab.~14.4, pp.~386--388]{waldron18}. The elements of these orbits have stabilizers (internal symmetries) that Scott tabulates, which we do not treat in this paper. 

\subsection{Index of notation}\label{subsec:15}
\begin{itemize}
\item
Boldface roman upper case letters represent continuous groups, e.g., $\EU(d)$.
\item
The boldface roman lower case letters $\bv,\bw$ represent column vectors in $\C^d$, for example, $\bv = (v_{0}, v_{1}, \cdots v_{d-1})^T$. On the other hand, the boldface roman lower case letters $\mathbf{p},\mathbf{q}$ represent ordered pairs in $(\Z/d\Z)^2$.
\item 
Upper case roman letters represent $d \times d$ complex matrices (e.g., $X$, $Z$); they also appear in names of finite matrix groups (e.g., $\WH(d)$) and sets (e.g., $\WHSIC{d}$).
\item
For a complex matrix $M$ or vector $\bv$, the conjugate transpose is denoted $M^\ct := \ol{M}^\top$ or $\bv^\ct = \ol{\bv}^\top$. For a vector, the norm $\norm{\bv}^2 := \bv^{\ct}\bv = \sum_{i=0}^{d-1} v_i^2$.
\item
Matrices and vectors have rows and columns numbered from $0$ to $d-1$. We follow the text of Waldron \cite{waldron18}, which identifies indices with $\Z/d\Z$.
\item
Greek letters generally represent algebraic numbers (e.g., $\xi_d = -e^{\frac{\pi i}{d}}$ and $\zeta_d= e^{\frac{2 \pi i}{d}}$). 
\item
However, the letter $\sigma$ is used for Galois automorphisms, with $\sigma_c$ denoting complex conjugation. Galois group actions are treated either as acting on scalars or as acting entrywise on vectors or matrices. The notation $\sigma_0$ is also used for the number-of-divisors function.
\item
The notation $\Qbarline{\bv}$ is used for the $\ol{\Q}$-line associated to a vector (\Cref{defn:alg-vectors}).
\item
The notation $S$ is used for a ($d$-dimensional) line-SIC; it is a set of $d^2$ complex lines. 
\item
The notation $S(\bw)$ is used for a ($d$-dimensional) Weyl--Heisenberg line-SIC generated by the fiducial vector $\bw \in \C^d$. It is determined by the complex line $\C\bw$.
\item
The notation $[S]$ is used for a geometric equivalence class (or $\PEC(d)$-orbit) of SICs. 
\item
The notation $[[S]]$ is used for a Galois multiplet of Weyl--Heisenberg SICs (\Cref{defn:gal-multiplet}).
\item
The notation $d'$ means $d'= d$ if $d$ is odd and $d'=2d$ if $d$ is even. The number $\xi_d = -e^{\frac{\pi i}{d}}$ is a $d'$-th root of unity. 
\item 
We write $\Delta$ for a general quadratic discriminant with associated fundamental discriminant $\Delta_0$. Here $\Delta= f^2\Delta_0$ for some $f \ge 1$, called the conductor of $\Delta$. When $\Delta=\Delta_d$, we write $\Delta_0 = (\Delta_d)_0$ when needed for clarity, otherwise $\Delta_0$. (The quantity $\Delta_0$ will always mean a fundamental discriminant, never ``$\Delta_d$ for $d=0$.'')
\item
We write $\Delta_d = (d+1)(d-3)$ for the (often non-fundamental) discriminant associated to dimension $d$, $K_d = \Q(\sqrt{\Delta_d})$ for the associated field, and $\e_d = \foh(d-1+\sqrt{\Delta_d})$ for the associated unit.
\item 
We write $\OO_\Delta$ for the quadratic order of discrimiant $\Delta$, that is, $\OO_\Delta = \Z\left[\frac{\Delta+\sqrt{\Delta}}{2}\right]$.
\item 
We write $\sqfreepart(r)$ for the squarefree part of a positive integer $r$; that is, if $r=st^2$ with $s$ squarefree, then $\sqfreepart(r) = s$.
\item 
The Weyl--Heisenberg group $\WH(d)$ and projective Weyl--Heisenberg group $\PWH(d)$ are defined in \Cref{def:WH-group} and \Cref{def:PWH-group}.
\item 
The Clifford group $\CG(d)$, projective Clifford group $\PC(d)$, extended Clifford group $\EC(d)$, and projective extended Clifford group $\PEC(d)$ are defined in \Cref{def:EC-PEC}.
\item 
The notation $\WHSIC{d}$ for the set of Weyl--Heisenberg covariant SICs is introduced in \Cref{def:WH-SIC}.
\item 
The fields $\fieldvec{S}$, $\fieldtrip{S}$, $\fieldinv{S}$, $\fieldproj{S}$, and $\fieldeproj{S}$ associated to a SIC $S$ are defined in \Cref{defn:fields}. Another associated field $\Nfieldproj{S}$ is defined in \Cref{def:Galois-proj-field}.
\item 
The class group $\Cl(\OO)$ of an order $\OO$ is defined in \Cref{defn:classgroup}, and the class monoid $\Clm(\OO)$ is defined in \Cref{defn:classmonoid}. 
\item 
The ring class number of an order $\OO$ of a quadratic field is the cardinality $\abs{\Cl(\OO)}$ of its ideal class group, sometimes written $h_{\Delta}$, where $\Delta$ is the discriminant of $\OO$; see \Cref{cor:54}.
\item The ray class group $\Cl_{\mm,\rS}(\OO)$ and associated ray class field $H_{\mm,\rS}^{\OO}$ are introduced in \Cref{sec:81}.
\end{itemize}

\section{Weyl--Heisenberg SICs}\label{sec:2}

In this section, we define and give background on Weyl--Heisenberg (covariant) SICs, which satisfy a particular group symmetry property.

\subsection{Weyl--Heisenberg SICs}\label{subsec:WHSIC}

Specifically, a Weyl--Heisenberg line-SIC will be defined as the orbit of a single line under the action of the Weyl--Heisenberg group, a finite subgroup of the group of $d \times d$ unitary matrices.

\begin{defn}\label{def:WH-group}
The \textit{Weyl--Heisenberg group} is $\H(d) = \{\xi_d^r X^{p}Z^{q} : p,q,r \in \Z\}$, where $\xi_d= -e^{\frac{\pi i}{d}}$, $\zeta_d = \xi_d^2 = e^{\frac{2\pi i}{d}}$, and
\begin{align}\label{eqn:heis21}
X &= \begin{pmatrix}
0 & 0 & \cdots & 0 & 1 \\
1 & 0 & \cdots & 0 & 0 \\
0 & 1 & \cdots & 0 & 0 \\
\vdots & \vdots & \ddots & \vdots & \vdots \\
0 & 0 & \cdots & 1 & 0
\end{pmatrix}, &
Z &= \begin{pmatrix}
1 & 0 & 0 & \cdots & 0 \\
0 & \zeta_d & 0 & \cdots & 0 \\
0 & 0 & \zeta_d^2 & \cdots & 0 \\
\vdots & \vdots & \vdots & \ddots & \vdots \\
0 & 0 & 0 & \cdots & \zeta_d^{d-1}
\end{pmatrix}.
\end{align}
The matrix $X$ is called the {\em cyclic shift matrix}, and $Z$ is called the {\em modulation matrix} \cite[Sec.~14.5]{waldron18}.
\end{defn} 

It is easily checked that $ZX = \zeta_d XZ$. It follows that $\H(d)$ is a group. Its order is $d^3$ if $d$ is odd and $2d^3$ if $d$ is even, noting that $\xi_d$ is a $d$-th root of unity for odd $d$ and a $2d$-th root of unity for even $d$. 

\begin{defn}\label{def:PWH-group}
The \textit{projective Weyl--Heisenberg group} $\PWH(d)$ is defined to be the quotient group $\WH(d)$ modulo its center, the scalar matrix subgroup $Z(\WH(d)) = \{\xi_d^{j} I :\, j \ge 0 \}$, where $I$ is the $d \times d$ identity matrix.
\end{defn}

The group $\PWH(d) := \H(d)/\ZZ(\H(d)) \isom \left(\Z/d\Z\right)^2$. It has a well-defined action on complex lines $\C\bw$.

\begin{defn}\label{def:WH-orbit}
For $\bw \in \C^d$, $\bw \ne 0$, its \textit{Weyl--Heisenberg orbit} is
\begin{equation}
\sO_{\WH}(\bw) := \{H \bw : H \in \WH(d)\}. 
\end{equation}

Its associated set of \textit{Weyl--Heisenberg complex lines}, denoted $\WHL(\bw)$, is 
\begin{equation}
\WHL(\bw) := \{\C \bw: \, \bw \in \sO_{\WH}(\bw)\}= \{\C H \bw: H \in \PWH(d)\}.
\end{equation} 
For any $\bw$, $\WHL(\bw)$ has cardinality at most $d^2$, and for generic $\bw$, exactly $d^2$.
\end{defn}

\begin{defn}\label{def:WH-SIC}
A \textit{Weyl--Heisenberg covariant SIC} (or simply a \textit{Weyl--Heisenberg SIC}) is any set of Weyl--Heisenberg complex lines $S =S(\bv) :=\WHL(\bv)$ for which $\PPi(S)$ is a $d$-dimensional SIC. We denote the set of all Weyl--Heisenberg SICs in dimension $d$ as $\WHSIC{d}$.
\end{defn}

We call a generating vector $\bv$ of a Weyl--Heisenberg SIC a {\em fiducial vector}. The associated SIC $S(\bw) = S(\bv)$ whenever $\bw= \lambda \bv$ for $\lambda \in \C^{\times}$. Each SIC has a choice of $d^2$ fiducial vectors, up to scalars. 

\begin{conj}[Strong Zauner Conjecture]
\label{conj:strongZ}
For each dimension $d \ge 1$, there exists at least one Weyl--Heisenberg SIC.
\end{conj} 

This conjecture is currently known by explicit constructions to hold for finitely many dimensions $d$, including $1 \le d \le 53$ and a extensive number of higher dimensions. 

\begin{conj}[Weyl--Heisenberg Boundedness Conjecture]
\label{conj:WHB}
For each dimension $d \ge 1$, except $d=3$, there are finitely many Weyl--Heisenberg SICs.
\end{conj} 

There are a few small dimensions where this conjecture has been proved rigorously, including $d \in \{1,2,4,5,6\}$. The cases $d = 1,2$ may be done by hand; the cases $d=4,5,6$ have been handled by computer using Gr\"{o}bner bases \cite[unpublished]{grasslpersonal2}. In higher dimensions, up to at least $d=50$, it is supported by non-rigorous evidence given by numerical optimization routines repeatedly returning the same SICs \cite{scott1,scott2}; see \Cref{table:classgroup-1CC}. In dimension $3$, there is a one-parameter family of inequivalent Weyl--Heisenberg SICs.

\subsection{Geometric equivalence classes of Weyl--Heisenberg SICs}\label{subsec:EUU}

There are two types of geometric operators acting on $\C^d$ that send SICs to SICs. Firstly, if $U$ is a unitary matrix in $\U(d) = \{U \in \GL(\C^d) : UU^\ct = 1\}$, and $S = \{\C \bv_1,\ldots,\C \bv_{d^2}\}$ is a SIC, then so is $US = \{\C U\bv_1,\ldots,\C U\bv_{d^2}\}$. (This is a \textit{unitary equivalence}.) Secondly, if $C_d$ is the complex conjugation operator acting pointwise on $\C^d$, so that $C_d \bv := \ol{\bv}$, then $C_d$ also preserves the SIC property, as does any ``antiunitary'' operator of the form $C_d U$. These operators may be collected together to form the extended unitary group.

\begin{defn}\label{EU-PEU-group}
The \textit{extended unitary group} is $\EU(d) := \U(d) \sqcup C_d \U(d)$. The \textit{projective extended unitary group} $\PEU(d) := \EU(d)/\{zI : z \in \C, \abs{z}=1\}$.
\end{defn}

All known SICs are group covariant, meaning that they are the orbit of a vector under the action of a finite subgroup of the unitary group.

\begin{defn}\label{def:group-covariant}
A SIC $S$ is \textit{group covariant} for a subgroup $G$ of $\U(d)$ if there exists a vector $\bv \in \C^d$, called a \textit{fiducial vector}, such that $S = \{\C A\bv : A \in G\}$.
\end{defn}

If $\bv$ is a fiducial vector for a Weyl--Heisenberg covariant SIC and $U \in \EU(d)$, then the vector $U\bv$ will be a fiducial vector for a SIC covariant for the group $G'= U \WH(d) U^{-1}$. All currently known SICs are unitary-equivalent to a Weyl--Heisenberg group $\H(d)$ covariant SIC, with a single exception, the Hoggar lines in $\C^8$ \cite{hoggar1,hoggar2}. The Hoggar lines are group covariant for the tensor product $\H(2) \tensor \H(2) \tensor \H(2)$. (The group $\H(2) \tensor \H(2) \tensor \H(2)$ may itself be thought of as a generalized (Weyl)--Heisenberg group.) 

The extended unitaries $U \in \EU(d)$ that preserve the Weyl--Heisenberg covariance property are precisely those for which $U \WH(d) U^{-1} = \WH(d)$, that is, elements of the extended Clifford group, defined as follows.
\begin{defn}\label{def:EC-PEC}
We define several closely related groups.
\begin{itemize}
\item[(1)] 
The \textit{Clifford group} $\CG(d)$ is the normalizer of $\WH(d)$ in $\U(d)$.
\item[(2)]
The \textit{projective Clifford group} $\PC(d)$ is the quotient group of $\CG(d)$ modulo its scalar matrix subgroup.
\item[(3)]
The \textit{extended Clifford group} $\EC(d)$ is the normalizer of $\WH(d)$ in $\EU(d)$. 
\item[(4)]
The \textit{projective extended Clifford group} $\PEC(d)$ is the quotient of the extended Clifford group $\EC(d)$ modulo its scalar matrix subgroup.
\end{itemize}
\end{defn}

The Clifford group $\CG(d)$ is infinite because it contains a copy of $\U(1)$ as scalar matrices $zI$ with $\abs{z}=1$; thus, $\EC(d)$ is infinite. The groups $\PC(d)$ and $\PEC(d)$ are finite groups. (The extended Clifford $\EC(d)$ group is studied in connection with SICs by Appleby \cite{App05}.)

\begin{defn}\label{def:29}
The set of \textit{geometric equivalence classes of Weyl--Heisenberg SICs} is the set of equivalence classes of $\WHSIC{d}$ under the equivalence relation $S({\bv}) \sim_{\PEC} S(U\bv)$ whenever $U \in \PEC(d)$ (equivalently, whenever $U \in \EC(d)$.) This set of equivalence classes is denoted $\WHSIC{d}/\PEC(d)$. (It is not a group.) We denote the geometric equivalence class of a Weyl--Heisenberg SIC $S$ by $[S]$.
\end{defn}

\begin{rmk}\label{rmk:GEC-Sic}
Inside each geometric equivalence class, there are finitely many Weyl--Heisenberg SICs, and their number is at most $\frac{\abs{\PEC(d)}}{d^2}$. In a fixed dimension $d$ their number can vary, as it depends on the size of the stabilizer group in ${\PEC}(d)$ taking a SIC to itself. It is conjectured that every Weyl--Heisenberg SIC has an order $3$ stabilizer element, but larger stabilizer groups are also possible. This paper does not treat this problem but discusses it in \Cref{rmk:whsic}. It can be shown that $\abs{\PEC(d)}$ is of size $O(d^5)$.
\end{rmk}

\section{Fields generated by SICs}\label{sec:fields}

In this section, we define several number fields attached to arbitrary SICs. Many of these fields have been studied before, often within the context of assuming various conjectures about SICs. Here, we focus on relationships between ``SIC fields'' that can be proven rigorously, either in general (\Cref{subsec:NFS}) or in the case of Weyl--Heisenberg SICs (\Cref{subsec:NFS3}). 

In \Cref{sec:algebraic}, we will focus on SICs expressible in terms of algebraic numbers and discuss some conditional results about fields generated by SICs in that setting.

\subsection{Fields generated by general SICs}\label{subsec:NFS}

The fields we study here are each associated to an individual SIC $S$ and depend only on data computed from the orthogonal projections $\{\Pi_{i}: 1 \le i \le d^2\}$ defining the SIC. 
\begin{defn}\label{defn:fields}
Let $\PPi(S) = \{\frac{1}{d}\Pi_1, \ldots, \frac{1}{d}\Pi_{d^2}\}$ be an arbitrary SIC in dimension $d$, with each $\Pi_i = \bv_i\bv_i^\ct$ for a {\em unit vector} $\bv_i$. (Thus each $\bv_i$ is unique up to multiplication by a unit scalar.)
\begin{itemize}
\item[(1)]
The \textit{ratio SIC field} $\fieldvec{S}$ is the field extension of $\Q$ generated by all the ratios $v_{ij}/v_{ik}$ (for $v_{ik} \neq 0$) of the entries of the $d^2$ fiducial vectors $\bv_i$ of $S$.
\item[(2)]
The \textit{triple product SIC field} $\fieldtrip{S}$ is the field extension of $\Q$ generated by all the \textit{triple products} $\Tr(\Pi_i\Pi_j\Pi_k)$ (also called \textit{$3$-vertex Bargmann invariants} \cite[Sec.~8.2]{waldron18}).
\item[(3)]
The \textit{unitary invariant SIC field} $\fieldinv{S}$ is the field extension of $\Q$ generated by all the \textit{unitary invariants} $\Tr(\Pi_{i_1}\Pi_{i_2}\cdots\Pi_{i_n})$ for $n \in \N$ (also called \textit{$n$-vertex Bargmann invariants} \cite[Sec.~8.2]{waldron18}).
\item[(4)]
The \textit{projection SIC field} $\fieldproj{S}$ is the field extension of $\Q$ generated by all the entries of the $d^2$ Hermitian projection matrices $\Pi_j$. 
\item[(5)]
The \textit{extended projection SIC field} $\fieldeproj{S} := \fieldproj{S}\!(\xi_d)$, adjoining $\xi_d = -e^{\frac{\pi i}{d}}$.
\end{itemize} 
\end{defn}

A priori, these fields could be transcendental extensions of $\Q$, as happens in dimension $d=3$.

\begin{rmk}\label{rmk:322}
These fields have previously been studied with various names and notations.

\begin{enumerate}
\item
The ratio SIC field $\fieldvec{S}$ is the {\em SIC field} $\Q[S]$ attached to a general SIC $S$ by Appleby, Flammia, McConnell, and Yard \cite[p.~212 bottom]{appleby1}, as used in Thm.~1 and Conj.~2 of that paper. By definition, this field is independent of scaling factors in the $\bv_i$. It is determined by $\PPi(S)$, because the ratios satisfy
\begin{equation}
\frac{v_{ij}}{v_{ik}} = \frac{v_{ij}\ol{v}_{ik}}{v_{ik}\ol{v}_{ik}} = \frac{(\Pi_i)_{jk}}{(\Pi_i)_{kk}}.
\end{equation}
We may rescale each of the $d^2$ vectors in $S$ to give $S= \{\C \bw_i: 1 \le i \le d^2\}$ so that all entries of each $\bw_i$ belong to $\fieldvec{S}$. A given $\bv_i=\left(v_{i0}, v_{i1}, \cdots, v_{i(d-1)}\right)^T$ has at least one nonzero entry, say $v_{ik}$, and we may set $\bw_{i}=\frac{1}{v_{ik}} \bv_i$, whence $w_{ij}= v_{ij}/v_{ik} \in \fieldvec{S}$. For $d>1$ the vectors $\bw_i$ obtained this way are never unit vectors. 

It is asserted in \cite[Prop.~3]{appleby1} that, for a Weyl--Heisenberg SIC, the ratio SIC field is the same for any SIC $S' \in [S]$. 
The presented proof outline \cite[Sec.~7.1]{appleby1} appears to 
depend on some conjectural equivalences and characterizations of other related fields. 

\item 
Concerning the triple product SIC field $\fieldtrip{S}$, the fact that a SIC is determined up to unitary equivalence by its triple products was proved in 2011 by Appleby, Flammia, and Fuchs \cite[Thm.~3]{siclie1}; see also \cite[Cor.~8.1]{waldron18}. We show unconditionally for a Weyl--Heisenberg SIC $S$ that $\fieldtrip{S}$ is an invariant of its geometric $\PEC(d)$-orbit $[S]$; see \Cref{prop:triple-prod-field}.

\item
The unitary invariant SIC field $\fieldinv{S}$ is generated by the full set of projective unitary invariants of a general finite set of lines in $\C^d$; see \cite{CW16} and \cite[Ch.~8, Thm.~8.2]{waldron18}. We show in \Cref{prop:sicfields} that $\fieldtrip{S}= \fieldinv{S}$.

\item 
The projection SIC field $\fieldproj{S}$, in the special case that $S$ is a Weyl--Heisenberg SIC, coincides with the field $L$ appearing in \cite[Sec.~7]{appleby1}.

\item
The extended projection SIC field $ \fieldeproj{S}$, in the special case that $S$ is a Weyl--Heisenberg SIC, is the {\em SIC field} $\E$ as defined by Appleby, Flammia, McConnell, and Yard \cite[Sec.~4]{appleby2}. This definition of the SIC field is also used in \cite{acfw} and is treated in detail in Waldron \cite[Sec.~14.20]{waldron18}. Appleby, Flammia, McConnell, and Yard \cite{appleby2} state that the field $\E$ is well-defined on the $\PEC(d)$-orbit $[S]$ of a Weyl--Heisenberg SIC $S$. We supply a proof in \Cref{prop:215}.
\end{enumerate}
\end{rmk}

The next result establishes some inclusion relations among these fields.

\begin{prop}\label{prop:sicfields}
Let $S$ be an arbitrary line-SIC in dimension $d$. The fields associated to $S$ are related in the following ways.
\begin{itemize}
\item[(1)] $\fieldproj{S} = \fieldvec{S}\ol{\fieldvec{S}}$ (the compositum of $\fieldvec{S}$ and its complex conjugate field).
In particular, $\fieldproj{S}$ is invariant under complex conjugation. 
\item[(2)] $\fieldtrip{S} = \fieldinv{S} \subseteq \fieldproj{S}$.
\end{itemize}
\end{prop}
\begin{proof}
Write $\PPi(S) = \{\frac{1}{d}\Pi_1, \ldots, \frac{1}{d}\Pi_{d^2}\}$.

\textit{(1).} Write $\Pi_i = \bv_i\bv_i^\ct$, where $\bv_i= (v_{i0},v_{i1}, \ldots, v_{i(d-1)})^T$ and $\abs{\bv_i}^2=1$, for $1 \le i \le d^2$. The entries of $\Pi_i$ are of the form $v_{im}\ol{v}_{in}$ for $0 \le m,n \le d-1$. When both $v_{im} \neq 0$, $v_{in} \neq 0$, these entries can be expressed as
\begin{equation}
v_{im}\ol{v}_{in} = \frac{v_{im}\ol{v}_{in}}{\abs{\bv_i}^2} 
= \frac{1}{\sum_{p=0}^{d-1}(v_{ip}/v_{im})\ol{(v_{ip}/v_{in})}} \in \fieldvec{S}\ol{\fieldvec{S}},
\end{equation}
where $0 \le m, n, p \le d-1$.
Thus, $\fieldproj{S} \subseteq \fieldvec{S}\ol{\fieldvec{S}}$. To show the reverse inclusion, note that for $v_{in} \neq 0$,
\begin{equation}
\frac{v_{im}}{v_{in}} = \frac{v_{im}\ol{v}_{in}}{v_{in}\ol{v}_{in}} \mbox{ and }
\frac{\ol{v_{im}}}{\ol{v_{in}}}= \frac{v_{in}\ol{v}_{im}}{v_{in}\ol{v}_{in}}.
\end{equation}

\textit{(2).} 
First, the inclusion $\fieldinv{S} \subseteq \fieldproj{S}$ of (2) follows because $\Tr(\Pi_{i_1}\Pi_{i_2}\cdots\Pi_{i_n})$ is an integer polynomial in the entries of the various $\Pi_j$.

Second, we show the equality $\fieldtrip{S} = \fieldinv{S}$. Clearly $\fieldtrip{S} \subseteq \fieldinv{S}$. To show the reverse inclusion, first note that the $\Pi_i$ form a basis for the space of $d \times d$ complex matrices. This follows from non-degeneracy of the trace pairing because
\begin{equation}
\Tr(\Pi_i\Pi_j) = \frac{\delta_{ij}d+1}{d+1}, 
\mbox{ where } 
\delta_{ij} = 
\begin{cases}
1, & \mbox{if } i=j, \\
0, & \mbox{if } i\neq j,
\end{cases}
\end{equation}
and so the $d^2 \times d^2$ matrix $T$ with entries $\Tr(\Pi_i\Pi_j)$ is $T = \frac{d}{d+1}I + \frac{1}{d+1}J$ (where $I$ is the $d \times d$ identity matrix and $J$ is the $d \times d$ all-ones matrix) having $(T-dI)(T-\frac{d}{d+1}I) = 0$ and thus $\det T \neq 0$.

We express
\begin{equation}\label{eq:structureprod}
\Pi_i\Pi_j = \sum_{\ell} \alpha_{ij}^\ell \Pi_\ell
\end{equation}
for \textit{structure constants} $\alpha_{ij}^\ell \in \C$, having $1 \le i, j , \ell \le d^2$. We have 
\begin{equation}
\sum_{\ell} \alpha_{ij}^\ell = \Tr(\Pi_i\Pi_j) = \frac{\delta_{ij}d+1}{d+1}.
\end{equation}
Moreover, the triple products are related to the structure constants by
\begin{align}
\Tr(\Pi_i\Pi_j\Pi_k)
&= \sum_{\ell} \alpha_{ij}^\ell \Tr(\Pi_k\Pi_\ell) \\
&= \frac{d}{d+1} \alpha_{ij}^k + \frac{1}{d+1}\sum_\ell \alpha_{ij}^\ell \\
&= \frac{d}{d+1} \alpha_{ij}^k + \frac{\delta_{ij}d+1}{(d+1)^2}.
\end{align}
Thus, $\alpha_{ij}^k = \frac{d+1}{d}\Tr(\Pi_i\Pi_j\Pi_k)-\frac{d(\delta_{ij}d+1)}{d+1}$, and the field generated over $\Q$ by the structure constants is $\fieldtrip{S}$. Moreover, using \eqref{eq:structureprod} repeatedly, we can express any product $\Pi_{i_1}\Pi_{i_2}\cdots\Pi_{i_n}$ in the $\left\{\Pi_\ell\right\}_\ell$-basis with coefficients that are integer polynomials in the structure constants. Thus, since $\Tr(\Pi_\ell)=1$, $\Tr(\Pi_{i_1}\Pi_{i_2}\cdots\Pi_{i_n})$ is also an integer polynomial in the structure constants and is thus contained in $\fieldtrip{S}$.
\end{proof}

\begin{prop}\label{prop:triple-prod-field}
Let $S$ be an arbitrary line-SIC in dimension $d$. Then the triple product field satisfies the invariance property $\fieldtrip{US} = \fieldtrip{S}$ for any $U \in \EU(d)$. 
\end{prop}

\begin{proof}
Write $S= \{\C\bv_i : 1 \le i \le d^2\}$, and set $\PPi(S) = \{\frac{1}{d}\Pi_1, \ldots, \frac{1}{d}\Pi_{d^2}\}$ with $\Pi_i= \bv_i \bv_i^{\dagger}$. Then set $\PPi(US) = \{\frac{1}{d}\Pi_1', \ldots, \frac{1}{d}\Pi_{d^2}'\}$. Write $U=CV$, where $C$ is identity or complex conjugation, and $V \in \U(d)$; then either $\Pi_i' = V^{-1}\Pi_iV$, or $\Pi_i' = V^{-1}\Pi_i^\top V$. In the first case, $\Tr(\Pi_i'\Pi_j'\Pi_k') = \Tr(V^{-1}\Pi_i\Pi_j\Pi_kV) = \Tr(\Pi_i\Pi_j\Pi_k)$. In the second case,
\begin{equation}
\Tr(\Pi_i'\Pi_j'\Pi_k') = \Tr\!\left(V^{-1}\Pi_i^\top\Pi_j^\top\Pi_k^\top V\right) = \Tr\!\left((\Pi_k\Pi_j\Pi_i)^\top\right) = \Tr(\Pi_k\Pi_j\Pi_i).
\end{equation}
So $\fieldtrip{US} = \fieldtrip{S}$ in both cases.
\end{proof}

We derive some unconditional implications when these fields give Galois extensions of $\Q$.
\begin{lem}\label{lem:normal-conditions}
Let $S$ be an arbitrary line-SIC in dimension $d$.
\begin{itemize}
\item[(1)]
If $\fieldvec{S}$ is invariant under complex conjugation $\sigma_c$, then $\fieldvec{S} = \fieldproj{S}$. Thus, if $\fieldvec{S}$ is a Galois extension of $\Q$, then $\fieldproj{S}$ is a Galois extension of $\Q$.
\item[(2)]
If $\fieldproj{S}$ is a Galois extension of $\Q$, then $\fieldeproj{S}$ is a Galois extension of $\Q$. 
\end{itemize}
\end{lem}
\begin{proof}
The implication $\fieldvec{S}= \fieldproj{S}$ follows from \Cref{prop:sicfields}(1). Moreover, if $K=\fieldproj{S}$ is Galois over $\Q$, then $\fieldeproj{S}$ is the compositum $KL$, with $L= \Q(e^{\frac{\pi i}{d}})$ an abelian extension of $\Q$, so the compositum is a Galois extension of $\Q$.
\end{proof}

\subsection{Fields generated by Weyl--Heisenberg SICs}\label{subsec:NFS3}

We now derive some properties of the fields $\fieldvec{S}$, $\fieldtrip{S}$, $\fieldproj{S}$, and $\fieldeproj{S}$ that are provable in the special case of Weyl--Heisenberg SICs.

\begin{prop}\label{prop:fieldproj}
Let $S$ be a Weyl--Heisenberg line-SIC $S = S(\bv)$ of dimension $d \ge 2$. Then:

\begin{itemize}
\item[(1)]
The field $\fieldtrip{S}$ contains the $d$-th root of unity $\zeta_d=e^{\frac{2\pi i}{d}}$.
\item[(2)]
The index $[\fieldeproj{S} : \fieldproj{S}] \in \{1,2\}$. It is $1$ whenever $d$ is odd. 
\end{itemize}
\end{prop}
\begin{proof}
We first note that for arbitrary column vectors $\bw_i \in \C^d$ the triple inner product identity
\begin{align}
\langle \bw_1, \bw_2\rangle \langle \bw_2, \bv_3\rangle \langle \bw_3, \bw_1\rangle 
&= \Tr\!\left(\bw_1^{\ct} \bw_2 \bw_2^{\ct} \bw_3 \bw_3^{\ct} \bw_1\right)\\
&= \Tr\!\left(\bw_1 \bw_1^{\ct} \bw_2 \bw_2^{\ct} \bw_3 \bw_3^{\ct} \right)\\
&= \Tr\!\left(\Pi_{\bw_1} \Pi_{\bw_2} \Pi_{\bw_3}\right)
\end{align}
expressing $\langle \bw_1, \bw_2\rangle \langle \bw_2, \bv_3\rangle \langle \bw_3, \bw_1\rangle$ as a trace of a product of projections. Consequently, when such a triple inner product consists of vectors $\bw_i$ in a line-SIC $S$, then it belongs to $\fieldtrip{S}$.

\textit{(1).} 
Let $\bv = (v_0, v_1, \ldots, v_{d-1})^\top$ be a fiducial vector for the Weyl--Heisenberg SIC $S$, and let $\Pi= \bv \bv^{\dagger}$. Let $X, Z$ be Weyl--Heisenberg group generators in \Cref{def:WH-group}. The ratio of triple inner products 
\begin{align}
\frac{\langle \bv, Z^{-1}\bv \rangle \langle Z^{-1}\bv, X\bv \rangle \langle X\bv, \bv \rangle}{\langle \bv, X^{-1}\bv \rangle \langle X^{-1}\bv, Z\bv \rangle \langle Z\bv, \bv \rangle}
&= \frac{\langle Z\bv, \bv \rangle \langle \bv, ZX\bv \rangle \langle X\bv, \bv \rangle}{\langle X\bv, \bv \rangle \langle \bv, XZ\bv \rangle \langle Z\bv, \bv \rangle} \\
&= \frac{\langle \bv, ZX\bv \rangle}{\langle \bv, XZ \bv \rangle} = \zeta_d.
\end{align}
Via the SIC relations, $\abs{\langle \bv_{i}, \bv_{j} \rangle} = \frac{1}{\sqrt{d+1}} \abs{\bv}^2$ if $i \ne j$, and $\abs{\langle \bv_{i}, \bv_{j} \rangle} = \abs{\bv}^2$ if $i=j$. All factors in these ratios are nonzero. Therefore, $\zeta_d \in \fieldtrip{S}$.

\textit{(2).} 
By (1), since $\fieldtrip{S} \subseteq \fieldproj{S}$, we have $\zeta_d \in \fieldproj{S}$. Now $\fieldeproj{S}= \fieldproj{S} (e^{\frac{\pi i}{d}})$ is at most a quadratic extension of $\fieldproj{S}$, since the field index $[\Q(e^{\frac{\pi i}{d}}): \Q (e^{\frac{2\pi i}{d}})]$ is $2$ if $d$ is even and is $1$ if $d$ is odd. 
\end{proof}

The next result bounds the index $[\fieldproj{S}: \fieldtrip{S}]$ for Weyl--Heisenberg line-SICs. 

\begin{prop}\label{prop:fieldtrip-up}
Let $S$ be a Weyl--Heisenberg line-SIC $S = S(\bv)$ of dimension $d \ge 2$. 

\begin{enumerate}
\item[(1)]
Let $L(S) = \fieldtrip{S}[\langle \bv, X\bv\rangle, \langle \bv, Z\bv\rangle]$. Then $\fieldproj{S} \subseteq L(S)$. 
\item[(2)]
One has $L(S) = \fieldtrip{S}\!({\sqrt[d]{\beta_1}, \sqrt[d]{\beta_2}})$, for some $\beta_1, \beta_2 \in \fieldtrip{S}$. Thus $L(S)$ is an abelian extension of $\fieldtrip{S}$ whose degree divides $d^2$, and hence $[\fieldproj{S}: \fieldtrip{S}]$ divides $d^2$.
\end{enumerate}
\end{prop}
\begin{proof}
In this proof, we set $D_\p = X^{p_1}Z^{p_2}$, where $X, Z$ are Weyl--Heisenberg group generators as in \Cref{def:WH-group}, and $\p = (p_1,p_2) \in \left(\Z/d\Z\right)^2$. The $D_\p$ are a choice of coset representatives for $\WH(d)$ modulo its center. (Note that this definition of $D_\p$ differs by scalar phase factors from the quantum mechanics definition of \textit{displacement operators} $D_\p$ used with SICs, e.g., \cite[eq.~(4)]{appleby2}.)
 
As in the previous proof, for arbitrary column vectors $\{\bw_i \in \C^d: 1 \le i \le r$ for each $r \ge 1$ we have the $r$-fold cyclic inner product identity
\begin{align}
\langle \bw_1, \bw_2\rangle \langle \bw_2, \bw_3\rangle \ldots \langle \bw_r, \bw_1\rangle 
&= \Tr \left( \Pi_{\bw_1} \Pi_{\bw_2}\cdots \Pi_{\bw_r}\right)
\end{align}
expressed as a trace of a cyclic product of projection matrices. If the vectors $\bw_i$ are vectors in a SIC $S$, then the right side is by definition in $\fieldinv{S}$, hence in $\fieldtrip{S}$ by \Cref{prop:sicfields}. 

\textit{(1).} 
Let $\bv = (v_0, v_1, \ldots, v_{d-1})^\top$ be a fiducial vector for the Weyl--Heisenberg SIC $S$, and let $\Pi= \bv \bv^{\dagger}$. For $X, Z$ as in \Cref{def:WH-group}, let $D_{\p}= X^{p_1} Z^{p_2}$, where $\p= (p_1, p_2)$ with $0 \le p_i < d$. Then, the vectors $\bv_{\p} = X^{p_1}Z^{p_2}\bv$ form a complete set of vectors for the line-SIC, all having $\norm{\bv_{\p}}^2= \norm{\bv}^2$. Below, we will also use the alternate numbering $\bv_i$ with $i = d p_1+ p_2$ having $0 \le i < d^2$.

Let ${\bf e}_1 = (1,0)$ and ${\bf e}_2 = (0,1)$ and set $\alpha_1:=\nu_{{\bf e}_1}= \nu_{(1,0) }= \langle \bv, X \bv \rangle$ and $\alpha_2= \nu_{{\bf e}_2}= \nu_{(0,1)}= \langle \bv, Z \bv \rangle$. Define the field $L(S) := \fieldtrip{S}(\alpha_1, \alpha_2)$. By \Cref{prop:fieldproj}, $\zeta_d \in \fieldtrip{S}$.

To show $\fieldproj{S} \subseteq L(S)$, it suffices to show that all entries $(\Pi_i)_{k, \ell}$ belong to $L(S)$ for $1 \le i \le d^2$ and $0 \le k, \ell < d$. Set $\nu_\p = \langle \bv, D_\p \bv \rangle$ with $\p= (p_1, p_2)$ as above. Now for $\p = p_1{\bf e}_1 + p_2{\bf e}_2$, the product $\nu_{{\bf e}_1}^{p_1} \nu_{{\bf e}_2}^{p_2} \nu_{-\p}$ is a $(p_1+p_2+1)$-fold cyclic inner product of various $\nu_{i}$ (up to multiplication by a $d$-th root of unity), and thus by \Cref{prop:sicfields}(2), $\nu_{{\bf e}_1}^{p_1} \nu_{{\bf e}_2}^{p_2} \nu_{-\p} \in \fieldtrip{S}(\alpha_1, \alpha_2, \zeta_d)$. We conclude that all $\nu_{-\p}= \langle \bv, X^{-p_1}Z^{-p_2}\bv \rangle$ are in $\fieldtrip{S}(\alpha_1, \alpha_2, \zeta_d)$.

{\bf Claim.} {\em The $D_\p$ for $\p \in \left(\Z/d\Z\right)^2$ form a basis of the matrix algebra ${\rm Mat}_{d \times d}(\Q(e^{\frac{2\pi i}{d}}))$ over the field $\Q(\zeta_d)=\Q(e^{\frac{2\pi i}{d}})$.}

{\it Proof of claim.}
By direct computation,
\begin{equation}
D_\p^\ct D_\q 
= Z^{-p_2}X^{-p_1} X^{q_1}Z^{q_2} 
= Z^{-p_2} X^{-p_1+q_1} Z^{q_2} 
= \zeta_d^{(p_1-q_1)p_2} X^{-p_1+q_1} Z^{-p_2} Z^{q_2} 
= \zeta_d^{(p_1-q_1)p_2} D_{-\p+\q}.
\end{equation}
The trace of a Weyl--Heisenberg matrix is zero unless that matrix is a multiple of the identity, so
\begin{equation}
\Tr(D_\p^\ct D_\q)
= d\,\zeta_d^{(p_1-q_1)p_2} \delta_{\p\q} = d \,\delta_{\p\q}.
\end{equation}
It follows that the $D_\p$ form an orthogonal basis for ${\rm Mat}_{d \times d}(\Q(\zeta_d))$ over $\Q(\zeta_d)$ with respect to the non-degenerate Hermitian pairing $(A,B) \mapsto \Tr(A^\ct B)$. (Concretely, if $M \in {\rm Mat}_{d \times d}(\Q(\zeta_d))$, then $M = \frac{1}{d}\sum_{\p} \Tr(D_\p^\ct M) D_\p$, and the coefficients $\Tr(D_\p^\ct M) \in \Q(\zeta_d)$.) This concludes the proof of the subclaim. 

\medskip

It follows from the claim that all the matrices $E_{k, \ell}$ for $0 \le k, \ell \le d-1$ with entry $1$ in the $(k, \ell)$-position and $0$ elsewhere are linear combinations of the $D_\p$ with coefficients over the field $\Q(\zeta_d)$. Thus 
\begin{equation}
(\Pi_{1})_{k, \ell}= \langle \bv, E_{k, \ell} \bv \rangle \in \fieldtrip{S}(\alpha_1, \alpha_2)=L(S).
\end{equation}
Now for $i=(d-1)p_1 +p_2$ we have 
\begin{equation}
(\Pi_{i})_{k, \ell}= \langle X^{p_1}Z^{p_2}\bv, E_{k, \ell} X^{p_1}Z^{p_2} \bv\rangle = \langle \bv, X^{-p_1}Y^{-p_2}E_{k, \ell} X^{p_1}Z^{p_2} \bv\rangle
\end{equation}
by unitarity. (Recall that, if $U \in U(d)$ is unitary, then $\langle \bw_1, \bw_2 \rangle = \langle U\bw_1, U \bw_2 \rangle$.) The right side $X^{-p_1}Y^{-p_2}E_{k, \ell} X^{p_1}Z^{p_2}$ is expressible as a linear combination of the $D_p$ with coefficients in $\Q(\zeta_d)$, each term of the expanded inner product is in $\fieldtrip{S}(\alpha_1, \alpha_2)=L(S)$. Hence $(\Pi_i)_{k, \ell} \in L(S)$ for all $i, k, \ell$. We conclude that $\fieldproj{S} \subseteq L(S)$.

\textit{(2).} 
Set $\beta_1 = \alpha_1^d = \nu_{{\bf e}_1}^{d}$. Since $X$ is unitary and $X^d=I$, we have
\begin{align}
\beta_1 
&= \langle \bv, X\bv \rangle^d \\
&= \langle \bv, X\bv \rangle \langle X\bv, X^2\bv \rangle \cdots \langle X^{d-1}\bv, X^d\bv \rangle\\
&= \langle \bv, X\bv \rangle \langle X\bv, X^2\bv \rangle \cdots \langle X^{d-1}\bv, \bv \rangle \in \fieldinv{S}.
\end{align}
Since $\fieldinv{S} = \fieldtrip{S}$ we obtain $\beta_1= (\alpha_1)^d \in \fieldtrip{S}$. Now $\alpha_1 = \sqrt[d]{\beta_1}$, and since $\zeta_d \in \fieldtrip{S}$ all of its algebraic conjugates are in $L(S)$.
 
Set $\beta_2 = (\alpha_2)^d= \nu_{{\bf e}_2}^{d} = \langle \bv, Z\bv^d \rangle$. Since $Z$ is unitary and $Z^d=I$ we obtain similarly that $\beta_2= \alpha_2^d \in \fieldtrip{S}$. Thus $\alpha_2=\sqrt[d]{\beta_2}$, and all of its algebraic conjugates are in $L(S)$. We conclude that adjoining these two elements to $\fieldtrip{S}$ gives an abelian extension $L(S)$ of $\fieldtrip{S}$, which is a Kummer extension whose degree divides $d^2$. The field $L(S)$ contains $\fieldproj{S}$, and $\fieldtrip{S} \subseteq \fieldproj{S}$, which implies $[\fieldproj{S}: \fieldtrip{S}]$ divides $d^2$.
\end{proof}

We now study the field $\fieldeproj{S}$ on the $\PEC(d)$-orbit of $S$ and show it is a $\PEC(d)$-invariant. We recall a characterization of generators of the Clifford group.
\begin{prop}\label{prop:Waldron}
The Clifford group $\CG(d)$ is generated by all unitary scalar matrices $z I$ with $\abs{z}=1$, together with
\begin{equation}
X, Z, F, R
\end{equation}
where $X, Z$ are Weyl--Heisenberg group generators, $F$ is the discrete Fourier transform
\begin{equation}
F_{j, k} = \frac{1}{\sqrt{d}} e^{ \frac{2 \pi i jk}{d}} \,\mbox{ for }\, 0 \le j, k \le d-1,
\end{equation}
and $R$ is a diagonal matrix with diagonal entries
\begin{equation}
R_{j,j} = \xi_d^{j(j+d)}= (-1)^{jd} e^{\frac{\pi i j^2}{d}} \,\mbox{ for }\, 0 \le j \le d-1.
\end{equation}
\end{prop} 
\begin{proof}
This is \cite[Thm.~14.1]{waldron18}. Every group element is a product of finitely many generators, with exactly one use of a unitary scalar matrix $zI$.
\end{proof}

\begin{prop}\label{prop:215}
Let $S$ be a Weyl--Heisenberg line-SIC $S= S(\bv)$ of dimension $d$.
\begin{itemize}
\item[(1)] 
For any $M \in \EC(d)$,
\begin{equation}
\fieldeproj{MS} = \fieldeproj{S},
\end{equation}
where $MS$ denotes the Weyl--Heisenberg line-SIC having fiducial vector $\bw:=M\bv$.
\item[(2)] 
$\fieldeproj{S}$ is equal to the compositum of all the $\fieldproj{MS}$ for $M \in \EC(d)$. Thus $\fieldeproj{S}$ depends only on the geometric equivalence class $[S]$.
\end{itemize}
\end{prop}
\begin{proof} 
\textit{(1).} To show $\fieldeproj{MS} = \fieldeproj{S}$, note that scalar matrices $zI$ with $\abs{z}=1$ act trivially on projections, with $(zI\bv)(zI\bv)^\ct = \bv\bv^\ct$. The matrices $X$, $Z$, and $R$ all have entries in $\Q(\xi)$ (where $\xi = \xi_d = -e^{\frac{\pi i}{d}}$). The matrix $\sqrt{d}F$ also has entries in $\Q(\xi)$, and $(F\bv)(F\bv)^\ct = \frac{1}{d}(\sqrt{d}F\bv)(\sqrt{d}F\bv)^\ct$. Writing any $M \in EC(d)$ as a finite product of the generators given in \Cref{prop:Waldron} shows that $(M\bv)(M\bv)^\ct = M\bv\bv^\ct M^\ct$ has entries in $\fieldeproj{S}$ (the compositum of $\fieldproj{S}$ and $\Q(\xi)$). This argument shows $\fieldproj{MS} \subseteq \fieldeproj{S}$, so $\fieldeproj{MS} \subseteq \fieldeproj{S}$. By considering $M^{-1}$, we also have $\fieldeproj{S} = \fieldeproj{M^{-1}MS} \subseteq \fieldeproj{MS}$, so $\fieldeproj{MS} = \fieldeproj{S}$.

\textit{(2).} To show $\fieldeproj{S}$ is equal to the compositum of all the $\fieldproj{MS}$, it suffices to show that $\xi= - e^{-\frac{\pi i}{d}}$ occurs in the compositum of two fields $\fieldproj{S}$ and $\fieldproj{RS}$ where $R$ is the element of $\CG(d)$ given in \Cref{prop:Waldron}. Given a fiducial $\bv$ generating $S=S(\bv)$ having $v_0v_1 \ne 0$, we consider the fiducial vector $\bw = R\bv$ of a geometrically equivalent SIC. It has entries $(R\bv)_j = \xi^{j(j+d)}v_j$ so $R\bv_0= v_0$ and $R\bv_1= \xi v_1$. Therefore $\Pi'= \bw \bw^{\ct}$ has
\begin{equation}
(\Pi')_{1,0} = \xi v_1 \overline{v}_0 \in \fieldproj{RS},
\end{equation}
while $v_1 \overline{v}_0 \in \fieldproj{S}$. Now $v_0v_1 \ne 0$ implies we may take their quotient in the compositum field and so get $\xi_d \in \fieldproj{S}\fieldproj{RS}$, as required.

Thus $\fieldeproj{S}$ is a well-defined invariant of the $\PEC(d)$-orbit of $S$.
\end{proof}

\section{Algebraic SICs and the Galois action}\label{sec:algebraic}
 
We now treat the subclass of algebraic SICs, for which the fields defined in \Cref{subsec:NFS} are algebraic number fields. In this paper, $\ol{\Q}$ is defined to be the algebraic closure of $\Q$ inside $\C$, so there is a canonical choice of complex conjugation automorphism $\sigma_c \in \Gal(\ol{\Q}/\Q)$, $\sigma_c(z)= \overline{z}$. We show that there is a well-defined Galois action on algebraic SICs and that the property of $S$ being an algebraic SIC is preserved by the subgroup of all automorphisms $\sigma \in \Gal(\ol{\Q}/\Q)$ that commute with $\sigma_c$ when restricted to the Galois closure of $\fieldproj{S}$.

It is conjectured for dimensions $d \ne 3$ that all Weyl--Heisenberg SICs are algebraic. In \Cref{subsec:G-WHSIC}, we define a combined geometric and Galois action on all algebraic Weyl--Heisenberg SICs and classify them into multiplets. 
 
\subsection{Algebraic SICs}\label{subsec:NFS2}

We consider SICs generated by algebraic vectors.
 
\begin{defn}\label{defn:alg-vectors}
A complex line $\C \bv$ generated by a nonzero vector $\bv \in \C^d$ is a {\em projective-algebraic line} if all the ratios $v_i/v_j \in \ol{\Q}$ (for $v_j \ne 0$). (The definition is independent of the choice of $\bv$.) A nonzero vector such that $\C\bv$ is a projective-algebraic line is called a \textit{projective-algebraic vector}.

An \textit{algebraic vector} is any $\bv \in {\ol{\Q}}^d$. For a projective-algebraic vector $\bv$, its {\em associated $\ol{\Q}$-line} is
\begin{equation}
\Qbarline{\bv} := \left\{\lambda \bv: \lambda \in \C \mbox{ and } \lambda \bv \in \ol{\Q}^{d}\right\} = \C\bv \cap \ol{\Q}^d.
\end{equation}
\end{defn}
The next result implies that the algebraic class $\Qbarline{\bv}$ of an algebraic complex line always includes algebraic unit vectors.
 
\begin{lem}\label{lem:alg-vector}
If $\bv \in \C^d$ is a projective-algebraic vector, then there exists $\lambda \in \C^{\times}$ such that $\bx = \lambda \bv$ is algebraic and $\vectornorm{\bx}=1$. The associated $\ol{\Q}$-line $\Qbarline{\bx}$ is independent of $\lambda$.
\end{lem}
\begin{proof}
Let $\bv$ have some $v_i \ne 0$ and set $\mu = \frac{1}{v_i}$. Then $\bw := \mu \bv = \left(\frac{v_0}{v_i}, \frac{v_1}{v_i}, \ldots, \frac{v_{d-1}}{v_i}\right) \in \ol{\Q}^d$ by the projective-algebraicity of $\bv$. Now the complex conjugates $\sigma_c\!\left(\frac{v_k}{v_i}\right) = \frac{\ol{v}_k}{\ol{v}_i}$ are algebraic numbers, hence so is $\abs{\frac{v_k}{v_i}}^2$. Thus, $\vectornorm{\bw} = \sqrt{\sum_{k} \abs{\frac{v_k}{v_i}}^2}$ is a positive real algebraic number. We set $\lambda = \frac{\mu}{\vectornorm{\bw}}$ and obtain the required $\bx := \lambda \bv=\frac{\bw}{\vectornorm{\bw}}$, having $\vectornorm{\bx} = 1$. 
 
Finally, if $\lambda$ and $\lambda'$ give two algebraic $\bx, \bx'\!$, necessarily $\lambda'/\lambda \in \ol{\Q}$.
\end{proof}

\begin{defn}\label{defn:alg-SIC}
A $d$-dimensional line-SIC $S =\{ \bv_i: 1 \le i \le d^2\}$ is {\em algebraic} if each of its $d^2$ vectors $\bv_i$ are projective-algebraic (or equivalently, if $\fieldvec{S}$ is a finite extension of $\Q$).
\end{defn}

\begin{lem}\label{lem:alg-SIC}
Algebraic line-SICs satisfy the following properties.
\begin{itemize}
\item[(1)]
A line-SIC is algebraic if and only if, for unit vectors $\bv_i$, all the SIC projection matrices $\{ \Pi_{\bv_i}: 1 \le i \le d^2 \}$ have algebraic number entries.
\item[(2)]
For an algebraic SIC $S$, the fields $\fieldvec{S}$, $\fieldtrip{S}$, $\fieldproj{S}$ and $\fieldeproj{S}$ are algebraic number fields, that is, finite algebraic extensions of $\Q$.
\item[(3)]
A Weyl--Heisenberg line-SIC with fiducial vector $\bv$ is an algebraic SIC if and only if $\bv$ is a projective-algebraic vector. 
\end{itemize}
\end{lem}
\begin{proof} 
\textit{(1).} 
This follows using \Cref{lem:alg-vector}.

\textit{(2).} 
That $\fieldvec{S}$ is an algebraic number field follows directly from the definition of algebraic. Given a line-SIC specified by $d^2$ projective algebraic lines $\C{\bv_i}$, using (1) we immediately have $\fieldtrip{S}$ being an algebraic number field. Furthermore, one may rescale all $\bv_i$ to be algebraic unit vectors by \Cref{lem:alg-vector}. Now the associated projection matrices $\Pi_{\bv} = \bv_i \bv_i^{\ct}$ have all entries in $\ol{\Q}$, so $\fieldproj{S}$ is an algebraic number field. Finally, $\fieldeproj{S}$ is a finite extension of $\fieldproj{S}$.

\textit{(3).} 
If the fiducial vector $\bv$ of a Weyl--Heisenberg line-SIC $S$ is projective-algebraic, then by \Cref{lem:alg-vector}, $\bv$ may be rescaled to be an algebraic unit vector $\bw = \lambda\bv$. The Weyl--Heisenberg group $\WH(d)$ consists of matrices with algebraic entries, hence its action on $\bw$ shows that all $\bw_i$ for $1 \le i \le d^2$ are algebraic, so $S$ is an algebraic line-SIC. The converse direction is immediate.
\end{proof}

\begin{lem}\label{lem:alg-SIC-PEC}
If $S$ is an algebraic Weyl--Heisenberg line-SIC, then every SIC $S'$ in its $\PEC(d)$-orbit $[S]$ is an algebraic Weyl--Heisenberg line-SIC. 
\end{lem}
\begin{proof}
The group $\PEC(d)$ is a finite group. According to \Cref{lem:alg-SIC}(3) it suffices to show that its elements map projective-algebraic fiducial vectors
to projective-algebraic fiducial vectors. The action of complex conjugation has this property, so it suffices to restrict consideration to unitaries in $\PC(d)$. This projective group has a finite cover $\widetilde{\PC}(d) \surj \PC(d)$ that is a finite matrix subgroup of $\C(d)$; this cover can be given explicitly as $\widetilde{\PC}(d) = \langle \xi_d I, X, Z, F, R \rangle$ in the notation of \Cref{prop:Waldron}. This finite cover consists of matrices with entries in $\ol{\Q}$, so each matrix in $\widetilde{\PC}(d)$ maps projective-algebraic fiducial vectors to projective-algebraic fiducial vectors.
\end{proof} 
\subsection{Galois actions on algebraic SICs}\label{subsec:G-SIC}

There is a well-defined action of $\Gal(\ol{\Q}/\Q)$ on algebraic complex lines which maps them to other algebraic complex lines. Given $\sigma \in \Gal(\ol{\Q}/\Q)$ and $\bv \in \ol{\Q}^d$, we send $\C\bv \mapsto \C \sigma(\bv) $. This action is well-defined on algebraic complex lines, since if $\lambda \in \ol{\Q}^{\times}$, then $\C \lambda \bv = \C\bv $ while
\begin{equation}
\C \sigma(\lambda\bv)
= \C \sigma(\lambda)\sigma(\bv)
= \C\sigma(\bv).
\end{equation}
Under this Galois action, the image under $\sigma$ of $d^2$ algebraic complex lines will be $d^2$ algebraic complex lines. If these algebraic lines form a line-SIC, then the image set of lines can be either the same SIC, a different SIC, or not a SIC at all.

For a particular SIC, there is a restricted set of Galois automorphisms (which may depend on the line-SIC) that map it to another line-SIC.

\begin{defn}\label{def:cent-cc}
Let $\sigma_c \in \Gal(\ol{\Q}/\Q)$ denote complex conjugation and $E$ an algebraic Galois extension of $\Q$. Let $\cent{\sigma_c}{E}$ denote the centralizer of $\sigma_c$ restricted to $\Gal(E/\Q)$, that is,
\begin{equation}
\cent{\sigma_c}{E} := \{\sigma \in \Gal(E/\Q): \, \sigma \sigma_c (\alpha)= \sigma_c \sigma(\alpha) \mbox{ for all } \alpha \in E\}.
\end{equation}
It is a subgroup of $\Gal(E/\Q)$.
\end{defn} 

\begin{rmk}\label{rmk:47}
Note that the centralizer $\cent{\sigma_c}{E}$ is equal to the normalizer $\normalizer{\sigma_c}{E}$ (because the group generated by $\sigma_c\hspace{-.25em}\mid\hspace{-.25em}E$ has order $1$ or $2$), and thus (by Galois theory) $\cent{\sigma_c}{E}$ is precisely the set of automorphisms preserving $E \cap \R$ setwise (because $E \cap \R$ is the fixed field of the group generated by $\sigma_c\hspace{-.25em}\mid\hspace{-.25em}E$).
\end{rmk}
For a algebraic line-SIC $S$, we will show there is a Galois field $E$ (depending on $S$) for which $\cent{\sigma_c}{E}$ is such a restricted set. Since neither the field $\fieldproj{S}$ nor $\fieldeproj{S}$ is known to always be a Galois extension of $\Q$, we give a name to the Galois closure of $\fieldeproj{S}$.
\begin{defn}\label{def:Galois-proj-field}
Given an algebraic line-SIC $S$, let $\Nfieldproj{S}$ denote the normal closure of $\fieldeproj{S}$ over $\Q$, and call it the {\em Galois-projection SIC field} of $S$.
\end{defn}

The next proposition shows that $\cent{\sigma_c}{\Nfieldproj{S}}$ is a group of Galois automorphisms that preserve the property of being an algebraic line-SIC. For $d \ge 2$, $\cent{\sigma_c}{\Nfieldproj{S}}$ is always nontrivial because it contains complex conjugation, which acts nontrivially because $\fieldproj{S}$ cannot be contained in $\R$. (A set of real equiangular lines in $\R^d$ has cardinality at most $\frac{d(d+1)}{2}< d^2$; see \cite{delsarte}.)

\begin{prop}\label{lem:Galois-action-SICs}
Let $S$ be an algebraic line-SIC. If $\sigma \in \cent{\sigma_c}{\Nfieldproj{S}}$, then $\sigma(S)$ is also an algebraic line-SIC, and $\Nfieldproj{\sigma(S)} = \Nfieldproj{S}$.
\end{prop}
\begin{proof}
Write the line-SIC and (projection) SIC as
\begin{align}
S &= \left\{\bv_1, \bv_2, \ldots, \bv_{d^2}\right\}, \\
\PPi(S) &= \left\{\tfrac{1}{d}\Pi_1, \tfrac{1}{d}\Pi_2, \ldots, \tfrac{1}{d}\Pi_{d^2}\right\},
\end{align}
where $\bv_i$ are unit vectors. By definition of $\fieldproj{S}$, the entries of all the $\Pi_i$ lie in $\fieldproj{S}$, so they lie in $\Nfieldproj{S}$. Also, using the commutativity of $\sigma$ and $\sigma_c$, 
\begin{equation}
\sigma(\Pi_i)
= \sigma(\bv_i\sigma_c(\bv_i)^\top)
= \sigma(\bv_i)\sigma(\sigma_c(\bv_i)^\top)
= \sigma(\bv_i)\sigma_c(\sigma(\bv_i)^\top) 
\end{equation}
is the Hermitian projection onto $\sigma(\bv_i)$, hence
\begin{equation}
\PPi(\sigma(S)) = \left\{\tfrac{1}{d}\sigma(\Pi_1), \tfrac{1}{d}\sigma(\Pi_2), \ldots, \tfrac{1}{d}\sigma(\Pi_{d^2})\right\}.
\end{equation}
Moreover,
\begin{align}
\Tr(\sigma(\Pi_i)\sigma(\Pi_j))
= \sigma(\Tr(\Pi_i\Pi_j))
= \sigma\!\left(\tfrac{\delta_{ij}d+1}{d+1}\right)
= \tfrac{\delta_{ij}d+1}{d+1},
\end{align}
certifying that $\PPi(\sigma(S))$ is a SIC. Finally, $\Nfieldproj{\sigma(S)} \subseteq \Nfieldproj{S}$, because the entries of $\Pi_i$ lie in $\Nfieldproj{S}$, and $\Nfieldproj{S}$ is Galois over $\Q$, hence $\sigma(\Nfieldproj{S}= \Nfieldproj{S}$. It follows by symmetry that $\Nfieldproj{S} = \Nfieldproj{\sigma^{-1}(\sigma(S))} \subseteq \Nfieldproj{\sigma(S)}$ and thus $\Nfieldproj{\sigma(S)} = \Nfieldproj{S}$.
\end{proof}

\subsection{Galois multiplets and super-equivalence classes of Weyl--Heisenberg SICs}\label{subsec:G-WHSIC}

We now assemble $\PEC(d)$-orbits of Weyl--Heisenberg SICs in dimension $d$ into \textit{multiplets} under the Galois action by the subgroup $\cent{\sigma_c}{\Nfieldproj{S}}$ of $\Gal(\Nfieldproj{S}/\Q)$ that commutes with complex conjugation. The coarser equivalence classes obtained under the combined geometric and Galois actions (i.e., unions of the orbits in a multiplet) will be called \textit{Galois super-equivalence classes}.

The conjecture of Appleby, Flammia, McConnell, and Yard \cite[Conj.~2]{appleby1} implies that many of the fields we have been studying are equivalent Galois extensions of $\Q$.

\begin{lem}\label{lem:conj11-consequence}
Assume \Cref{conj:AFMY20} holds. Then in every dimension $d \ge 4$, for each Weyl--Heisenberg line-SIC $S$, the field $\fieldvec{S}$ is a finite Galois extension of $\Q$, and 
\begin{equation}
\fieldvec{S} = \fieldproj{S} = \fieldeproj{S} = \Nfieldproj{S}.
\end{equation} 
\end{lem} 
\begin{proof}
By \Cref{conj:AFMY20}(2), $\fieldvec{S}$ is a Galois extension of $\Q$. Consequently, $\fieldvec{S}$ is preserved setwise by complex conjugation, so $\fieldvec{S}= \fieldproj{S}$ holds by \Cref{lem:normal-conditions}(1).

Also by \Cref{conj:AFMY20}(2), $\fieldvec{S}$ always contains the ray class field $H_{d'\infty_1\infty_2}$ of $\Q(\Delta_d)$. This class field must contain the class field of $\Q$ with conductor $d'\infty$, which is $\Q(\zeta_{d'}) = \Q(\xi_d)$. Therefore, $\fieldvec{S}$ contains $\fieldproj{S}(\xi_d) = \fieldeproj{S}$. Since $\fieldvec{S} \subseteq \fieldproj{S} \subseteq \fieldeproj{S}$, equality of those three fields follows. So $\fieldeproj{S}$ is Galois and hence equal to $\Nfieldproj{S}$.
\end{proof} 

The evidence for a Galois action on Weyl--Heisenberg SICs was extensively detailed by Appleby, Yadsan-Appleby, and Zauner \cite{AYZ13} in 2013. The group $\Gal(\overline{\Q}/\Q)$ acts on the elements of Weyl--Heisenberg group $\WH(d)$ coordinatewise on its matrices. Any automorphism $\sigma: \overline{\Q} \to \overline{\Q}$, sending a matrix $H \mapsto \sigma(H)$, acts as a homomorphism $\sigma(H_1H_2) = \sigma(H_1) \sigma(H_2)$. This Galois action fixes the group $\WH(d)$ as a set, with any automorphism permuting its elements, since every Galois conjugate $\sigma(H)$ of a group element $H$ is also an element of $\WH(d)$.

By \Cref{prop:215}(2), $\fieldeproj{S}$ is an invariant of the $\PEC(d)$ orbit of the Weyl--Heisenberg SIC $S$, hence so is its Galois closure $\Nfieldproj{S}$. Thus the field $\Nfieldproj{S}$ is invariant under both the geometric and the Galois actions.

The following definition differs slightly from some prior usage of the term ``Galois multiplet''; see \Cref{rmk:413}.

\begin{defn}\label{defn:gal-multiplet}
For $d \ge 4$, the {\em Galois multiplet} (or simply \textit{multiplet}) of a $\PEC(d)$-orbit $[S]$ of algebraic Weyl--Heisenberg line-SICs is
\begin{equation}
[[S]] = \{[\sigma(S)] : \sigma \in \cent{\sigma_c}{\Nfieldproj{S}}\}.
\end{equation}
We also define an associated equivalence relation $\approx_{\GPEC}$, termed \textit{Galois super-equivalence}, by 
\begin{equation}
[S_1] \approx_{\GPEC} [S_2] \iff [[S_1]] = [[S_2]],
\end{equation}
and term the union of the elements $[S_i]$ of $[[S]]$ the \textit{Galois super-equivalence class of $S$}.
\end{defn}

This definition implies that each $[[S]]$ is a disjoint union of $\PEC(d)$-orbits $[S]$ because $\cent{\sigma_c}{\Nfieldproj{S}}$ is a group; all SICs in $[S]$ are algebraic by \Cref{lem:alg-SIC-PEC}.

We summarize the properties noted above in the following proposition.
\begin{prop} \label{prop:413}
Assume \Cref{conj:AFMY20} holds. Then, for $d \ge 4$, the set $\WHSIC{d}$ is partitioned into a finite number of Galois super-equivalence classes under the equivalence relation $\approx_{\GPEC}$. Each Galois super-equivalence class is a finite union of geometric equivalence classes.
\end{prop}

Galois multiplets in a fixed dimension $d$ may have different sizes. For example, in the case $d=35$ illustrated in \Cref{fig:35a} and \Cref{fig:35b}, the ten $\PEC(35)$-orbits labeled 
\begin{equation}
\{35a, 35b, 35c, 35d, 35e, 35f, 35g, 35h, 35i, 35j\}
\end{equation}
form six Galois multiplets labeled 
\begin{equation}
\{35bcdg, 35af, 35e, 35h, 35i, 35j\}.
\end{equation}

\begin{rmk}\label{rmk:413}
\Cref{defn:gal-multiplet} defines a ``Galois multiplet'' as a potentially smaller set $[[S]]$ than that defined by same term in some previous work. The term ``Galois multiplet'' was defined in \cite[p.~1051]{appleby2} and \cite[Sec.~14.22]{waldron18} to refer to the set of all geometric equivalence classes $[S]$ having a fixed associated SIC field $\E(S)= \fieldeproj{S}$. In particular, the term ``minimal multiplet'' was used in \cite[Sec.~6]{appleby2} and \cite[Sec.~14.25]{waldron18} to refer to the set of all geometric equivalence classes $[S]$ having the minimal SIC field, which is the field $H_{d' \infty_1\infty_2}$ mentioned in \Cref{conj:AFMY20}(2). For the purpose of this remark, we refer to the set of all geometric equivalence class $[S]$ having a fixed $\fieldeproj{S}$ as a \textit{field multiplet}.
 
Assuming \Cref{conj:AFMY20}(2), all elements in a Galois multiplet $[[S]]$ have the same SIC field. However, it is a priori possible that two different Galois multiplets $[[S]]$ have the same associated SIC field. Thus, field multiplets are finite unions of Galois multiplets. Grassl \cite{grasslpersonal2} has found examples of field multiplets consisting of two distinct Galois multiplets. This phenomenon is studied in \Cref{sec:degeneration}, whose results imply that the a field multiplet is always a union of either one or two Galois multiplets. According to \Cref{thm:99}(c), cases of two Galois multiplets arise precisely for $d$ with $\Delta_d= f^2 \Delta_0$ having fundamental discriminant $\Delta_0 \equiv 1 \Mod{8}$.

For the rest of the paper the terms ``Galois multiplet,'' ``multiplet,'' and ``minimal multiplet'' follow \Cref{defn:gal-multiplet} without further comment.
\end{rmk}

\section{Conjectures on counting geometric equivalence classes and Galois multiplets of Weyl--Heisenberg SICs}\label{sec:5}

In this section, we formulate two conjectures relating the number of SICs in dimensions $d \ge 4$ to class numbers of certain real quadratic orders. Both these conjectures are numerically testable, and they are consequences of \Cref{conj:sic-to-order0}.

\subsection{Class groups and class monoids of orders}\label{subsec:51}

Let $K$ be a number field. An \textit{order} in $K$ is any subring $\OO$ including $1$ that is finitely additively generated and whose rational span $\Q\OO = K$. There is a unique \textit{maximal order} $\OO_K$ of $K$ that contains all the others; $\OO_K$ coincides with the set of algebraic integers in $K$. The maximal order is a Dedekind domain, while every non-maximal order is not a Dedekind domain. In particular, a non-maximal order will have non-invertible nonzero ideals. However, the set of \textit{invertible} ideals of a non-maximal order behaves much like set of nonzero ideals in a Dedekind domain.

Recall that a {\em fractional ideal} of a (commutative) integral domain $\OO$ is a finitely generated $\OO$-module $\mm$ contained in its quotient field $K$. A fractional ideal $\aa$ is {\em invertible} if there exists a fractional ideal $\bb$ with $\aa \bb = \OO$. Dedekind domains are characterized by the property that all nonzero fractional ideals are invertible. 

The class group of an order, defined in terms of invertible fractional ideals, generalizes the usual notion of the class group of (the maximal order of) a number field.

\begin{defn}\label{defn:classgroup}
The \textit{class group} (also called the \textit{Picard group} or \textit{ring class group}) of an order $\OO$ is
\begin{equation}
\Cl(\OO) := \frac{\{\mbox{invertible fractional ideals of } \OO\}}{\{\mbox{principal fractional ideals of } \OO\}}.
\end{equation}
\end{defn}

If one instead considers \textit{all} nonzero ideals, one obtains a monoid (that is, a semigroup with identity) rather than a group.

\begin{defn}\label{defn:classmonoid}
The \textit{class monoid} of an order $\OO$ is
\begin{equation}
\ClM(\OO) := \frac{\{\mbox{nonzero fractional ideals of } \OO\}}{\{\mbox{principal fractional ideals of } \OO\}}.
\end{equation}
\end{defn}

For a more in-depth discussion of the properties of non-maximal orders, see \cite{kopplagarias} and \cite[Ch.~I, Sec.~12]{Neukirch13}.

If $K$ is a quadratic field of discriminant $\Delta_0$, it has a unique suborder of discriminant $f^2\Delta_0$ for every positive integer $f$. This order is given explicitly as
\begin{equation}
\OO_{\!\Delta} 
:= \Z\!\left[\frac{\Delta+\sqrt{\Delta}}{2}\right] 
= \Z + \frac{\Delta+\sqrt{\Delta}}{2}\Z
= \Z + f\frac{\Delta_0+\sqrt{\Delta_0}}{2}\Z.
\end{equation}

A {\em fundamental discriminant} $\Delta_0$ is any nonsquare integer equal to the field discriminant of $\Q(\sqrt{\Delta_0})$. Equivalently, $\Delta_0$ is a fundamental discriminant if and only if $\Delta_0 \equiv 0,1 \Mod{4}$, $16\!\nmid\!\Delta_0$, and the odd part of $\Delta_0$ squarefree. Any discriminant $\Delta$ of a quadratic order $\OO$ has form $\Delta = f^2 \Delta_0$ for a unique $\Delta_0$ and $f \ge 1$.

For the SIC problem, we specifically consider the case when $\Delta = \Delta_d = (d+1)(d-3)$ for some integer $d$. In this case, the class monoid $\ClM(\OO_{\!\Delta_d})$ may be understood through bijections with other sets of number-theoretic interest.

\begin{thm}\label{thm:equiv}
Let $\Delta$ be a non-square discriminant. Write $\Delta = f^2\Delta_0$ for some $f \in \N$ and a fundamental discriminant $\Delta_0$. Then there are explicit, canonical bijections between the sets
\begin{itemize}
\item[(0)] $\ClM(\OO_\Delta)$;
\item[(1)] the disjoint union $\bigsqcup_{f'|f} \Cl(\OO_{(f')^2\Delta_0})$; and
\item[(2)] the set $\QQ(\Delta)/\GL_2(\Z)$, where $\QQ(\Delta)$ is the set of 
integral binary quadratic forms of discriminant $\Delta$, and $\rA = \smmattwo{s}{t}{u}{v} \in \GL_2(\Z)$ acts on the right by 
the determinant-twisted group action 
$Q^{\rA}(x,y) = (\det{\rA})\,Q(sx+ty,ux+vy)$.
\end{itemize}
Now suppose that $\Delta = \Delta_d = (d+1)(d-3)$ for a positive integer $d$. Then the sets (0), (1), and (2) also have explicit, canonical bijections with
\begin{itemize}
\item[(3)] the set $\CC_{d-1} = \{\{\rR^{-1}\rA\rR : \rR \in \GL_2(\Z)\} : \rA \in \SL_2(\Z), \Tr(\rA)=d-1\}$ of $\GL_2(\Z)$-conjugacy class of elements of $\SL_2(\Z)$ of trace $d-1$.
\end{itemize}
\end{thm}
\begin{proof}
The bijection between (0) and (1) holds because, for any quadratic discriminant $\Delta$ one has, by Halter-Koch \cite[Sec.~5.5]{HK:13}, the equality of sets $\ClM(\OO_\Delta) = \bigsqcup_{f'|f} \Cl(\OO_{(f')^2\Delta_0})$.

The bijection between (0) and (2) is a variant of Gauss composition and is proven as \cite[Thm.~1.2]{woodgauss}; for an introduction to Gauss composition, see \cite[Ch.~5]{cohencourse}. Note that quadratic forms $Q(x,y) = ax^2+bxy+cy^2$ with $\gcd(a,b,c) > 1$ are included in $\QQ(\Delta)$.

We now exhibit a bijection between (2) and (3) that specifically holds only for $\Delta=\Delta_d$. It is defined by a map $\phi: \QQ(\Delta) \to \GL_2(\Z)$ that induces the bijection. For a quadratic form $Q(x,y) = ax^2+bxy+cy^2$ of discriminant $\Delta_d = (d+1)(d-3)$, define $\phi(Q) := \smmattwo{\foh(b+d-1)}{-a}{c}{\foh(-b+d-1)}$. Note that 
\begin{equation}
\det(\phi(Q)) = \frac{1}{4}(-b^2+(d-1)^2)+ac = \frac{1}{4}(-\Delta+(d-1)^2) = 1
\end{equation}
and $\Tr(\phi(Q)) = d-1$. It is straightforward to check that $\phi(Q^\rR) = \rR^{-1}\phi(Q)\rR$ for any $\rR \in \GL_2(\Z)$. Thus, $\phi$ defines a map $\ol{\phi} : \QQ(\Delta)/\GL_2(\Z) \to \CC_{d-1}$. In the other direction, for $\rA = \smmattwo{s}{t}{u}{v} \in \SL_2(\Z)$, define the quadratic form 
$(\psi(\rA))(x,y) := -tx^2 + (s-v)xy + uy^2$. It is straightforward to check that $\psi(\rR^{-1}\rA\rR) = \psi(\rA)^\rR$ for $R \in \GL_2(\Z)$ and, if $\Tr(A)=d-1$, then $\disc(\psi(\rA)) = \Delta$. Thus $\psi$ defines a map $\ol{\psi} : \CC_{d-1} \to \QQ(\Delta)/\GL_2(\Z)$. A direct calculation shows $\ol\phi \circ \ol\psi = \id$ and $\ol\psi \circ \ol\phi = \id$.
\end{proof}

We obtain: 
\begin{cor}\label{cor:54}
Suppose $d$ is a positive integer, and let $\Delta = \Delta_d = (d+1)(d-3)$. Write $\Delta = f^2\Delta_0$ for some $f \in \N$ and a fundamental discriminant $\Delta_0$.
The quantity $\abs{\ClM(\OO_\Delta)}$ is equal to
\begin{itemize}
\item[(1)] the sum of (ring) class numbers $\sum_{f'|f} \hcn{(f')^2\Delta_0}$, where $\hcn{(f')^2\Delta_0} = \abs{\Cl(\OO_{(f')^2\Delta_0})}$;
\item[(2)] $\abs{\QQ(\Delta)/\GL_2(\Z)}$, the number of twisted $\GL_2(\Z)$-classes of binary quadratic forms of discriminant $\Delta$;
\item[(3)] $\abs{\CC_{d-1}}$, the number of $\GL_2(\Z)$-conjugacy class of elements of $\SL_2(\Z)$ of trace $d-1$.
\end{itemize}
\end{cor}

\subsection{Statement of the Geometric Count Conjecture}\label{subsec:52}

The following conjecture counts the number of geometric equivalence classes of Weyl--Heisenberg SICs in dimension $d$. It is implied by \Cref{conj:sic-to-order0}.

\begin{conj}[Geometric Count Conjecture]\label{conj:count}
Fix a positive integer $d \neq 3$, let $\Delta_d = (d+1)(d-3)$, and let $\OO_{\Delta_d}$ be the quadratic order of discriminant $\Delta_d$. Then, the number of geometric equivalence classes $[S]$ of Weyl--Heisenberg covariant SICs is
\begin{equation}
\abs{{\rm WHSIC}(d)/\PEC(d)} = \abs{\ClM(\OO_{\Delta_d})}.
\end{equation}
\end{conj}

\Cref{conj:count} immediately gives an upper bound for the total number of Weyl--Heisenberg line-SICs in dimensions $d \ge 4$.

\begin{prop}\label{prop:WHSIC-bound}
\Cref{conj:count} implies that the number $\abs{\WHSIC{d}}$ of Weyl--Heisenberg line-SICs in dimensions $d \ge 4$ satisfies
\begin{equation}\label{eqn:WH-count-number}
\abs{\WHSIC{d}} \le \frac{\abs{\PEC(d)}}{d^2} \abs{\Clm(\OO_{\Delta_d})}.
\end{equation}
\end{prop}

\begin{proof}
The number of distinct Weyl--Heisenberg line-SICs in a $\PEC(d)$-orbit $[S]$ is $\frac{\abs{\PEC(d)}}{\abs{\stab(S)}}$, where $\stab(S)$ is the subgroup of $\PEC(d)$ stabilizing the line-SIC $S \in [S]$. The cardinality $\abs{\stab(S)}$ is independent of the choice of $S \in [S]$. The Weyl--Heisenberg group $\WH(d)$ acts transitively on the set of $d^2$ complex lines $S$ in the $\WH(d)$-orbit of a fiducial vector. Thus, $\PWH(d)$ is a subgroup of $\stab(S)$, and so $d^2$ divides $\abs{\stab(S)}$. Hence we have $\frac{\abs{\PEC(d)}}{\abs{\stab(S)}}\le \frac{ \abs{\PEC(d)}}{d^2}$. The result now follows from \Cref{conj:count}.
\end{proof}

\begin{rmk}\label{rmk:whsic}
\Cref{table:classgroup-1CC} presents, in the column labeled $\abs{\WHSIC{d}}$, empirical data for the number of observed (numerical) SICs for $4 \le d \le 50$ from Scott and Grassl \cite[Table 1]{scott1}. Scott and Grassl also observe that all known numerical Weyl--Heisenberg SICs of dimensions $4 \le d \le 50$ have an additional stabilizer element of order $3$. If this statement were true in general, then we could strengthen \Cref{prop:WHSIC-bound} to conclude
\begin{equation}\label{eqn:WH-count-sharp}
\abs{\WHSIC{d}} \le \frac{1}{3} \frac{\abs{\PEC(d)}}{d^2} \abs{\Clm(\OO_{\Delta_d})}
\end{equation}
holds for all $d \ge 4$. Examination of \Cref{table:classgroup-1CC} shows for known numerical SICs that \eqref{eqn:WH-count-sharp} holds for $4\le d \le 50$ and that equality in \eqref{eqn:WH-count-sharp} holds for
\begin{equation}
d \in \{5, 13, 17, 23, 25, 27, 29, 37, 41, 43, 47, 49\}.
\end{equation}
\end{rmk}

\subsection{Verification of cases of the Geometric Count Conjecture}\label{sec:geover}

We have verified in Table \Cref{table:classgroup-1CC} that the two sets in the Geometric Count Conjecture have the same size for $4 \le d \leq 90$, under the assumption that the published lists of Weyl--Heisenberg SICs in these dimensions, given in Scott \cite{scott2}, are complete. (The lists build on earlier work of Scott and Grassl \cite{scott1}.)

\begin{empirp}
\label{prop:emp1}
For $\Delta_d = (d+1)(d-3)$, the total number of currently known $\PEC(d)$-orbits $[S]$ of numerical SICs, equals $\abs{\Clm(\OO_{\Delta_d})}$ for $4 \le d \le 50$ in the Scott and Grassl list \cite{scott1}, and for $51\le d \le 90$ in the Scott list \cite{scott2}.
\end{empirp}
\begin{proof}
For $4 \leq d \leq 90$, we computed $\abs{\ClM(\OO_{\Delta_d})}$ using Magma \cite{magma} by computing the equivalent quantity $\abs{\QQ(\Delta_d)/\GL_2(\Z)}$. 

Write $\Delta_d = f^2\Delta_0$ for a fundamental discriminant $\Delta_0 = (\Delta_d)_0$. By \Cref{thm:equiv},
\begin{equation}
\ClM(\OO_{\Delta_d}) = \bigsqcup_{f'|f} \Cl(\OO_{(f')^2\Delta_0}).
\end{equation}
For each $f'\div f$, we computed the set of reduced primitive binary quadratic forms of discriminant $(f')^2\Delta_0$ using the Magma function \texttt{ReducedForms()}, written by David Kohel. This function produces a list of $\SL_2(\Z)$-classes of primitive binary quadratic forms in which a twisted $\GL_2(\Z)$-class is always represented either by a single form $ax^2+bxy+cy^2$ with $a>0$ or a pair of forms $\{ax^2+bxy+cy^2,-ax^2+bxy-cy^2\}$. Removing forms with $a<0$ produced a list containing one representative for each primitive twisted $\GL_2(\Z)$-class. Multiplying by $f/f'$ and joining the lists gave a list of representatives of all (not necessarily primitive) classes in $\QQ(\Delta_d)/\GL_2(\Z)$. The counts of quantities for the different $f'$ are given are the unordered set of all ring class numbers $\hcn{(f')^2 \Delta_0} = \abs{\Cl(\OO_{(f')^2\Delta_0})}$. They are given in column labeled by $\hcn{(f')^2 \Delta_0}$ in \Cref{table:classgroup-1CC}. The total counts of class numbers are given in the column labeled by $\abs{\ClM(\OO_\Delta)}$.

{\em Numerical Weyl--Heisenberg SICs} are arrangements of $d^2$ complex lines that are the Weyl--Heisenberg orbit of a single numerical fiducial vector and that appears to have equal angles to $10$ decimal places. Numerical SICs are produced as output of an optimization algorithm. Empirical counts for the number of $\PEC(d)$-orbits of numerical SICs were given by Scott and Grassl for $4 \le d \le 50$ \cite[Tab.~1 and Tab.~2]{scott1} and by Scott \cite[pp.~3--6]{scott2} for $51 \le d \le 90$. The classes in the tables are denoted by dimension and letter (for example, $31a$). Tab.~1 of \cite{scott1} contains the subset of exact SICs, grouped into Galois multiplets, while Tab.~2 of \cite{scott1} gives a list of numerical SICs, not grouped into Galois multiplets. In both cases one can count the total number of $\PEC(d)$-orbits. 

We also found agreement for all dimensions $51 \le d \le 90$ of $\abs{\Clm(\OO_{\Delta_d})}$ given in \Cref{table:classgroup-2BB} with the total count of $\PEC(d)$-classes of SICs in each dimension listed in \cite[App.~C]{scott2}. The empirical search of Scott for Weyl--Heisenberg SICs for dimensions $51\le d \le 90$ was restricted to eigenspaces having a extra order $3$ symmetry of Zauner's type $F_z$, as well as those having an order $3$ symmetry of type $F_a$, as defined in \cite{scott2}. It would constitute an additional conjecture that all Weyl--Heisenberg SICs in dimension $d \ge 4$ have a fiducial vector in one of these eigenspaces. 
\end{proof}

\clearpage
%

\bgroup
\def\arraystretch{1.1}
\begin{scriptsize}
\begin{table}[ht!]
\centering
\begin{tabular}{|||c||c|c|c|c|c|c||}
\hline
$d$\vphantom{$\frac{\abs{\PEC(d)}}{d^2}\hcn{(f')^2\Delta_0}$} & $\Delta$ & $\Delta_0$ & $f'$ & $\sigma_0(f)$ & $\hcn{(f')^2\Delta_0}$ & $\abs{\Clm(\OO_{\Delta})}$ \\
\hline\hline
$1$ & $-4$ & $-4$ & $1$ & $1$ & $1$ & $1$ \\
\hline
$2$ & $-3$ & $-3$ & $1$ & $1$& $1$ & $1$ \\
\hline
$3$ & $0$ & $0$ & $\N$ & $\infty$ & $1$ & $\infty$ \\
\hline
$4$ & $5$ & $5$ & $1$ & $1$ & $1$ & $1$ \\
\hline
$5$ & $12$ & $12$ & $1$ & $1$ & $1$ & $1$ \\
\hline
$6$ & $21$ & $21$ & $1$ & $1$ & $1$ & $1$ \\
\hline
$7$ & $32$ & $8$ & $1,2$ & $2$ & $1,1$ & $2$ \\
\hline
$8$ & $45$ & $5$ & $1,3$ & $2$ & $1,1$ & $2$ \\
\hline
$9$ & $60$ & $60$ & $1$ & $1$ & $2$ & $2$ \\
\hline
$10$ & $77$ & $77$ & $1$ & $1$ & $1$ & $1$ \\
\hline
$11$ & $96$ & $24$ & $1,2$ & $2$ & $1,2$ & $3$ \\
\hline
$12$ & $117$ & $13$ & $1,3$ & $2$ & $1,1$ & $2$ \\
\hline
$13$& $140$ & $140$ & $1$ & $1$ & $2$ & $2$ \\
\hline
$14$ & $165$ & $165$ & $1$ & $1$ & $2$ & $2$ \\
\hline
$15$ & $192$ & $12$ & $1,2,4$ & $3$ & $1,1,2$ & $4$ \\
\hline
$16$ & $221$ & $221$ & $1$ & $1$ & $2$ & $2$ \\
\hline
$17$ & $252$ & $28$ & $1,3$ & $2$ & $1,2$ & $3$ \\
\hline
$18$ & $285$ & $285$ & $1$ & $1$ & $2$ & $2$ \\
\hline
$19$ & $320$ & $5$ & $1,2,4,8$ & $4$ & $1,1,1,2$ & $5$ \\
\hline
$20$& $357$ & $357$ & $1$ & $1$ & $2$ & $2$ \\
\hline
$21$ & $396$ & $44$ & $1,3$ & $2$ & $1,4$ & $5$ \\
\hline
$22$ & $437$ & $437$ & $1$ & $1$ & $1$ & $1$ \\
\hline
$23$ & $480$ & $120$ & $1,2$ & $2$ & $2,4$ & $6$ \\
\hline
$24$ & $525$ & $21$ & $1,5$ & $2$ & $1,2$ & $3$ \\
\hline
$25$ & $572$ & $572$ & $1$ & $1$ & $2$ & $2$ \\
\hline
$26$ & $621$ & $69$ & $1,3$ & $2$ & $1,3$ & $4$ \\
\hline
$27$ & $672$ & $168$ & $1,2$ & $2$ & $2,4$ & $6$ \\
\hline
$28$ & $725$ & $29$ & $1,5$ & $2$ & $1,2$ & $3$ \\
\hline
$29$ & $780$ & $780$ & $1$ & $1$ & $4$ & $4$ \\
\hline
$30$ & $837$ & $93$ & $1,3$ & $2$ & $1,3$ & $4$ \\
\hline
$31$ & $896$ & $56$ & $1,2,4$ & $3$ & $1,2,4$ & $7$ \\
\hline
$32$ & $957$ & $957$ & $1$ & $1$ & $2$ & $2$ \\
\hline
$33$ & $1020$ & $1020$ & $1$ & $1$ & $4$ & $4$ \\
\hline
$34$ & $1085$ & $1085$ & $1$ & $1$ & $2$ & $2$ \\
\hline
$35$ & $1152$ & $8$ & $1,2,3,4,6,12$ & $6$ & $1,1,1,1,2,4$ & $10$ \\
\hline
$36$ & $1221$ & $1221$ & $1$ & $1$ & $4$ & $4$ \\
\hline
$37$ & $1292$ & $1292$ & $1$ & $1$ & $4$ & $4$ \\
\hline
$38$ & $1365$ & $1365$ & $1$ & $1$ & $4$ & $4$ \\
\hline
$39$ & $1440$ & $40$ & $1,2,3,6$ & $4$ & $2,2,2,4$ & $10$ \\
\hline
$40$ & $1517$ & $1517$ & $1$ & $1$ & $2$ & $2$ \\
\hline
$41$ & $1596$ & $1596$ & $1$ & $1$ & $8$ & $8$ \\
\hline
$42$ & $1677$ & $1677$ & $1$ & $1$ & $4$ & $4$ \\
\hline
$43$ & $1760$ & $440$ & $1,2$ & $2$ & $2,4$ & $6$ \\
\hline
$44$ & $1845$ & $205$ & $1,3$ & $2$ & $2,4$ & $6$ \\
\hline
$45$ & $1932$ & $1932$ & $1$ & $1$ & $4$ & $4$ \\
\hline
$46$ & $2021$ & $2021$ & $1$ & $1$ & $3$ & $3$ \\
\hline
$47$ & $2112$ & $33$ & $1,2,4,8$ & $4$ & $1,1,2,4$ & $8$ \\
\hline
$48$ & $2205$ & $5$ & $1,3,7,21$ &$4$ & $1,1,1,4$ & $7$ \\
\hline
$49$ & $2300$ & $92$ & $1,5$ & $2$ & $1,6$ & $7$ \\
\hline
$50$ & $2397$ & $2397$ & $1$ & $1$ & $2$ & $2$ \\
\hline
\end{tabular}\!\!
\begin{tabular}{||c|c|c|c|||}
\hline
$d$\vphantom{$\frac{\abs{\PEC(d)}}{d^2}\hcn{(f')^2\Delta_0}$} & $\abs{\WHSIC{d}}$ & $\frac{\abs{\PEC(d)}}{d^2}$ & $\#$ \\
\hline\hline
$1$ &$1$ & $1$ & $1$ \\
\hline
$2$ & $2$ & $12$ & $1$ \\
\hline
$3$ & $\infty$ & $48$ & $\infty$ \\
\hline
$4$ & $16$ & $96$ & $1$ \\
\hline
$5$ & $80$ & $240$ & $1$ \\
\hline
$6$ & $96$ & $288$ & $1$ \\
\hline
$7$ & $336$ & $672$ & $2$ \\
\hline
$8$ & $320$ & $768$ & $2$ \\
\hline
$9$ & $864$ & $1296$ & $2$ \\
\hline
$10$ & $480$ & $1440$ & $1$ \\
\hline
$11$ & $2640$ & $2640$ & $3$ \\
\hline
$12$ & $1152$ & $2304$ & $2$ \\
\hline
$13$ & $2912$ & $4368$ & $2$ \\
\hline
$14$ & $2688$ & $4032$ & $2$ \\
\hline
$15$ & $6720$ & $5760$ & $4$ \\
\hline
$16$ & $4096$ & $6144$ & $2$ \\
\hline
$17$ & $9792$ & $9792$ & $3$ \\
\hline
$18$ & $5184$ & $7776$ & $2$ \\
\hline
$19$ & $16720$ & $13680$ & $5$ \\
\hline
$20$ & $7680$ & $11520$ & $2$ \\
\hline
$21$ & $26880$ & $16128$ & $5$ \\
\hline
$22$ & $5280$ & $15840$ & $1$ \\
\hline
$23$ & $48576$ & $24288$ & $6$ \\
\hline
$24$ & $15360$ & $18432$ & $3$ \\
\hline
$25$ & $20000$ & $30000$ & $2$ \\
\hline
$26$ & $34944$ & $26208$ & $4$ \\
\hline
$27$ & $69984$ & $34992$ & $6$ \\
\hline
$28$ & $26880$ & $32256$ & $3$ \\
\hline
$29$ & $64960$ & $48720$ & $4$ \\
\hline
$30$ & $46080$ & $34560$ & $4$ \\
\hline
$31$ & $138880$ & $59520$ & $7$ \\
\hline
$32$ & $32768$ & $49152$ & $2$ \\
\hline
$33$ & $84480$ & $63360$ & $4$ \\
\hline
$34$ & $39168$ & $58752$ & $2$ \\
\hline
$35$ & $235200$ & $80640$ & $10$ \\
\hline
$36$ & $82944$ & $62208$ & $4$ \\
\hline
$37$ & $134976$ & $101232$ & $4$ \\
\hline
$38$ & $109440$ & $82080$ & $4$ \\
\hline
$39$ & $314496$ & $104832$ & $10$ \\
\hline
$40$ & $61440$ & $92160$ & $2$ \\
\hline
$41$ & $367360$ & $137360$ & $8$ \\
\hline
$42$ & $129024$ & $96768$ & $4$ \\
\hline
$43$ & $317856$ & $158928$ & $6$ \\
\hline
$44$ & $253440$ & $126720$ & $6$ \\
\hline
$45$ & $207360$ & $155520$ & $4$ \\
\hline
$46$ & $145728$ & $145728$ & $3$ \\
\hline
$47$ & $553472$ & $207552$ & $8$ \\
\hline
$48$ & $276480$ & $147456$ & $7$ \\
\hline
$49$ & $537824$ & $230496$ & $7$ \\
\hline
$50$ & $120000$ & $180000$ & $2$ \\
\hline
\end{tabular}
\medskip
\caption{
Comparison for $1 \le d \leq 50$ of ring class numbers with counts of SICs. The three rightmost columns are from \cite[Tab.~1]{scott1}. They include empirical data for $\abs{\WHSIC{d}}$, and column $\#$ gives empirical data for $\abs{\WHSIC{d}/\PEC(d)}$. The left columns give $\Delta = \Delta_d = (d+1)(d-3) = f^2\Delta_0$ for a fund.~disc.~$\Delta_0$, $f'$ runs over divisors of $f$, $\hcn{(f')^2\Delta_0} = \abs{\Cl(\OO_{(f')^2\Delta_0})}$, and $\sigma_0(f) = \abs{\{\OO : \OO_{\Delta} \subseteq \OO \subseteq{\OO_{\Delta_0}}\}}$. 
}
\label{table:classgroup-1CC}
\end{table}
\end{scriptsize}
\egroup

\clearpage
%

\bgroup
\def\arraystretch{1.1}
\begin{scriptsize}
\begin{table}[ht!]
\centering
\begin{tabular}{||c||c|c|c|c|c|c|}
\hline
$d$ & $\Delta$ & $\Delta_0$ & $f'$ & $\sigma_0(f)$ & $\hcn{(f')^2\Delta_0}$ & $\abs{\Clm(\OO_{\Delta})}$ \\
\hline\hline
$51$ & $2496$ & $156$ & $1,2,4$ & $3$ & $2,4,8$ & $14$ \\
\hline
$52$ & $2597$ & $53$ & $1,7$ & $2$ & $1,3$ & $4$ \\
\hline
$53$ & $2700$ & $12$ & $1,3,5,15$ & $4$ & $1,1,2,6$ & $10$ \\
\hline
$54$ & $2805$ & $2805$ & $1$ & $1$ & $4$ & $4$ \\
\hline
$55$ & $2912$ & $728$ & $1,2$ & $2$ & $2,4$ & $6$ \\
\hline
$56$ & $3021$ &$3021$ & $1$ & $1$ & $6$ & $6$ \\
\hline
$57$ & $3132$ & $348$ & $1,3$ & $1$ & $2,6$ & $8$ \\
\hline
$58$ & $3245$ & $3245$ & $1$ & $1$ & $4$ & $4$ \\
\hline
$59$ & $3360$& $840$ & $1,2$ & $2$ & $4,8$ & $12$ \\
\hline
$60$ & $3477$ & $3477$ & $1$ & $1$ & $4$ & $4$ \\
\hline
$61$ & $3596$ & $3596$ & $1$ & $1$ & $6$ & $6$ \\
\hline
$62$ & $3717$ & $413$ & $1,3$ & $2$ & $1,4$ & $5$ \\
\hline
$63$ & $3840$ & $60$ & $1,2,4,8$ & $4$ & $2,2,4,8$ & $16$ \\
\hline
$64$ & $3965$ & $3965$ & $1$ & $1$ & $4$ & $4$ \\
\hline
$65$ & $4092$ & $4092$ & $1$ & $1$ & $8$ & $8$ \\
\hline
$66$ & $4221$ & $469$ & $1,3$ & $2$ & $3,6$ & $9$ \\
\hline
$67$ & $4352$ & $17$ & $1,2,4,8,16$ & $5$ & $1,1,1,2,4$ & $9$ \\
\hline
$68$ & $4485$ & $4485$ & $1$ & $1$ & $4$ & $4$ \\
\hline
$69$ & $4620$ & $4620$ & $1$ & $1$ & $8$ & $8$ \\
\hline
$70$ & $4757$ & $4757$ & $1$ & $1$ & $5$ & $5$ \\
\hline
$71$ & $4896$ & $136$ & $1,2,3,6$ & $4$ & $2,4,4,8$ & $18$ \\
\hline
$72$ & $5037$ & $5037$ & $1$ & $1$ & $4$ & $4$ \\
\hline
$73$ & $5180$ & $5180$ & $1$ & $1$ & $4$ & $4$ \\
\hline
$74$ & $5325$ & $213$ & $1,5$ & $2$ & $1,6$ & $7$ \\
\hline
$75$ & $5472$ & $152$ & $1,2,3,6$ & $4$ & $1,2,4,8$ & $15$ \\
\hline
$76$ & $5621$ & $5621$ & $1$ & $1$ & $6$ & $6$ \\
\hline
$77$ & $5772$ & $5772$ & $1$ & $1$ & $8$ & $8$ \\
\hline
$78$ & $5925$ & $237$ & $1,5$ & $2$ & $1,6$ & $7$ \\
\hline
$79$ & $6080$ & $380$ & $1,2,4$ & $3$ & $2,4,8$ & $14$ \\
\hline
$80$ & $6237$ & $77$ & $1,3,9$ & $3$ & $1,2,6$ & $9$ \\
\hline
$81$ & $6396$ & $6396$ & $1$ & $1$ & $12$ & $12$ \\
\hline
$82$ & $6557$ & $6557$ & $1$ & $1$ & $3$ & $3$ \\
\hline
$83$ & $6721$ & $105$ & $1,2,4,8$ & $4$ & $2,2,4,8$ & $16$ \\
\hline
$84$ & $6885$ & $85$ & $1,3,9$ & $3$ & $2,2,6$ & $10$ \\
\hline
$85$ & $7052$ & $7052$ & $1$ & $1$ & $4$ & $4$ \\
\hline
$86$ & $7221$ & $7221$ & $1$ & $1$ & $10$ & $10$ \\
\hline
$87$ & $7392$ & $1848$ & $1,2$ & $2$ & $4,8$ & $12$ \\
\hline
$88$ & $7565$ & $7565$ & $1$ & $1$ & $4$ & $4$ \\
\hline
$89$ & $7740$ & $860$ & $1,3$ & $2$ & $2,8$ & $10$ \\
\hline
$90$ & $7917$ & $7917$ & $1$ & $1$ & $4$ & $4$ \\
\hline
\end{tabular}
\medskip
\caption{
Ring class numbers and ring class monoid numbers for $51 \le d \leq 90$. Here $\Delta = \Delta_d = (d+1)(d-3) = f^2\Delta_0$ for a fund.~disc.~$\Delta_0$, $f'$ runs over divisors of $f$, $\hcn{(f')^2\Delta_0} = \abs{\Cl(\OO_{(f')^2\Delta_0})}$, and $\widetilde{\#}$ is $\abs{\{ \OO : \OO_{\Delta} \subseteq \OO \subseteq{\OO_{\Delta_0}}\}} = \sigma_0(f)$. (The values $\abs{\Clm(\OO_{\Delta})}$ agree with empirical data for $\abs{\WHSIC{d}/\PEC(d)}$ of Scott \cite{scott2}).}
\label{table:classgroup-2BB}
\end{table}
\end{scriptsize}
\egroup

\subsection{Statement of the Multiplet Count Conjecture}\label{subsec:55}

We showed in \Cref{prop:413} that \Cref{conj:AFMY20} implies that there are a finite number of well-defined Galois multiplets of WHSICs in dimensions $d \ge 4 $, each consisting of a finite number of geometric equivalence classes; we now refine this prediction to a precise count in every dimension. 

\begin{conj}[Multiplet Count Conjecture]\label{conj:supercount}
Fix a positive integer $d \neq 3$. Let $\Delta = \Delta_d = (d+1)(d-3)$ and $K = \Q(\sqrt{\Delta})$. The number of Galois multiplets of Weyl--Heisenberg SICs in dimension $d$ equals the number of quadratic orders $\OO'$ with $\OO_{\!\Delta} \subseteq \OO' \subseteq \OO_K$. That is,
\begin{equation}
\abs{{\rm WHSIC}(d)/\approx_{GPEC}} = \abs{\{\OO' : \, \OO_{\!\Delta} \subseteq \OO' \subseteq \OO_K\}}.
\end{equation}
\end{conj} 

The content of \Cref{conj:supercount} is that the multiplets of Weyl--Heisenberg SICs are in one-to-one correspondence with orders between $\OO_\Delta$ and $\OO_K$. 

We note that if $\Delta = f^2 \Delta_0$ for a fundamental discriminant $\Delta_0$, with $f>0$, then
\begin{equation}\label{eqn:over-order-number}
s(\OO) := \abs{\{\OO' : \, \OO_{\!\Delta} \subseteq \OO' \subseteq \OO_K\}} = \sigma_0(f),
\end{equation}
where the arithmetic function $\sigma_0(f)$ denotes the number of positive divisors of $f$. (The function $\sigma_0(f)$ is also denoted $\tau(f)$ or $d(f)$ by some authors.) The formula \eqref{eqn:over-order-number} is valid for all quadratic orders.

\subsection{Verification of cases of the Multiplet Count Conjecture}\label{subsec:56}

We have verified that the two sets in the Multiplet Count Conjecture (\Cref{conj:supercount}) have the same size for dimensions $d \leq 90$ ($d \neq 3$) under the assumption that the published lists of Weyl--Heisenberg SICs in these dimensions are complete.

\begin{empirp}
\label{prop:emp2}
Consider the finite list of dimensions
\begin{equation}\label{eq:emp2dlist}
\mathcal{D} = \{ 4, 5, 6,7, 8, 9, 10, 11, 12, 23, 14, 15, 16, 17, 18, 19, 20, 21, 24, 28, 30, 35, 39, 48\}.
\end{equation}
\begin{itemize}
\item[(1)]
The total number of currently known constructed Galois multiplets $[[S]]$ of exact SICs for dimensions $d \in \mathcal{D}$ equals the total number of orders
\begin{equation}
s(\OO_{\Delta_d}) := \abs{\{ \OO' : \OO_{\Delta_d} \subseteq \OO' \subseteq \OO_{\Delta_0} \}},
\end{equation}
where $\Delta = \Delta_d = (d+1)(d-3)$. Equivalently, writing $\Delta = f^2 \Delta_0$, where $\Delta_0$ is a fundamental discriminant, this number $s(\OO_{\Delta_d}) = \sigma_0(f)$, the number of positive divisors of $f$.
\item[(2)]
For the dimensions $d \in \mathcal{D}$ and the multiplets $[[S]]$ above, there is at least one bijective map 
\begin{equation}
\MM: \{ \OO' : \OO_{\Delta} \subseteq \OO' \subseteq \OO_{\Delta_0} \} \to \{ [[S]]: S \, \mbox{a Weyl--Heisenberg line-SIC in} \, \C^d\}
\end{equation}
such that the number of geometric equivalence classes $[S] \in \MM(\OO_{(f')^2 \Delta_0})$ is $\abs{\Cl(\OO_{(f')^2 \Delta_0})}$.
\end{itemize}
\end{empirp}
\begin{proof}
The result follows from comparison of \cite[Tab.~1 and Tab.~2]{acfw} on known exact SICs with the number theoretic data on $\hcn{(f')^2 \Delta_0}$ from our \Cref{table:classgroup-1CC}. Here \cite[Tab.~1]{acfw} gives results of Scott and Grassl \cite{scott1} for exact multiplets for $4 \le d \le 16$, excluding $d=15$; \cite[Tab.~2]{acfw} gives the remaining exact multiplets in \eqref{eq:emp2dlist}. Each Galois multiplet contains a number of $\PEC(d)$-orbits, indicated by letters; for example, the multiplet $35bcdg$ consists of four $\PEC(35)$-orbits. 

\textit{(1).} In the dimensions in \eqref{eq:emp2dlist}, the number of Galois multiplets is found to agree with the data for $s(\OO_{\Delta_d}) = \sigma_0(f)$, given in column $\sigma_0(f)$ of \Cref{table:classgroup-1CC}.

\textit{(2).} The existence of a bijective map $\MM$ in (2) is equivalent to the assertion that the unordered multiset of cardinalities $\abs{\left\{[S] \in [[S]]\right\}}$ of $\PEC(d)$-orbits of SICs in each Galois multiplet is equal to the unordered multiset of ring class numbers $ \hcn{(f')^2\Delta_0}$ such that $f' \div f$. The multiset equality was verified by comparing the exact multiplet data of \cite{acfw} with that in column labeled $\hcn{(f')^2 \Delta_0}$ in \Cref{table:classgroup-1CC}.
\end{proof}

\begin{rmk}\label{rmk:511}
The bijective map $\MM$ is not uniquely specified using these class number counts alone. However, when augmented by requiring agreement of all inclusion relations between the associated SIC fields $\fieldproj{S}$, the truth of \Cref{conj:14} in dimension $d$ uniquely determines the map $\MM$ for the dimensions $d$ in \eqref{eq:emp2dlist}. (See \Cref{subsec:85a}; the smallest dimension where $\MM$ is not determined by the inclusion relations is $d=47$.)
\end{rmk}

\begin{egz}\label{example:59}
Let $d=35$. Here $\Delta_d= 36\cdot 32= 2^7 \cdot 3^2$. Here $\Delta_0=8$ and $f=12$. The set of divisors $f'$ of $12$ is $\{ 1, 2, 3, 4, 6, 12\}$.

We exhibit a bijection $\MM$ of \Cref{conj:sic-to-order0} from orders to multiplets in \Cref{fig:35-order-multiplet-corresp}. This bijection satisfies part (1) of \Cref{conj:sic-to-order0} for $d=35$. This choice of $\MM$ also satisfies part (2) of \Cref{conj:sic-to-order0} for $d=35$, as shown in \Cref{example:67}.
 
\begin{figure}[htb]\label{fig:35a}
\begin{tikzpicture}
[auto, nd/.style={circle,minimum size = 13 mm,draw}]
\matrix[row sep=2mm,column sep=8mm] {
& & \node (35bcdg) [nd] {${\OO_{1152}}\atop{h=4}$};
\\
& \node (35af) [nd] {${\OO_{288}}\atop{h=2}$}; &
\\
\node (35e) [nd] {${\OO_{72}}\atop{h=1}$}; 
& & & \node (35h) [nd] {${\OO_{128}}\atop{h=1}$};
\\
& & \node (35i) [nd] {${\OO_{32}}\atop{h=1}$}; &
\\
& \node (35j) [nd] {${\OO_{8}}\atop{h=1}$}; & &
\\
};
\draw[->] (35j) to node[swap] {4} (35i);
\draw[->] (35i) to node[swap] {4} (35h);
\draw[->] (35j) to node{9} (35e);
\draw[->] (35i) to node[swap] {9} (35af);
\draw[->] (35e) to node{4} (35af);
\draw[->] (35af) to node {4} (35bcdg);
\draw[->] (35h) to node[swap] {9} (35bcdg);
\end{tikzpicture}
\begin{tikzpicture}
[auto, nd/.style={circle,minimum size = 13 mm,draw}]
\matrix[row sep=2mm,column sep=8mm] {
& & \node (35bcdg) [nd] {$35bcdg$};
\\
& \node (35af) [nd] {$35af$}; &
\\
\node (35e) [nd] {$35e$}; & & & \node (35h) [nd] {$35h$};
\\
& & \node (35i) [nd] {$35i$}; &
\\
& \node (35j) [nd] {$35j$}; & &
\\
};
\draw[->] (35j) to node[swap]{} (35i);
\draw[->] (35i) to node[swap]{} (35h);
\draw[->] (35j) to node{} (35e);
\draw[->] (35i) to node[swap]{} (35af);
\draw[->] (35e) to node{} (35af);
\draw[->] (35af) to node{} (35bcdg);
\draw[->] (35h) to node[swap]{} (35bcdg);
\end{tikzpicture}
\caption{\label{fig:35-order-multiplet-corresp} Left side: For $d=35$, over-orders of $\OO_{1152}$, and their ring class numbers $h=\hcn{\Delta_0(f')^2}= \abs{\Cl(\OO_{(f')^2 \Delta_0})}$, with $\Delta_0=8$. The direction on edges indicate reverse set inclusions, with edge labels giving relative indices of orders. Thus $\OO_{1152} \subset \OO_8$ has index $144$ in $\OO_8$. Right side: For $d=35$, $\MM$-correspondence to Galois multiplets of $d=35$ SICs, with multiplet letter labels ($35a$ to $35j$) corresponding to $\PEC(35)$-orbits of SICs. Class numbers on the left match the number of $\PEC(35)$-orbit letter labels in each Galois multiplet.}
\end{figure}
\end{egz}

\section{Conjectures on precise fields of definition of Weyl--Heisenberg SICs}\label{sec:field}

In this section, we study the number-theoretic properties of the ray class fields of orders $\OO'$ predicted to be fields of definition of Weyl--Heisenberg SICs by \Cref{conj:14}. We give formulas for the degrees of those ray class field over $K$, by means of studying the associated ray class groups. In \Cref{prop:66}, we test numerical predictions of relative degrees and field inclusions implied by \Cref{conj:sic-to-order0} and \Cref{conj:14}. We also study the relative degrees of field inclusions predicted in \Cref{conj:sic-to-order0}(2), noting there exist cases where the map $\MM$ is not uniquely determined by the conclusions of \Cref{conj:sic-to-order0} and \Cref{conj:14}. Before studying the special ray class fields conjecturally associated to SICs, we summarize some properties of ray class fields of orders in general.

\subsection{Ray class fields of orders}\label{sec:81}

The fields attached to a Weyl--Heisenberg SIC are believed to be abelian extensions of $K = \Q(\sqrt{\Delta_d})$. Class field theory gives an abstract description of abelian extensions of $K$, viewed with respect to the ideals of the maximal order $\OO_K$. \Cref{conj:sic-to-order0} implies that a Weyl--Heisenberg SIC can unambiguously be assigned an order $\OO$ of $K$ with $\Z[\e_d] \subseteq \OO \subseteq \OO_K$. There is a long history of work which supplies a notion of class field theory attached to orders of number fields. In particular, the \textit{ring class field} of an order $\OO$ has Galois group over $K$ isomorphic to $\Cl(\OO)$ and generalizes the \textit{Hilbert class field}, which is the maximal unramified extension of $K$ and has class group $\Cl(\OO_K)$.
 
In \cite{kopplagarias}, the authors modify the standard constructions of class field theory to describe a distinguished collection of \textit{ray class fields of orders} $H^{\OO}_{\mm, \Sigma}$, each specified by a \textit{level datum} $(\OO; \mm, \Sigma)$, where $\OO$ is a (possibly non-maximal) order of $K$, $\mm$ an integral $\OO$-ideal, and $\Sigma$ a subset of real places of $K$. In the case that $K$ is real quadratic, there are two infinite places, so $\Sigma$ is one of $\emptyset$, $\{ \infty_1\}$, $\{\infty_2\}$, or $\{ \infty_1, \infty_2\}$. For the datum $(\OO; \OO, \emptyset)$, the associated field $H^{\OO}_{\OO, \emptyset}$ is the ring class field associated to the order. For the datum $(\OO_K; \mm, \rS)$, the associated field $H^{\OO_K}_{\mm,\rS}$ is the standard (Takagi) ray class field $H_{\mm,\rS}$. For a fixed order $\OO$, any abelian extension of $K$ is a subfield of some $H^{\OO}_{\mm,\rS}$.

We recall \cite[Thm.~1.1]{kopplagarias}:
\begin{thm}\label{thm:main1-KL} 
Let $K$ be a number field and $(\OO; \mm, \rS)$ a level datum for $K$. There exists a unique abelian Galois extension $H_{\mm,\rS}^{\OO}/K$ with the property that a prime ideal $\pp$ of $\OO_K$ that is coprime to the quotient ideal $\colonideal{\mm}{\OO_K}$ splits completely in $H_{\mm,\rS}^{\OO}/K$ if and only if $\pp \cap \OO = \pi\OO$, a principal prime $\OO$-ideal having $\pi \in \OO$ with $\pi \equiv 1 \Mod{\mm}$ and $\rho(\pi)>0$ for $\rho \in \rS$.
\end{thm}

To understand \Cref{thm:main1-KL}, recall that for any integral ideal $\mm$ of $\OO$, the quotient ideal $\colonideal{\mm}{\OO_K}$ is the largest ideal of $\OO_K$ contained in $\mm$. An important invariant of an order $\OO$ of an algebraic number field $K$ is its (absolute) conductor $\ff = \ff(\OO) = \colonideal{\OO}{\OO_K}$, which is the (set-theoretically) largest integral ideal $\ff \subseteq \OO$ of $\OO$ that is also an ideal of the maximal order $\OO_K$. The conductor ideal $\ff (\OO)$ encodes information on all the non-invertible ideals in the order $\OO$. The quotient ideal $\colonideal{\mm}{\OO_K}$ always satisfies the inclusions (as $\OO_K$-ideals)
\begin{equation}\label{eq:112} 
\ff(\OO) \mm \subseteq \colonideal{\mm}{\OO_K} \subseteq \ff(\OO) \cap \mm\OO_K;
\end{equation} 
in particular, $\colonideal{\mm}{\OO_K} \subseteq \ff(\OO)$. The three $\OO_K$-ideals in \eqref{eq:112} all have the same prime divisors. This holds because a prime ideal $\pp$ of $\OO_K$ such that $\pp \supseteq \ff(\OO) \mm = \ff(\OO) \mm\OO_K $ satisfies either $\pp \supseteq \ff(\OO)$ or $\pp \supseteq \mm\OO_K$, and thus, $\pp \supseteq \ff(\OO) \cap \mm\OO_K$. 

There is a Galois correspondence between ray class groups of an order and ray class fields of an order. The \textit{ray class group of an order} $\Cl_{\mm,\rS}(\OO)$ for level datum $(\OO; \mm, \rS)$ is defined as a quotient of certain groups of invertible fractional ideals of the order $\OO$. The ray class field $H_{\mm,\rS}^{\OO}$ of an order $\OO$ is associated to an appropriate $\Cl_{\mm,\rS}(\OO)$ in such a way that $\Gal(H_{\mm,\rS}^{\OO}/K) \isom \Cl_{\mm,\rS}(\OO)$ as abelian groups. Moreover, if $(\OO;\mm,\rS)$ and $(\OO';\mm',\rS')$ are level data satisfying $\OO \subseteq \OO'$, $\mm\OO' \subseteq \mm'$, and $\rS \supseteq \rS'$, then there is a natural quotient map $\Cl_{\mm,\rS}(\OO) \surj \Cl_{\mm',\rS'}(\OO')$ and a corresponding natural field inclusion $H_{\mm',\rS'}^{\OO'} \subseteq H_{\mm,\rS}^{\OO}$.

We recall \cite[Thm.~1.3]{kopplagarias}:
\begin{thm}\label{thm:main3-KL} 
For an order $\OO$ in a number field $K$ and any level datum $(\OO; \mm, \rS)$ with associated ray class field $H_0:=H_{\mm,\rS}^{\OO}$, there is an isomorphism $\Art_{\OO} : \Cl_{\mm, \rS}(\OO) \to \Gal(H_0/K)$, uniquely determined by its behavior on prime $\OO$-ideals $\pp$ coprime to $\ff(\OO) \cap \mm$, having the property that
\begin{equation}\label{eq:artinlocal}
\Art_{\OO}([\pp])(\alpha) \equiv \alpha^q \Mod{\mathfrak{P}},
\end{equation}
where $\mathfrak{P}$ is any prime of $\OO_{H_0}$ lying over $\pp\OO_K$, and $q=p^j$ is the number of elements in the finite field $\OO/\pp$. For any (not necessarily prime) ideal $\aa$ of $\OO$ coprime to $\ff(\OO) \cap \mm$,
\begin{equation}\label{eqn:artin-3}
\Art_{\OO}([\aa]) = \left.\Art([\aa\OO_K])\right|_{H_0},
\end{equation}
where $\Art : \Cl_{\colonideal{\mm}{\OO_K},\rS}(\OO_K) \to \Gal(H_1/K)$ is the usual Artin map in class field theory (as defined by \eqref{eq:artinlocal} in the special case $\OO=\OO_K$), with $H_1 = H_{\colonideal{\mm}{\OO_K},\rS}^{\OO_K}$ being a (Takagi) ray class field for the maximal order $\OO_K$, and $H_0 \subseteq H_1$.
\end{thm}

In \Cref{thm:main3-KL}, the set of prime ideals coprime to $\ff(\OO) \cap \mm$ includes all but finitely many of the prime ideals of $\OO$. The Artin map is fully defined by \eqref{eq:artinlocal}, which specifies the map on prime ideals of $\OO$ coprime to $\ff(\OO)\cap\mm$, because any ray ideal class in $\Cl_{\mm,\rS}(\OO)$ has a representative that is coprime to $\ff(\OO)\cap\mm$ (by \cite[Lem.~4.12]{kopplagarias}). Ideals of $\OO$ that are coprime to $\ff$ enjoy unique factorization into prime ideals (even though arbitrary nonzero ideals of $\OO$ do not). When $\OO=\OO_K$, \eqref{eq:artinlocal} specializes to the standard definition of the Artin map in the case of a Takagi ray class field.

\subsection{Special ray class fields of real quadratic orders}\label{subsec:62}

We recall the fact that orders of a quadratic field $K$ are completely specified by their discriminants $\Delta' = (f')^2\Delta_0$, where $\Delta_0$ is the field discriminant of $K$ (a fundamental discriminant; see the beginning of \Cref{subsec:55}) and $f'$ is a positive integer. 
\begin{defn}
\label{defn:Srcfo}
Fix a positive integer $d \geq 4$, write $\Delta = \Delta_d = (d+1)(d-3) = f^2\Delta_0$ with $\Delta_0$ a fundamental discriminant, and let $f'$ be a positive divisor of $f$. Let $\Delta' = (f')^2\Delta_0$ and $\OO' = \OO_{\Delta'}$. We introduce two families of ray class fields of the order $\OO'$ parametrized by $d$ and $f'$:
\begin{align}
E_{d,f'} &= H_{d\OO',\{\infty_1,\infty_2\}}^{\OO'} \quad \mbox{and} \quad
\tE_{d,f'} := H_{d'\OO',\{\infty_1,\infty_2\}}^{\OO'}.
\end{align}
Here $d'=d$ if $d$ is odd and $d'=2d$ if $d$ is even.
\end{defn}

The definitions yield $E_{d, f'} = \tE_{d, f'}$ when $d$ is odd, and the natural field inclusion $E_{d, f'} \subseteq \tE_{d', f'}$ when $d$ is even. We have $[\tE_{d, f'} : E_{d, f'}] = 2$ when $d$ is even, using \Cref{thm:rclsize}(2) below.

The larger fields $\tE_{d,f'}$ appear in \Cref{conj:14}, which asserts $\tE_{d, f'} =\fieldproj{S}$ whenever $\MM(\OO') = [[S]]$. \Cref{conj:14} is a natural extension of \Cref{conj:AFMY20}. \Cref{conj:AFMY20}(1) is equivalent to the statement that $\fieldproj{S} = \tE_{d,1}$ for any $S$ in the minimal multiplet. In \Cref{sec:A1} the smaller fields $E_{d, f'}$ are conjecturally associated to $\fieldtrip{S}$ for $[[S]]$, the multiplet associated to $\OO'$ by $\MM$.

We wish to compute the degrees of these fields. \Cref{thm:main3-KL} implies that the degree of the field $\tilde{E}_{d,f'}$ (respectively, $E_{d,f'}$) over $K= \Q(\sqrt{\Delta})$ equals the ray class number $\abs{\Cl_{d'\OO', \{\infty_1, \infty_2\}} (\OO')}$ (respectively, $\abs{\Cl_{d\OO', \{\infty_1, \infty_2\}} (\OO')}$). 

For our numerical computations, we calculated these ray class numbers using formulas from the ray class theory for orders developed in \cite[Thm.~5.6]{kopplagarias}. The ray class numbers calculation uses special features of the unit in the real quadratic order $\OO_{\Delta_d}$, $ d \ge 4$, obeying a congruence $\Mod{d}$. The formulas and the unit congruence are given in \Cref{sec:C}, in \Cref{thm:rclsize}. 

\subsection{Evidence for the Ray Class Fields of Orders Conjecture}\label{subsec:63}

Assuming the truth of \Cref{conj:14}, it follows from \Cref{thm:main3-KL} that the degree of $\fieldproj{S}$ over $K=\Q(\sqrt{(d+1)(d-3)})$ equals the class number $\abs{\Cl_{d\OO', \{\infty_1, \infty_2\}} (\OO')}$. We verify this numerical requirement in several cases. 

\begin{empirp}\label{prop:66}
For each dimension $d$ with $4 \le d \le 15$ (and $d=35$ under unpublished calculations of Appleby and Flammia), there exists a bijective map $\MM$ from orders to the (empirical) currently known constructed Galois multiplets $[[S]]$. Taking $\Delta = \Delta_d = (d+1)(d-3)= f^2 \Delta_0$ and $K= \Q(\sqrt{\Delta})$, the map $\MM$ may be chosen to have the property that, for any $f' \div f$ and constructed SIC $S$ such that $[[S]] = \MM( \OO_{ ((f')^2 \Delta_0})$, we have
\begin{equation}
[\fieldproj{S}: K] = [\tE_{d,f'}: K].
\end{equation} 
\end{empirp}
\begin{proof}
The numerical equalities are tabulated for $4 \le d \le 15$ and for $d=35$ in \Cref{table:classfield-1a} below. By \Cref{thm:main3-KL}, the degree of $\tE_{d,f'}$ over $K=\Q(\sqrt{\Delta_d})$ equals the class number $\abs{\Cl_{d\OO', \{\infty_1, \infty_2\}}(\OO')}$. We computed the relevant class numbers using the right side of \eqref{eq:rclsize}. 

The empirical calculations of field degrees $[\fieldproj{S}: K]$ are based on data for these fields given in \cite[Tab.~III]{scott1} and in \cite[Tab.~A7]{acfw}. The calculations of field degrees $[\fieldproj{S}: K]$ for multiplets $[[S]]$ of exact SICs for $d=15$ and $d=35$ can be extracted from \cite[Tab.~ A7]{acfw}. There $\fieldproj{S}$ corresponds to $\mathbb{E}_1$; the degree for $\mathbb{E}$ given in the table \cite[Tab.~ A7]{acfw} must be divided by $4$, because the degree of $\mathbb{E}$ is given over $\Q$, and because $\mathbb{E}= \mathbb{E}_1(i)$, an extension of degree $2$ since $\mathbb{E}_1$ does not contain $i$ for odd $d \ge 4$. For inclusions of fields for $d=35$, we made use of additional calculations \cite[unpublished]{applebyflammiapersonal}, which determined the inclusions between the fields $\fieldproj{S}$ in the six multiplets $[[S]]$ of exact SICs for $d=35$, correcting the picture in \cite[Fig.~1]{acfw}, by showing the field associated to multiplet $35e$ is contained in the field associated to multiplet $35af$.

The inclusions of fields $\tE_{d, f'}$ coincide exactly with the reverse ordering of the lattice of divisibility relations of the set $\{f' : f' \ge 1, \,f' \div f\}.$ On comparison we find there exists a unique bijection $\MM$ for all the dimensions in \Cref{table:classfield-1a} compatible with the field inclusions. In \Cref{table:classfield-1a} the orbit labels in the ``multiplet'' column are based on Flammia's online table \cite{flammiaweb}, following \cite{scott1,scott2,acfw}, specifically using the labels for exact SICs, which differ in some cases from Scott and Grassl's labels for numerical SICs \cite{scott1}. The correspondence between the columns labeled $f'$ and ``multiplet'' in \Cref{table:classfield-1a} specifies the map $\MM$.
\end{proof} 

\begin{egz}\label{example:67}
We illustrate in \Cref{fig:35-vec-fields} a choice of map $\MM$ in the case $d=35$, matching inclusions of ray class fields of orders with inclusions of SIC fields $\fieldvec{S(35*)}$ where $* \in \{ a, b, c,d,e,f,g,h,i,j \}$. We have $\Delta_d = 36\cdot 32= 2^7 \cdot 3^2$, so $\Delta_0=8$ and $f= 12$. The set of divisors $f'$ of $12$ are $\{1, 2, 3, 4, 6, 12\}$. The SIC fields for the multiplets for $d=35$ are extracted from \cite[Tab.~9]{acfw}. The illustrated choice of $\MM$ satisfies \Cref{conj:sic-to-order0}(2), which determines $\MM$ uniquely. This choice also satisfies \Cref{conj:sic-to-order0}(1); see \Cref{example:59}. 

\begin{figure}[htb]\label{fig:35b}
\begin{tikzpicture}
[auto, nd/.style={circle,minimum size = 13 mm,draw}]
\matrix[row sep=2mm,column sep=8mm] {
& & \node (35bcdg) [nd] {$\tilde{E}_{35,12}$};
\\
& \node (35af) [nd] {$\tilde{E}_{35,6}$}; &
\\
\node (35e) [nd] {$\tilde{E}_{35,3}$}; & & & \node (35h) [nd] {$\tilde{E}_{35,4}$};
\\
& & \node (35i) [nd] {$\tilde{E}_{35,2}$}; &
\\
& \node (35j) [nd] {$\tilde{E}_{35,1}$}; & &
\\
};
\draw[->] (35j) to node[swap] {2} (35i);
\draw[->] (35i) to node[swap] {2} (35h);
\draw[->] (35j) to node{4} (35e);
\draw[->] (35i) to node[swap] {4} (35af);
\draw[->] (35e) to node{2} (35af);
\draw[->] (35af) to node {2} (35bcdg);
\draw[->] (35h) to node[swap] {4} (35bcdg);
\end{tikzpicture}
\begin{tikzpicture}
[auto, nd/.style={circle,minimum size = 13 mm,draw}]
\matrix[row sep=2mm,column sep=8mm] {
& & \node (35bcdg) [nd] {$35bcdg$};
\\
& \node (35af) [nd] {$35af$}; &
\\
\node (35e) [nd] {$35e$}; & & & \node (35h) [nd] {$35h$};
\\
& & \node (35i) [nd] {$35i$}; &
\\
& \node (35j) [nd] {$35j$}; & &
\\
};
\draw[->] (35j) to node[swap] {2} (35i);
\draw[->] (35i) to node[swap] {2} (35h);
\draw[->] (35j) to node{4} (35e);
\draw[->] (35i) to node[swap] {4} (35af);
\draw[->] (35e) to node{2} (35af);
\draw[->] (35af) to node {2} (35bcdg);
\draw[->] (35h) to node[swap] {4} (35bcdg);
\end{tikzpicture}
\caption{\label{fig:35-vec-fields} Left side: For $d=35$, field inclusions of ray class fields $\tE_{35, f'}$ $\Mod{35\OO_{(f')^2\Delta_0}}$ of over-orders of $\OO_{1152}$, with edge labels giving relative field degrees. $\tilde{E}_{35,1}$ has degree $72$ over $\Q(\sqrt{\Delta_0})$, where $\Delta_0=8$. Right side: For $d=35$, SIC multiplets with letter labels ($35a$ to $35j$) specifying $\PEC(35)$-orbits, and edge labels giving relative degrees of associated fields $\fieldvec{S_{35*}}$ for $*\in \{a,b,c,d,e,f,g,h,i,j\}$. $\fieldvec{S_{35j}}$ has degree $72$ over $\Q(\sqrt{\Delta_0})$.}
\end{figure}

Inclusions of SIC fields on the right side of \Cref{fig:35-vec-fields} for $d=35$ were given in \cite[Fig.~1]{acfw}. Our figure for $d=35$ here incorporates a later correction \cite[unpublished]{applebyflammiapersonal}: the ratio SIC field for $35e$ is contained in the ratio SIC field for $35af$.
\end{egz}

\bgroup
\def\arraystretch{1.1}
\begin{table}[ht!]
\begin{tabular}{|||c|cccc|||cc|||}
\hline
$d$ & $\OO'$ & $f'$ & $\hcn{\OO'}$ & $[\cfp{d}{f'} : K]$ & multiplet & $[\fieldproj{S} : K]$ \\
\hline
$4$ & $\Z[\frac{1+\sqrt{5}}{2}]$ & $1$ & $1$ & $8$ & $4a$ & $8$ \\
$5$ & $\Z[\sqrt{3}]$ & $1$ & $1$ & $16$ & $5a$ & $16$ \\
$6$ & $\Z[\frac{1+\sqrt{21}}{2}]$ & $1$ & $1$ & $24$ & $6a$ & $24$ \\
$7$ & $\Z[\sqrt{2}]$ & $1$ & $1$ & $12$ & $7b$ & $12$ \\
 & $\Z[2\sqrt{2}]$ & $2$ & $1$ & $24$ & $7a$ & $24$ \\
$8$ & $\Z[\frac{1+\sqrt{5}}{2}]$ & $1$ & $1$ & $16$ & $8b$ & $16$ \\
 & $\Z[\frac{1+3\sqrt{5}}{2}]$ & $3$ & $1$ & $64$ & $8a$ & $64$ \\
$9$ & $\Z[\sqrt{15}]$ & $1$ & $2$ & $72$ & $9ab$ & $72$ \\ 
$10$ & $\Z[\frac{1+\sqrt{77}}{2}]$ & $1$ & $1$ & $96$ & $10a$ & $96$ \\
$11$ & $\Z[\sqrt{6}]$ & $1$ & $1$ & $80$ & $11c$ & $80$ \\
 & $\Z[2\sqrt{6}]$ & $2$ & $2$ & $160$ & $11ab$ & $160$ \\ 
$12$ & $\Z[\frac{1+\sqrt{13}}{2}]$ & $1$ & $1$ & $32$ & $12b$ & $32$ \\
 & $\Z[\frac{1+3\sqrt{13}}{2}]$ & $3$ & $1$ & $96$ & $12a$ & $96$ \\
$13$ & $\Z[\sqrt{35}]$ & $1$ & $2$ & $192$ & $13ab$ & $192$ \\
$14$ & $\Z[\frac{1+\sqrt{165}}{2}]$ & $1$ & $2$ & $288$ & $14ab$ & $288$ \\
$15$ & $\Z[\sqrt{3}]$ & $1$ & $1$ & $48$ & $15d$ & $48$ \\ 
 & $\Z[2\sqrt{3}]$ & $2$ & $1$ & $96$ & $15b$ & $96$ \\ 
 & $\Z[4\sqrt{3}]$ & $4$ & $2$ & $192$ & $15ac$ & $192$ \\
\hline
$35$ & $\Z[\sqrt{2}]$ & $1$ & $1$ & $72$ & $35j$ & $72$ \\ 
 & $\Z[2\sqrt{2}]$ & $2$ & $1$ & $144 $ & $35i$ & $144$ \\ 
 & $\Z[3\sqrt{2}]$ & $3$ & $1$ & $288$ & $35c$ & $288$ \\ 
 & $\Z[4\sqrt{2}]$ & $4$ & $1$ & $288$ & $35h$ & $288$ \\ 
 & $\Z[6\sqrt{2}]$ & $6$ & $2$ & $576$ & $35af$ & $576$ \\ 
 & $\Z[12\sqrt{2}]$ & $12$ & $4$ & $1152$ & $35bcdg$ & $1152$ \\
\hline
\end{tabular}
\medskip
\caption{Table of $d$, orders $\OO'=\OO_{(f')^2\Delta_0}$ with $\Delta_d = (d+1)(d-3) = f^2\Delta_0$, class numbers $\hcn{\OO'}$, degrees of fields $\tE_{d,f'}$, and index of fields $[\fieldproj{S} : K]$, $K= \Q(\sqrt{\Delta_d})$.}
\label{table:classfield-1a}
\end{table}
\egroup

\subsection{Verification of cases of the Order-to-Multiplet Conjecture}\label{sec:63}

We consider whether the bijection $\MM$ in \Cref{conj:sic-to-order0} can be determined from numerical data. We have performed a test of the conjecture on numerical invariants:~We did not check inclusions $\fieldproj{S_1} \subseteq \fieldproj{S_2}$ directly, but we did check divisibility of the degrees of the fields $\left.[\fieldproj{S_1}: \Q] \,\middle|\, [\fieldproj{S_2}: \Q]\right.$.

However, there would be unresolvable ambiguity in any case where there are two distinct multiplets $[[S_1]]$ and $[[S_2]]$ having
\begin{enumerate}
\item identical sizes $\abs{[[S_1]]} = \abs{[[S_2]]}$, and
\item identical projection fields $\fieldproj{S_1} = \fieldproj{S_2}$.
\end{enumerate}

Such an ambiguity can occur. It was discovered experimentally by Markus Grassl \cite[unpublished]{grasslpersonal} that two distinct multiplets $[[S_1]]\neq[[S_2]]$ can have the same associated projection SIC field $\fieldproj{S_1}=\fieldproj{S_2}$. The identified ambiguous multiplets occurred in dimensions $d \in \{47, 67, 259\}$, where in each case, there is a secondary multiplet having the same projection SIC field as the minimal multiplet (conjecturally corresponding to $\OO'=\OO_K$ or $f'=1$). Algebraic SICs have currently not been computed in every dimension up to $d=259$, but rather, Grassl has computed them is certain dimensions (including $d=259$) where the projection field has small enough expected degree to permit the calculation. 

In \Cref{sec:degeneration}, we report theoretical results (proved elsewhere) classifying when inclusions of the ray class fields $\tilde{E}_{d,f'}$ for fixed $d$ can be equalities; see \Cref{table:degeneration}. The truth of \Cref{conj:14} implies that the map $\MM$ is determined uniquely whenever $\sqfreepart(\Delta) \not\equiv 1 \Mod{8}$, where $\sqfreepart(r)$ for a positive integer $r$ denotes the squarefree part of $r$; see \Cref{thm:density-1mod8}. We show the set of such $d$ has natural density $\frac{47}{48}$; see \Cref{prop:density-18a}. We determine in \Cref{subsec:85a} that inclusions of ``SIC field candidate'' ray class fields of orders (those that \Cref{conj:14} predicts to equal $\fieldproj{S}$) do degenerate to the identity map in each of the dimensions $d \in \{47, 67, 259\}$ observed by Grassl. We also show such degeneration in dimensions $d \in \{83, 275, 211, 303, 339, 431, 447, 467\}$, in the range $4 \le d \le 500$. (It remains to numerically test these predictions when exact SICs are found in these dimensions.)

\section{Theorems on sizes of class monoids of real quadratic fields}\label{sec:class-monoids}

In this section, we state rigorous results about the behavior of $\abs{\Clm(\OO)}$ for orders with discriminants $\Delta_d = (d+1)(d-3)$. These results have conditional consequences for SICs via \Cref{conj:count}.

The connection between SICs and class numbers gives us a precise prediction for the number of Weyl--Heisenberg SICs up to geometric equivalence, which is $\abs{\Clm(\OO_{\Delta_d})}$. In this section, we use known results on the growth of class numbers to derive conditional results on the number of Weyl--Heisenberg SICs as $d \to \infty$.

The first growth result characterizes those dimensions in which our conjectures predict a unique Weyl--Heisenberg SIC, using a theorem of Byeon, Kim, and Lee \cite{BKL}; see also \cite[Thm.~1.5]{KL-UGO1}. 

\begin{prop}\label{prop:twentytwo}
One has $\abs{\Clm(\OO_{\!\Delta_d})}>1$ for all $d > 22$. The values of $d \ne 3$ with $\abs{\Clm(\OO_{\!\Delta_d})}=1$ are $d \in \{1,2,4,5,6,10,22\}$.
\end{prop}
\begin{proof}
Set $\Delta = \Delta_d = (d+1)(d-3)$ and $s(d) := \abs{\Clm(\OO_{\!\Delta})}$; we are trying to determine when $s(d)=1$. By \Cref{thm:equiv}, 
\begin{equation}
s(d) = \sum_{f'|f} \hcn{(f')^2\Delta_0}.
\end{equation}
where $\Delta = f^2\Delta_0$ and $\Delta_0$ is fundamental. If $f>1$, then $s(d)>1$. 

If $f=1$, then $\Delta=\Delta_0$ is fundamental. Write $\Delta = (d-1)^2-4 = n^2-4$, where $n:=d-1$. Byeon, Kim, and Lee solved the class number 1 problem for real quadratic fields with fundamental discriminant of the form $n^2-4$, a result previously known as Mollin's conjecture \cite[Thm~1.2]{BKL}; whenever $n>21$ (so $d>22$), $s(d)=h_\Delta>1$. The list of dimensions where $s(d)=1$ follows from a finite calculation of class numbers.
\end{proof}

\begin{cor}\label{cor:twentytwo}
Assuming \Cref{conj:count}, there is more than one $\PEC(d)$-class of Weyl--Heisenberg SIC in dimension $d$ for all $d > 22$. The dimensions $d$ with a unique $\PEC(d)$-class of Weyl--Heisenberg SIC are $d \in \{1,2,4,5,6,10,22\}$.
\end{cor}
\begin{proof}
The corollary follows from \Cref{prop:twentytwo} combined with \Cref{conj:count}.
\end{proof}

The second growth result is an asymptotic formula for the size of the ray class monoid of $\OO_{\!\Delta_d}$.
\begin{thm}\label{thm:bs}
As $d \to \infty$, and $\Delta_d = (d+1)(d-3)$, the size of the class monoid of $\OO_{\!\Delta_d}$ obeys the asymptotic formula
\begin{equation}
\log \abs{\Clm(\OO_{\!\Delta_d})} = \log d + o(\log d).
\end{equation}
\end{thm}
\begin{proof}
This is proved in \cite{KL-UGO2}, in which it follows easily from \cite[Thm.~1.2]{KL-UGO1}. The proof in \cite{KL-UGO1} uses an extension of the original Siegel form of the Brauer--Siegel theorem to quadratic orders, which is due to L.-K.~Hua \cite[Thm.~12.15.4]{Hua:1982}. The Brauer--Siegel theorem, in Siegel's original form \cite{siegel}, is an (ineffective) estimate of the growth of the product of the class number and the regulator for quadratic maximal orders. The presented result for orders must account for the class numbers of non-maximal orders. The conductor $f = f_{\Delta_d}$ is an extra degree of freedom not present in the Siegel theorem.
\end{proof}

\begin{cor}\label{cor:sicno}
Assume \Cref{conj:count}. Then, as $d \to \infty$,
\begin{equation}
\log \abs{{\rm WHSIC}(d)/\PEC(d)} = \log d + o(\log d).
\end{equation}
\end{cor}
\begin{proof}
The corollary follows from \Cref{thm:bs} combined with the conjecture.
\end{proof}

\section{Theorems on strictness or nonstrictness of inclusions of ray class fields of orders}\label{sec:degeneration}

In \cite{KL-UGO2} we prove unconditional results determining exactly when the ray class fields $\tE_{d,f'}$ introduced in \Cref{defn:Srcfo} can coincide for fixed $d$ and varying $f'$ (respectively for $E_{d,f'}$ for fixed $d$ and varying $f'$). We call such a coincidence a \textit{degeneration} of ray class field inclusions.

The first theorem gives numerical criteria on $d$ that determine all dimensions $d \ge 4$ in which a degeneration occurs. To state it, recall that a positive integer $r$ can be uniquely factored as $r=st^2$, where $t^2$ is the largest square factor of $r$. We call $s$ the {\em squarefree part} of $r$ and set $\sqfreepart(r)=s$.

\begin{thm}\label{thm:density-1mod8}
The set of $d \ge 4$ such that, for $\Delta = \Delta_d = (d+1)(d-3)$ and $\varepsilon_d = \frac{d-1 + \sqrt{\Delta}}{2}$, there exist distinct orders $\OO \neq \OO'$ with 
\begin{align}
\Z[\e_d] \subseteq \OO \subseteq \OO_{\Q(\sqrt{\Delta})}, && \Z[\e_d] \subseteq \OO' \subseteq \OO_{\Q(\sqrt{\Delta})},
\end{align}
having $H_{d\OO, \rS}^{\OO} = H_{d\OO', \rS}^{\OO'}$ for some $\rS$ (equivalently, all $\rS$) is exactly the set of $d$ that satisfy any of the following three equivalent numerical conditions.
\begin{enumerate}
\item[(1)]
The fundamental discriminant $(\Delta_d)_0$ associated to $\Delta_d$ satisfies: $(\Delta_d)_0 \equiv 1 \Mod{8}$.
\item[(2)]
$\Delta_d$ satisfies: $\sqfreepart{(\Delta_d)} \equiv 1 \Mod{8}$.
\item[(3)]
$\Delta_d = 2^{2k} D$ for some integer $k \ge 0$, with $D \equiv 1 \Mod{8}$.
\end{enumerate}
\end{thm}

The second theorem classifies all pairs $(d,f)$ where degeneration occurs.

\begin{thm}\label{thm:99}
Let $d \geq 4$. Write $\Delta = \Delta_d = (d+1)(d-3) = f^2\Delta_0$ for $\Delta_0:= (\Delta_d)_0$ a fundamental discriminant, and suppose $f',f''$ are positive integers satisfying $f' \div f$, $f'' \div f$ and $f'' < f'$. The following are equivalent:
\begin{itemize}
\item[(1)] $\tE_{d,f'} = \tE_{d,f''}$
\item[(2)] $E_{d,f'} = E_{d,f''}$
\item[(3)] $(\Delta_d)_0 \con 1 \Mod{8}$, with $f'=2f''$, and $f''$ is odd.
\item[(4)] $\sqfreepart(\Delta_d) \con 1 \Mod{8}$, with $f'=2f''$, and $f''$ is odd.
\end{itemize}
\end{thm}

A further result shows that the numerical condition $\sqfreepart(\Delta_d) \equiv 1 \Mod{8}$ of \Cref{thm:density-1mod8} is satisfied for a set of $d$ having natural density $\frac{1}{48}$. 

\begin{prop}\label{prop:density-18a}
We have 
\begin{equation}
\lim_{X \to \infty} \frac{\abs{\{d \in \Z \cap [1,X] : \sqfreepart((d+1)(d-3)) \con 1 \Mod{8}\}}}{X} = \frac{1}{48}.
\end{equation}
\end{prop}

Proofs of \Cref{thm:density-1mod8}, \Cref{thm:99}, and \Cref{prop:density-18a} may be found in \cite{KL-UGO2}. The proof of \Cref{thm:density-1mod8} shows that, if dimension $d$ has a degeneration, there is one having $f''=1$ and $f'=2$, and necessarily $d \equiv 3 \Mod{4}$ and $\Delta_d \equiv 0 \Mod{16}$. 

\subsection{Numerical data on degeneration for ray class fields of orders}\label{subsec:85a}

We present data on the implications of \Cref{thm:99} for $4 \le d \le 500$. In \Cref{table:degeneration}, we use the conditions of \Cref{thm:99} to write down all cases of equality of ray class fields of relevant orders $\tE_{d,f''} = \tE_{d,f'}$ with $(f'', f') = (f'', 2f'')$. For $d \leq 500$, equality of ray class fields $\tE_{d,f''} = \tE_{d,f'}$ occurs for 
\begin{equation} 
d \in \{47, 67, 83, 175, 211, 259, 303, 339, 431, 447, 467\}.
\end{equation}
For every such $d$, the conductor pair $(1,2)$ degenerates (that is, gives the same field ray class field in the manner of \Cref{thm:99}), and an additional conductor pair degenerates in three of these dimensions. This data is consistent with, and accounts for, Grassl's data for collapse of SIC field inclusions in dimensions $\{47, 67, 259\}$.

\bgroup
\def\arraystretch{1.1}
\begin{table}[ht!]
\centering
\begin{tabular}{|||c||c|c|c|c||}
\hline
$d$ & $\sqrt{\Delta_0}$ & $f'',f'$ & $h$ & $[\tE_{d,f''} : \Q(\sqrt{\Delta_0})]$ \\
\hhline{|||=||=|=|=|=||}
$47$ & $33$ & $1,2$ & $1$ & $1472$ \\
\hline
$67$ & $17$ & $1,2$ & $1$ & $1452$ \\
\hline
$83$ & $105$ & $1,2$ & $2$ & $9184$ \\
\hline
$175$ & $473$ & $1,2$ & $3$ & $43200$ \\
\hline
$211$ & $689$ & $1,2$ & $4$ & $117600$ \\
\hline
$259$ & $65$ & $1,2$ & $2$ & $31104$ \\
\hline
$303$ & $57$ & $1,2$ & $1$ & $40800$ \\
\hline
\end{tabular}\!\!
\begin{tabular}{||c||c|c|c|c|||}
\hline
$d$ & $\sqrt{\Delta_0}$ & $f'',f'$ & $h$ & $[\tE_{d,f''} : \Q(\sqrt{\Delta_0})]$ \\
\hhline{||=||=|=|=|=|||}
$303$ & $57$ & $5,10$ & $6$ & $244800$ \\
\hline
$339$ & $1785$ & $1,2$ & $8$ & $408576$ \\
\hline
$431$ & $321$ & $1,2$ & $3$ & $371520$ \\
\hline
$431$ & $321$ & $3,6$ & $9$ & $1114560$ \\
\hline
$447$ & $777$ & $1,2$ & $4$ & $355200$ \\
\hline
$467$ & $377$ & $1,2$ & $2$ & $290784$ \\
\hline
$467$ & $377$ & $3,6$ & $8$ & $1163136$ \\
\hline
\end{tabular}
\medskip
\caption{
Dimensions ~$d \leq 500$ with an equality of ray class fields of orders $\tE_{d,f''} = \tE_{d,f'}$. The columns show the fund.~disc.~$\Delta_0$, each pair of conductors $(f'',f')=(f'',2f'')$ exhibiting field equality, the order class number $h = \hcn{(f'')^2\Delta_0} = \hcn{(f')^2\Delta_0}$, and ray class field extension degree $[\tE_{d,f''} : \Q(\sqrt{\Delta_0})]=[\tE_{d,f'} : \Q(\sqrt{\Delta_0})]$.}
\label{table:degeneration}
\end{table}
\egroup

\section{Concluding remarks}\label{sec:9}

We state a conjecture refining \Cref{conj:sic-to-order0} and \Cref{conj:14}. We also discuss other number fields attached to SICs that have been studied in the literature.

\subsection{The Ideal-Class-to-SICs Conjecture}\label{subsec:91}

We formulate a final conjecture that refines \Cref{conj:sic-to-order0} and \Cref{conj:14} by insisting on compatibility with the Artin map. We formulate \Cref{conj:class} because its truth would provide a unifying principle explaining the structure of all the other conjectures. In this paper, we present no mechanism to construct a candidate map $\Omega$ made in its statement. 

To state this conjecture, we decompose 
\begin{equation}
\ClM(\OO_{\!\Delta}) = \bigsqcup_{f'|f} \Cl(\OO_{(f')^2\Delta_0})
\end{equation}
by means of \Cref{thm:equiv}. 

\begin{conj}[Ideal-Class-to-SIC Conjecture]\label{conj:class}
Fix a positive integer $d > 3$, let $\Delta = \Delta_d = (d+1)(d-3)$, and write $\Delta = f^2 \Delta_0$ where $\Delta_0$ is a fundamental discriminant of a quadratic field. Then there is a bijection $\Omega$, 
\begin{align}
\Omega: \bigsqcup_{f'|f} \Cl(\OO_{(f')^2\Delta_0}) \xrightarrow{\Omega} \WHSIC{d}/\PEC(d),
\end{align}
under which an ideal class $\AA \in \Cl(\OO')$ with $\OO'=\OO_{(f')^2\Delta_0}$ maps to a geometric equivalence class $\Omega(\AA)=[S_\AA]$ with associated fields $\fieldtrip{S_\AA}=E_{d,f'}$ and $\fieldproj{S_\AA}=\tE_{d,f'}$.

Furthermore, $\Omega$ can be chosen to be compatible with the Artin map (defined in \Cref{thm:main3-KL})
\begin{equation}
\Art_{\OO'} : \Cl_{d'\OO',\{\infty_1,\infty_2\}}(\OO') \to \Gal(\tE_{d,f'}/\Q(\Delta_0)),
\end{equation}
in the sense that 
\begin{equation}
[S_{\bar{\AA}\bar{\BB}}] = [(\Art_{\OO'}(\AA))(S_{\bar{\BB}})] \mbox{ for any } \AA, \BB \in \Cl_{d'\OO',\{\infty_1,\infty_2\}}(\OO')
\end{equation}
and their images $\bar{\AA}, \bar{\BB} \in \Cl(\OO')$ under the natural quotient map.
\end{conj}

The required compatibility of $\Omega$ with the Artin map implies that there is exactly one Galois multiplet for each order $\OO$ between $\OO_{\!\Delta_0}$ and $\OO_{\!\Delta}$.

\Cref{conj:class} would define a map $\MM$ of the type given in \Cref{conj:sic-to-order0} by means of the following commutative diagram,
\begin{equation}
\begin{tikzcd}
\ds\bigsqcup_{f'|f} \Cl(\OO_{(f')^2\Delta_0}) \ar[r,"\Omega"] \ar[d,two heads] & \WHSIC{d}/\PEC(d) \ar[d,two heads] \\
\{\mbox{positive divisors of $f$}\} \ar[r,dashed,"\MM"] & \{[[S]] : S \mbox{ a line-SIC in } \C^d\}
\end{tikzcd}
\end{equation}
where the downward maps send $\AA \in \Cl(\OO_{(f')^2\Delta_0})$ to $f'$ and $[S]$ to $[[S]]$, respectively.

\Cref{conj:class} is compatible with conjectures of the first author \cite{koppsic} and other authors \cite{abghm,bgm} connecting SICs to the Stark conjectures for real quadratic fields \cite{STARK77,Stark3}, which posit the existence of special algebraic units constructed from partial zeta functions and carrying a Galois action compatible with the Artin map. Work by Appleby, Flammia, and Kopp \cite{afk} (subsequent to the arXiv publication of the present paper) makes the classification scheme suggested by \Cref{conj:class} explicit. The explicit classification also yields exact formulas for the stabilizers $\stab(S)$ and thus the size of the set $\WHSIC{d}$ discussed in \Cref{rmk:whsic}.

\subsection{Refined SIC fields}\label{subsec:92}

This paper considered fields that are geometric invariants of a SIC. It can be useful in practice to define SICs using data in smaller subfields that are not geometric invariants. One such subfield is the field generated by the entries of a single fiducial projection $\Pi$. That field can be smaller than the full projection field by nearly a factor of $d$ is some special cases, and this is extremely useful for computing exact SICs in some large dimensions that would otherwise be inaccessible to numerical investigation \cite{abghm,bgm}. Another is the field generated by the \textit{SIC overlaps} $\Tr(\Pi D_\p)$ (or else the closely related \textit{SIC overlap phases} $\sqrt{d+1}\Tr(\Pi D_\p)$), where $D_p$ are the {\em displacement operators}, defined as $D_\p = \xi_d^{p_1p_2}X^{p_1}Z^{p_2}$.

The overlap field (and the overlap phase field) both appear (from numerical evidence) to be ray class fields specified by ramification at just one of the infinite places $\infty_1$, provided one chooses a ``strongly centered'' fiducial projection; see \cite{appleby2,DixonS20}. That phenomenon relates to the connection to Stark units, which live in a ray class field with ramification at the other infinite place $\infty_2$ and are (at least in some cases) Galois conjugate to squares of overlap phases.

\section{Acknowledgements}

Work of the first author was supported by NSF DMS grant \#2302514. Work of the second author was supported by NSF DMS grant \#1701576.

We are grateful to the developers of Magma and in particular David Kohel, whose software package \texttt{reduced\_forms} (and whose package \texttt{class\_number}, on which \texttt{reduced\_forms} relies) we used to compute class numbers of real quadratic orders.

We thank Markus Grassl for finding numerical examples of degeneration of SIC fields and sharing those examples with us. We thank Marcus Appleby, Steven Flammia, and Gary McConnell for valuable conversations. We thank Dustin Mixon for suggesting we use the Tarski--Seidenberg theorem in \Cref{sec:C1}. We also acknowledge the contributions of Steve Donnelly, who independently connected SICs to orders in real quadratic number fields and shared his insights with Marcus Appleby, Steven Flammia, and Markus Grassl through informal discussions.

We are grateful to the anonymous referee for many helpful comments and corrections. These comments specifically include \Cref{rmk:47}.

\appendix

\section{Conjectures for triple product fields}\label{sec:A1}

We formulate conjectures about the triple product fields $\fieldtrip{S}$. Recall that $\fieldtrip{S} \subseteq \fieldproj{S}$. We have limited empirical evidence for these conjectures. 

Recall that $E_{d,f'} = H_{d\OO',\{\infty_1,\infty_2\}}^{\OO'}$. Note that $E_{d, f'} = \tE_{d, f'}$ if $d$ is odd, and $E_{d, f'} \subseteq \tE_{d, f'}$ if $d$ is even. 

\begin{conj}[Extended Ray Class Fields of Orders Conjecture]\label{conj:TPfields}
Fix an integer $d \geq 4$ and write $\Delta = \Delta_d = (d+1)(d-3) = f^2\Delta_0$ where $\Delta_0$ is a fundamental discriminant. Then the bijection $\MM$ of \Cref{conj:14} can be chosen to satisfy the properties: If $S$ is a Weyl--Heisenberg SIC with $[S] \in \MM(f')$, then
\begin{align}
\fieldtrip{S} &= E_{d,f'} & \mbox{and} && \fieldvec{S}=\fieldproj{S} &= \tE_{d,f'}. 
\end{align}
\end{conj}

\Cref{conj:TPfields} implies part (1) of the following conjecture, which can be tested against data.

\begin{conj}\label{conj:trip}
Let $d \ge 4$. Then for a Weyl--Heisenberg SIC $S$, 
\begin{itemize}
\item[(1)]
The index 
\begin{equation}
[\fieldproj{S}: \fieldtrip{S}] = \begin{cases}
1 & \mbox{ if } d \mbox{ is odd},\\
2 & \mbox{ if } d \mbox{ is even}.
\end{cases}
\end{equation}
\item[(2)]
$\fieldproj{S} = \fieldtrip{S}(\xi_d)$.
\end{itemize}
\end{conj}

We have numerically verified part (1) of this conjecture, the index computation, for $4 \le d \le 10$. Recall that \Cref{prop:fieldproj}(2) showed unconditionally that $[\fieldproj{S}: \fieldtrip{S}]$ divides $d^2$.

\section{Calculations of ray class groups of orders}\label{sec:C}

This appendix presents a formula, given as \Cref{thm:rclsize} for the order of certain ray class groups of orders, valid for ray class groups of real quadratic orders specifically of the form $\Delta = \Delta_d = (d+1)(d-3)$. Let $K = \Q(\sqrt{\Delta})$. Write $\Delta = f^2 D$, $f \ge 1$, where $D$ is squarefree, so $K= \Q (\sqrt{D})$ and
\begin{equation}
\OO_K = \left\{\frac{1}{2}(a+b\sqrt{D}) : (a,b) \in \Z^2, \, a\equiv b \Mod{2}, \, a \equiv 0 \Mod{2} \mbox{ if } D\not\equiv 1 \Mod{4}\right\}.
\end{equation}
For the datum $(\mm, \Sigma)$ consisting of an ideal $\mm$ of the order $\OO$ in $K$ and a set of real places $\rS$ of $K$, recall the definition of the unit subgroup
\begin{equation}\label{eq:unit-group}
\UU_{\mm, \rS}(\OO) := \left\{\alpha \in \OO^\times : \alpha \equiv 1 \Mod{\mm}, \, \rho(\alpha)>0 \mbox{ for } \rho \in \rS\right\}.
\end{equation}

To state this class number formula, we make a preliminary definition. For the real quadratic field $K$ we define the {\em number field totient function} $\varphi_K(n) = \abs{\left(\OO_K/n\OO_K\right)^\times}$. It is computed by the following formula.

\begin{lem}
Let $K$ be a quadratic field of discriminant $\Delta_0$. Then
\begin{equation}\label{eq:phiK}
\varphi_K(n) = n^2 \prod_{p|n} \left(1-\frac{1}{p}\right)\left(1-\left(\frac{\Delta_0}{p}\right)\frac{1}{p}\right),
\end{equation}
where the product is over primes dividing $n$, and $\left(\frac{\Delta_0}{p}\right)$ is the Kronecker symbol.
\end{lem}
\begin{proof}
The function $\varphi_K(n)$ is multiplicative by the Chinese remainder theorem. It is computed on prime powers by splitting into cases based on the splitting behavior of $p$ in $K$, which is measured by the Kronecker symbol. A straightforward calculation yields \eqref{eq:phiK}. 
\end{proof}

The main result of this Appendix is the following formula for the cardinality of various ray class groups $\abs{\Cl_{d'\OO',\rS}(\OO')}$. Recall that for $K$ a real quadratic field, the {\em regulator} $\Reg_K$ of $K$ takes the value $\log \e$, where $\e=\e_K>1$ is the fundamental unit of $\OO_K$.

\begin{thm}\label{thm:rclsize}
Let $d \geq 4$ and let $\Delta = \Delta_d = (d+1)(d-3) = f^2\Delta_0$ for some fundamental discriminant $\Delta_0$. Let $K = \Q(\sqrt{\Delta})$. Let $\e_d = \frac{(d-1)+\sqrt{\Delta}}{2}$, and let $\OO'$ be a $K$-order such that $\Z[\e_d] \subseteq \OO' \subseteq \OO_K$, that is, with $\disc(\OO') = (f')^2\Delta_0$ for some $f'\div f$.
\begin{itemize}
\item[(1)] Then
\begin{equation}\label{eq:rclsize}
\abs{\Cl_{d\OO',\rS}(\OO')} 
= 
\begin{cases}
\dfrac{2^{\abs{\rS}} h_K \Reg_K \varphi_K(df')}{9 \log(\e_d) \varphi(f')}, & \mbox{if } d \con 3 \Mod{9} \mbox{ and } 3\div f', \vspace{5pt}\\
\dfrac{2^{\abs{\rS}} \hcn{K} \Reg_K \varphi_K(df')}{6 \log(\e_d) \varphi(f')}, & \mbox{otherwise}.
\end{cases}
\end{equation}
\item[(2)] If $d$ is even, then $\abs{\Cl_{2d\OO',\rS}(\OO')} = 2\abs{\Cl_{d\OO',\rS}(\OO')}$.
\end{itemize}
\end{thm}

\Cref{thm:rclsize} depends on the following lemma concerning unit groups of orders, which is special to the family $\Delta= (d+1)(d-3)$ of real quadratic orders, with $d \ge 4$.

\begin{lem}\label{lem:units}
Let $d \geq 4$, $\Delta = \Delta_d = (d+1)(d-3)$, and $\e_d = \frac{d-1+\sqrt{\Delta}}{2}$. Let $\rS$ be a subset of the real embedding of $\Q(\sqrt{\Delta})$. Fix an order $\OO$ satisfying
\begin{equation}
\OO_{\!\Delta} :=\Z[\e_d] \subseteq \OO \subseteq \OO_{\Q(\sqrt{\Delta})}.
\end{equation}
Then,
\begin{equation}\label{eq:71claim}
\UU_{d\OO, \rS}(\OO) = \langle \e_d^3 \rangle.
\end{equation}
\end{lem}

For \Cref{lem:units}, the unit group of interest is $\UU_{d\OO, \rS}(\OO)$. The unit $\e_d$ is totally positive, and 
\begin{equation}
\e_d^3= \frac{1}{2} \left((d-1)(d^2-2d-2) + d(d-2) f \sqrt{D}\right) \equiv 1 \Mod{d\OO_{\!\Delta}}.
\end{equation}
Thus, we have the containment relations
\begin{equation}
\langle \e_d^3 \rangle \subseteq \UU_{d\OO_{\!\Delta}, \{\infty_1, \infty_2\} } \subseteq \UU_{d\OO, \rS}(\OO) \subseteq \UU_{d\OO_K}(\OO_K).
\end{equation}
These relations show it suffices to prove \eqref{eq:71claim} for the special case $\OO = \OO_K$ with $\rS = \emptyset$.

\begin{proof}[Proof of \Cref{lem:units}]
We omit this proof, which is shown in \cite{KL-UGO2}, taking $d=n+1$. This lemma is also the case $m=1$ of \cite[Lem.~6.5]{afk}, where is different proof is given via an embedding of the order into a matrix algebra.
\end{proof} 

Now we prove \Cref{thm:rclsize}.
\begin{proof}[Proof of \Cref{thm:rclsize}]
Let $\td = d$ if $d$ is odd and $\td \in \{d,2d\}$ if $d$ is even. By a general theorem for ray class fields for orders, given as \cite[Thm.~5.6]{kopplagarias}, the cardinality of the ray class group is
\begin{equation}\label{eq:rcls0}
\abs{\Cl_{\td\OO',\rS}(\OO')} 
= \frac{\hcn{K}}{\left[\OO_K^\times : \UU_{\td\OO',\rS}(\OO')\right]} \cdot \frac{2^{\abs{\rS}} \abs{\left(\OO_K\middle/\colonideal{\td\OO'}{\OO_K}\right)^\times}}{\abs{\UU_{\td\OO'}\!\left(\OO'\middle/\colonideal{\td\OO'}{\OO_K}\right)}}.
\end{equation}
By \Cref{lem:units}, $\UU_{d\OO',\rS}(\OO') = \langle \e_d^3 \rangle$. On the other hand, when $\td = 2d$, 
\begin{equation}
e_d^3 = -d+1 + (d^2-2d)\e_d \equiv -d+1 \not\equiv 1 \Mod{2d},
\end{equation} 
but $\e_d^6 \equiv 1 \Mod{2d}$, so $\UU_{\td\OO',\rS}(\OO') = \langle \e_d^6 \rangle$. Moreover, $\OO_K^\times = \langle -1, \e_K \rangle$ for some fundamental unit $\e_K$. Thus,
\begin{equation}\label{eq:rcls1}
\left[\OO_K^\times : \UU_{\td\OO',\rS}(\OO')\right] 
= 
\begin{cases}
\left[\langle -1, \e_K \rangle : \langle \e_d^3 \rangle\right]
= \dfrac{2\log(\e_d^3)}{\log(\e_K)}
= \dfrac{6\log(\e_d)}{\Reg_K}, & \mbox{if } \td=d, \vspace{5pt}\\
\left[\langle -1, \e_K \rangle : \langle \e_d^6 \rangle\right]
= \dfrac{2\log(\e_d^6)}{\log(\e_K)}
= \dfrac{12\log(\e_d)}{\Reg_K}, & \mbox{if } \td=2d.
\end{cases}
\end{equation}
The quotient ideal appearing in \eqref{eq:rcls0} is $\colonideal{\td\OO'}{\OO_K} = \td\colonideal{\OO'}{\OO_K} = \td f'\OO_K$. Thus,
\begin{equation}\label{eq:rcls2}
\abs{\left(\OO_K\middle/\colonideal{\td\OO'}{\OO_K}\right)^\times}
= \abs{\left(\OO_K\middle/\td f'\OO_K\right)^\times}
= \varphi_K(\td f').
\end{equation} 
Now we must compute the size of $\UU_{\td\OO'}\!\left(\OO'\middle/\colonideal{\td\OO'}{\OO_K}\right) = \UU_{\td\OO'}\!\left(\OO'\middle/\td f'\OO_K\right)$. Write $\OO_K = \Z + \omega\Z$ and $\OO' = \Z + f'\omega\Z$. Thus, $\OO'/\td f'\OO_K = \frac{\Z+f'\omega\Z}{\td f'\Z+\td f'\omega\Z}$ and $\td\OO' = \td\Z + \td f'\omega\Z$, so the multiplicative group
\begin{align}
\UU_{\td\OO'}\!\left(\OO'\middle/\colonideal{\td\OO'}{\OO_K}\right) 
&= \left\{u \in \frac{\Z+f'\omega\Z}{\td f'\Z+\td f'\omega\Z} : u \equiv 1 \Mod{\td\Z+\td f'\omega\Z}\right\} \\
&\isom \left\{u \in \Z/\td f'\Z : u \equiv 1 \Mod{\td}\right\}.
\end{align}
If $\gcd(\td,f')=1$, then this group is isomorphic to $\left(\Z/f'\Z\right)^\times$ by the Chinese remainder theorem. Its order is then $\varphi(f')$. The only case in which $\gcd(\td,f') \neq 1$ is when $d \equiv 3 \Mod{9}$ and $3\div f'$, in which case $\gcd(\td,f') = 3$ and $\td=3d_1$ with $\gcd(d_1,3f')=1$. The Chinese remainder theorem gives a ring isomorphism $\pi : \Z/\td f'\Z \to \Z/d_1\Z \times \Z/3f'\Z$ given by $\pi(u) = (\pi_1(u), \pi_2(u)) = (u \Mod{d_1}, u \Mod{3f'})$. Under this isomorphism, the condition $u \equiv 1 \Mod{\td}$ is equivalent to $\pi_1(u)=1$ and $\pi_2(u) \equiv 1 \Mod{3}$; the latter condition is satisfied on an index $2$ subgroup of $(\Z/3f'\Z)^\times$. Thus, the order of $\UU_{\td\OO'}\!\left(\OO'\middle/\colonideal{\td\OO'}{\OO_K}\right)$ is $\varphi(3f')/2 = \frac{3}{2}\varphi(f')$. In summary,
\begin{equation}\label{eq:rcls3}
\abs{\UU_{\td\OO'}\!\left(\OO'\middle/\colonideal{\td\OO'}{\OO_K}\right)}
= 
\begin{cases}
\frac{3}{2}\varphi(f'), & \mbox{if } d \con 3 \Mod{9} \mbox{ and } 3\div f', \\
\varphi(f'), & \mbox{otherwise}.
\end{cases}
\end{equation}
This proves (1). 

We now prove (2). In the case $\td = d$, plugging in \eqref{eq:rcls1}, \eqref{eq:rcls2}, and \eqref{eq:rcls3} to \eqref{eq:rcls0} yields \eqref{eq:rclsize}. In the case $\td = 2d$ (with $d$ even), we have $\varphi_K(\td f') = 4 \varphi_K(df')$, and again using plugging in \eqref{eq:rcls1}, \eqref{eq:rcls2}, and \eqref{eq:rcls3} to \eqref{eq:rcls0} and comparing to \eqref{eq:rclsize}, we see that $\abs{\Cl_{2d\OO',\rS}(\OO')} = 2\abs{\Cl_{d\OO',\rS}(\OO')}$.
\end{proof}

\section{Equivalence of finiteness and algebraicity of SICs}\label{sec:C1}

In this appendix, we prove that the conjecture that ``there are finitely many Weyl--Heisenberg SICs in dimension $d$'' is equivalent to the statement that ``every Weyl--Heisenberg SICs in dimension $d$ has an algebraic fiducial vector.'' This equivalence is a special case of a result in real algebraic geometry, \Cref{thm:realalgebraicdim0} below, which we deduce from a version of the Tarski--Seidenberg theorem within the formalism adopted by the lecture notes of Schweighofer \cite{schweighofer}. 

We recall a few definitions from algebraic geometry and from earlier in this paper before proving the key result. For a subfield $K$ of $\C$ and an ideal $I \subseteq K[x_1, \ldots, x_n]$, the \textit{affine algebraic set} $X=V(I)$ in $\C^n$ is the zero set in $\C^n$ of all polynomials in $I$. For every field $L$ with $K \subseteq L \subseteq \C$, the set of \textit{$L$-points} of $X$ is
\begin{equation}
X(L) := \{\mathbf{x} \in L^n : f(\mathbf{x})=0 \mbox{ for all } f \in I\}.
\end{equation}
Recall also that we have defined $\ol{\Q}$ to be the algebraic closure of $\Q$ within the complex numbers $\C$, so consequently $\ol{\Q} \cap \R$ is a degree two subfield of $\ol{\Q}$ satisfying $(\ol{\Q} \cap \R)(i) = \ol{\Q}$ (where $i=\sqrt{-1}$). The field $\Ralg :=\ol{\Q} \cap \R$ of real algebraic numbers is an example of a \textit{real closed field} (see \cite[Defn.~1.4.9]{schweighofer}), a field sharing many properties with the real numbers $\R$.

\begin{thm}\label{thm:realalgebraicdim0}
Let $I$ be an ideal in $\ol\Q[x_1, \ldots, x_n]$, and let $X = V(I)$ be the associated affine variety. As a set, $X(\C) \cap \R^n$ is finite if and only if $X(\C) \cap \R^n \subseteq \ol{\Q}^n$.
\end{thm}
\begin{proof}
Let $\Ralg = \ol\Q\cap\R$ and $S = \Ralg[x_1, \ldots, x_n]$. Let $\RR = \{\Ralg,\R\}$ and $\RR_N = \{(R,\mathbf{x}) : R \in \RR, \, \mathbf{x} \in R^N\}$ for any natural number $N$, following the notation of \cite[Defn.~and Rmk.~1.8.3]{schweighofer}.

Write $I = (f_1, \ldots, f_m)$ for generators $f_k \in S[i]$ (with $i=\sqrt{-1}$), which has $S[i]= \ol{\Q}[x_1, \ldots, x_n]$, and write each $f_k = g_k + h_k i$ for polynomials $g_k, h_k \in S$. Consider the $S$-ideal $\tilde{I} = (g_1, h_1, \ldots, g_m, h_m)$ and set $\tilde{X} = V(\tilde{I})$. Then $X(\C) \cap \R^n = \tilde{X}(\R)$ and $X(\C) \cap K^n = \tilde{X}(K)$.

For $1 \leq j \leq n$, let $\pi_j : \R^n \to R$ be the $j$-th coordinate map $\pi_j(\mathbf{x}) = x_j$. The field $\Ralg$ is a real closed field, as defined in \cite[Defn.~1.4.9]{schweighofer}; see \cite[Ex.~1.4.10 and Ex.~1.7.12]{schweighofer}. Consider the formal disjoint union of images
\begin{align}
&\left(\{\Ralg\} \times \pi_j(\tilde{X}(\Ralg))\right) \sqcup \left(\{\R\} \times \pi_j(\tilde{X}(\R))\right) \\
&= \left\{(R,y) \in \RR_1 : \exists \mathbf{x} \in R^n \text{ s.t.~} (g_1(\mathbf{x})=0 \,\&\, h_1(\mathbf{x})=0 \,\&\, \cdots \,\&\, g_m(\mathbf{x})=0 \,\&\, h_m(\mathbf{x})=0 \,\&\, y = \pi_j(\mathbf{x}))\right\}. 
\end{align}
We now apply real quantifier elimination as stated by Schweighofer \cite[Thm.~1.8.17]{schweighofer} (a version of the Tarski--Seidenberg theorem, within Schweighofer's formalism of classes of real closed fields). Applying that theorem $n$ times to eliminate each coordinate of $\mathbf{x}$ from the ``there exists'' statement, it follows that $\left(\{\Ralg\} \times \pi_j(\tilde{X}(\Ralg))\right) \sqcup \left(\{\R\} \times \pi_j(\tilde{X}(\R))\right)$ is a $(\Ralg,(\Ralg)_{\geq 0})$-semialgebraic class in the sense of \cite[Defn.~and Rmk.~1.8.3]{schweighofer}. That is,
\begin{equation}
\left(\{\Ralg\} \times \pi_j(\tilde{X}(\Ralg))\right) \sqcup \left(\{\R\} \times \pi_j(\tilde{X}(\R))\right)
= \left\{(R,y) \in \RR_1 : p_{j,1}(y) \geq 0 \,\&\, \cdots \,\&\, p_{j,r_j}(y) \geq 0\right\}
\end{equation}
for polynomials $p_{j,k} \in \Ralg[y]$. Every polynomial in $\Ralg[y]$ factors as a product of linear and irreducible quadratic factors (with the quadratic factors having no real roots), from which one sees that in fact this set is a ``union of intervals''
\begin{align}
&\left(\{\Ralg \} \times \pi_j(\tilde{X}(\Ralg))\right) \sqcup \left(\{\R\} \times \pi_j(\tilde{X}(\R))\right) \\
&= \left\{(R,y) \in \RR_1 : a_{j,1} \leq y \leq b_{j,1} \mbox{ or } \cdots \mbox{ or } a_{j,t_j} \leq y \leq b_{j,t_j} \right\}
\end{align}
with $a_{j,k}, b_{j,k} \in \Ralg$ satisfying $a_{j,k} \leq b_{j,k}$.

We thus have a dichotomy:
\begin{enumerate}
\item
If there exists any pair $(j,k)$ with a strict inequality $a_{j,k} < b_{j,k}$, then $\pi_j(\tilde{X}(\R))$ is an uncountably infinite set, so the set $X(\C) \cap \R^n = \tilde{X}(\R)$ is both infinite and contains points with transcendental coordinates. 
\item
If one has $a_{j,k} = b_{j,k}$ for all pairs $(j,k)$, then each $\pi_j(\tilde{X}(\R)) = \pi_j(\tilde{X}(F)) = \{a_{j,1}, \ldots a_{j,t_j}\} \subset \Ralg$, so the set $X(\C) \cap \R^n = \tilde{X}(\R)$ is a finite set of points with all coordinates algebraic. \qedhere
\end{enumerate}
\end{proof}

\begin{lem}\label{lem:whsicvariety}
For subfields $L \subseteq \C$, regard $L^{2d^2}$ as a product of two copies of the set of $d \times d$ matrices $\Mat_{d \times d}(L)$. For $\p\in (\Z/d\Z)^2$, set $D_\p = X^{p_1}Z^{p_2}$ (where $X$ and $Z$ are the generators of the Weyl--Heisenberg group $\WH(d)$). Let $W_d$ be the affine algebraic variety defined for $L \subseteq \Q(i, \zeta_d)$ by
\begin{equation}
W_d(L) = \left\{(\Phi,\Psi) \in L^{2d^2} \, : \, 
\begin{array}{c}
(\Phi + i\Psi)^2 = \Phi + i\Psi,\,
\Tr(\Phi + i\Psi) = 1,\, 
\Phi^\top - i\Psi^\top = \Phi + i\Psi, \text{ and} \\
\Tr((\Phi + i\Psi)D_\p)\Tr((\Phi + i\Psi) D_\p^{-1})=\frac{1}{d+1} \text{ for } \p\in (\Z/d\Z)^2 \setminus \{(0,0)\}
\end{array}
\right\}.
\end{equation}
Then the linear map sending $(\Phi,\Psi) \mapsto \Pi = \Phi + i\Psi$ defines a bijection between $W_d(\C) \cap \R^{2d^2}$ and the set of Weyl--Heisenberg SIC fiducial projections in dimension $d$.
\end{lem}
\begin{proof}
This is a direct consequence of the definition of a Weyl--Heisenberg SIC (\Cref{def:WH-SIC}; see also \Cref{sec:11}). Specifically, taking $\Pi = \Phi + i\Psi$, we reason as follows.
\begin{itemize}
\item[(1)] The conditions $(\Phi + i\Psi)^2=\Phi + i\Psi$ and $\Tr(\Phi + i\Psi)=1$ are equivalent to the statement that $\Pi$ is a rank one projection matrix, that is, $\Pi = \bv\bw^\top$ for column vectors $\bv,\bw$ satisfying $\bw^\top\bv=1$.
\item[(2)] Assume (1). The condition that $\Phi^\top - i\Psi^\top = \Phi + i\Psi$ is equivalent to the condition that $\Pi$ is Hermitian, that is that $\bw = \ol{\bv}$.
\item[(3)] Assume (1) and (2). One has $\Tr(\Pi D_\p) = \Tr(\bv\bv^\ct D_\p) = \Tr(\bv^\ct D_\p\bv) = \bv^\ct D_\p\bv = \langle \bv, D_\p\bv \rangle$ and similarly $\Tr(\Pi D_\p^{-1}) = \langle \bv, D_\p^{-1}\bv \rangle = \langle D_\p\bv, \bv \rangle = \ol{\langle \bv, D_\p\bv \rangle}$. The condition that $\Tr((\Phi + i\Psi)D_\p)\Tr((\Phi + i\Psi) D_\p^{-1})=\frac{1}{d+1} \text{ for } \p\in (\Z/d\Z)^2 \setminus \{(0,0)\}$ is thus equivalent to the condition that $\abs{\langle \bv, D_\p\bv \rangle}^2 = \frac{1}{d+1}$ for $\p\in (\Z/d\Z)^2 \setminus \{(0,0)\}$, from which it follows that $\abs{\langle D_\p\bv, D_\q\bv \rangle}^2 = \abs{\langle \bv, D_{\q-\p}\bv \rangle}^2 = \frac{1}{d+1}$ for $\p \neq \q$, that is, the lines $\{\C D_\p\bv : \p \in (\Z/d\Z)^2\}$ are equiangular. 
\end{itemize}
Therefore, the condition that $\Pi = \Phi+i\Psi$ with $(\Phi,\Psi) \in W_d(\R)$ is equivalent to the condition that $\Pi$ is the fiducial projection of a Weyl--Heisenberg SIC.
\end{proof}

\begin{prop}\label{prop:finalgequiv}
Let $d \in \N$. The following statements are equivalent.
\begin{itemize}
\item[(1)] The set of Weyl--Heisenberg line-SICs in dimension $d$ is finite.
\item[(2)] Every fiducial vector of a Weyl--Heisenberg line-SICs in dimension $d$ can be scaled to have all entries in $\ol{\Q}$.
\end{itemize}
\end{prop}
\begin{proof}
By \Cref{lem:whsicvariety}, the set of Weyl--Heisenberg SIC fiducial projections can be identified with the real points $W_d(\C) \cap \R^{2d^2}$ of a complex affine algebraic variety defined as $W_d = V(I_d)$ for an ideal $I_d$ of $\ol{\Q}[x_1, \ldots, x_{2d^2}]$. 

Therefore, by \Cref{thm:realalgebraicdim0}, the set of fiducial projections of Weyl--Heisenberg SICs in dimension $d$ is finite if and only if every fiducial projection of a Weyl--Heisenberg SIC in dimension $d$ has algebraic entries. The statement of the proposition then follows using \Cref{lem:alg-SIC}(3).
\end{proof}

\bibliographystyle{abbrv}
\bibliography{references}

\end{document}